\documentclass[mathpazo]{cicp}
\usepackage[margin=1in]{geometry}
\usepackage{amsmath,epsf,cite}
\usepackage{amssymb,amsthm,tocvsec2}
\usepackage{graphicx}
\usepackage{epsfig,epsf,latexsym,subfigure}
\usepackage{float}
\usepackage{color}
\usepackage{mathrsfs}
\usepackage{indentfirst}

\date{ }

\numberwithin{equation}{section}
\numberwithin{figure}{section}
\numberwithin{table}{section}

\theoremstyle{plain}

\theoremstyle{remark}
\newtheorem{rem}{Remark}[section]

\def\bx{\mathbf{x}}

\def\cL{\mathcal{L}}

\def\cG{\mathcal{G}}

\def\cE{\mathcal{E}}

\def\cV{\mathcal{V}}

\newcommand{\ben}{\begin{eqnarray}}
\newcommand{\een}{\end{eqnarray}}
\newcommand{\beq}{\begin{equation}}
\newcommand{\eeq}{\end{equation}}
\newcommand{\bea}{\begin{array}}
\newcommand{\eea}{\end{array}}
\newcommand{\bef}{\begin{figure}}
\newcommand{\eef}{\end{figure}}

\newtheorem{scheme}{Scheme}[section]

\begin{document}
\title{Improving the Accuracy and Consistency of the Scalar Auxiliary Variable (SAV) Method with Relaxation}
\author[Maosheng  Jiang, Zengyan Zhang, and J. Zhao]{
Maosheng Jiang\affil{1} , Zengyan Zhang\affil{2} and 
Jia Zhao\affil{2}\comma \corrauth}
\address{\affilnum{1}\ School of Mathematics and Statistics, Qingdao University, Qingdao 266071,  China \\
\affilnum{2}\ Department of Mathematics \& Statistics, Utah State University, Logan, UT, 84322, USA }
\email{ {\tt jia.zhao@usu.edu.} (J.~Zhao)}

\begin{abstract}
The scalar auxiliary variable (SAV) method was introduced by Shen et al. in \cite{ShenXuYang2018} and has been broadly used to solve thermodynamically consistent PDE problems. By utilizing scalar auxiliary variables, the original PDE problems are reformulated into equivalent PDE problems. The advantages of the SAV approach, such as linearity, unconditionally energy stability, and easy-to-implement, are prevalent. However, there is still an open issue unresolved, i.e., the numerical schemes resulted from the SAV method preserve a "modified" energy law according to the auxiliary variables instead of the original variables. Truncation errors are introduced during numerical calculations so that the numerical solutions of the auxiliary variables are no longer equivalent to their original continuous definitions. In other words, even though the SAV scheme satisfies a modified energy law, it does not necessarily satisfy the energy law of the original PDE models.
This paper presents one essential relaxation technique to overcome this issue, which we named the relaxed-SAV (RSAV) method. Our RSAV method penalizes the numerical errors of the auxiliary variables by a relaxation technique. In general, the RSAV method keeps all the advantages of the baseline SAV method and improves its accuracy and consistency noticeably. Several examples have been presented to demonstrate the effectiveness of the RSAV approach.

\end{abstract}

\maketitle

\section{Introduction}
Many physical problems, such as interface dynamics\cite{Anderson1998, Guo2021,Yue2014}, crystallization\cite{Elder2004,Liu2007}, thin films\cite{Otto2001,Wan10}, polymers\cite{Fraaije1993,Doi1988}, and liquid crystallization\cite{Forest2004_1,Forest2004_2} could be modeled by gradient flow systems which also agree with the second law of thermodynamics. If the total free energy is known, the gradient flow model could be obtained according to the mobility and the variation of free energy. Because of the nonlinear terms in the governing equation, neither the exact solution nor the numerical solution is easy to obtain. In general, consider the spatial-temporal domain $\Omega_t :=\Omega \times (0, T]$. The dissipative dynamics of the state variable $\phi$ are driven by
\beq 
\partial_t \phi(\bx, t) = -\cG \frac{\delta \cE}{\delta \phi}, \quad (\bx, t) \in \Omega_t,
\eeq 
where $\cG$ is a semi-positive definite operator known as the mobility operator, and $\cE$ is a functional of $\phi$ known as the free energy. The triplet $(\phi, \cG, \cE)$ uniquely defines the dissipative system (gradient flow dynamics). For instance, given $\cG=1$ and $\cE =\displaystyle \int_\Omega \frac{\varepsilon^2}{2}|\nabla \phi|^2 + \frac{1}{4} (\phi^2 -1)^2d\bx$, with $\displaystyle\varepsilon$ as a model parameter, we obtain the following Allen-Cahn equation
\beq \label{eq:Intro-AC}
\partial_t \phi = \varepsilon^2 \Delta \phi - \phi^3 + \phi.
\eeq 
If we consider $\cG=-\Delta$ and $\cE =\displaystyle\int_\Omega \frac{\varepsilon^2}{2}|\nabla \phi|^2 + \frac{1}{4} (\phi^2 -1)^2d\bx$, we obtain the Cahn-Hilliard equation
\beq
\partial_t \phi = \Delta (-\varepsilon^2 \Delta \phi + \phi^3 - \phi).
\eeq 
Given $\cG=1$ and $\cE = \displaystyle\int_\Omega D|\nabla \phi|^2 d\bx$, we have the heat equation
\beq
\partial_t \phi = D \Delta \phi,
\eeq 
where $D$ is the diffusion coefficient.

All these models discussed above have an energy dissipation property. Mainly, 
\beq
\frac{d}{dt} \cE(\phi) = -(\frac{\delta \cE}{\delta \phi}, \frac{\delta \phi}{\delta t}) = - \Big( \frac{\delta \cE}{\delta \phi}, \cG \frac{\delta \cE}{\delta \phi} \Big) \leq 0,
\eeq  
given the boundary terms are diminished to zero. Here we have used the inner producut notation $\displaystyle(f, g) =\int_\Omega fg d\bx$, $\forall f,g \in L^2(\Omega)$.
Numerical algorithms that solve such models shall also preserve the energy dissipation structure, i.e., follow the thermodynamic physical laws. Numerical schemes that preserve the energy dissipation structure are known as energy stable schemes. And if such structure-preserving doesn't depend on the time step, the numerical schemes are known as unconditionally energy stable.

Many energy stable numerical schemes are proposed to approach the solutions of gradient flow models or the dissipative systems. The classical approaches are the fully implicit schemes \cite{FengNM2004}. Though some of them are unconditionally energy stable, solving such fully implicit schemes is not trivial. Nonlinear problems have to be solved in each time step. However, the existence and uniqueness of the solution usually have strong restrictions on the time steps which prevent those fully implicit numerical schemes from being widely used. One remedy is the convex splitting method \cite{Eyre1998Unconditionally}, which splits the nonlinear terms of free energy into the subtraction of two convex functions.  It is easy to check the convex-splitting schemes are unconditional energy stable and uniquely solvable. However, the general type of second-order convex splitting schemes is not available. So far, it is only possible to design second-order convex-splitting schemes case-by-case\cite{wise2011,wise2012, wise2013,Wan10}. Meanwhile, there are many other unconditional energy stable methods, such as stabilization method\cite{Liu2007, Xu2006}, exponential time discretization method \cite{Du2016,Ju&Li&Qiao&ZhangMC2017,Cheng2019}. The stabilization method represents the nonlinear terms explicitly and adds some regularization terms to relax strict constraints for the time steps. Similarly, with the convex splitting method, it is usually limited to first-order accuracy. The exponential time discretization (ETD) method shows high-order accuracy by integrating the governing equation over a single time step and uses polynomial interpolations for the nonlinear terms. But the theoretical proofs for energy stability properties of high-order ETD schemes are still missing.

Recently, the numerical method named invariant energy quadratization(IEQ) or energy quadratization(EQ) are proposed \cite{Yang2017,Yang&Zhao&HeJCAM2018,Yang&Zhao&Wang&ShenM3AS2017, Zhao2017, ZhaoYangGongWang2017, Zhao2018, GongZhaoWang2020, GongZhaoWang2020_2}. It is a generalization of the method of Lagrange multipliers or auxiliary variables from \cite{Guillen2011, Guillen2013}. The IEQ approach permits us to construct linear, second-order, unconditional energy stable schemes, and furthermore arbitrarily high-order unconditionally energy stable schemes\cite{GongZhaoWang2020_2, GongEnergy}. With many advantages of the IEQ or EQ approach, it usually leads to a coupled system with time-dependent coefficients. As a remedy, the SAV approach \cite{ShenXuYang2018, ShenXuYang2019, Lishen2020, LiShenRui2019, LiShen2020_2,XuJCP,Qiao2019} have been proposed by introducing scalar auxiliary variables instead of auxiliary function variables. The SAV method also can be applied to a large class of gradient flow systems, which keeps the advantages of the EQ approach but usually leads to decoupled systems with constant coefficients. These properties make the SAV method easier to implement, so it is highly efficient. Besides, when the researchers applied the SAV approach to many different systems, several modified schemes were developed. Multiple scalar auxiliary variable(MSAV) approach \cite{Chenshen2018} was proposed to solve the phase-field vesicle membrane model where two auxiliary variables introduced to match two additional penalty terms enforcing the volume and surface area. If using the introduced scalar variable to control both the nonlinear and the explicit linear terms, one high efficient SAV approach was developed\cite{Huang2020}, which spend half of time compared with  the original SAV approach while keeping all its other advantages.
One stabilized-scalar auxiliary variable(S-SAV)\cite{Kim2021} approach was proposed to solve the  phase-field surfactant model, which is a decoupled scheme and allowed to be solved step by step. For phase-field surfactant model, the authors in \cite{Zhu2019} also presented certain subtle 
explicit-implicit treatments for stress and convective terms to construct the linear, decoupled, unconditional energy stable schemes based on the classical SAV approach.

However, there is still a big gap for the IEQ or SAV method, making them not as perfect as expected. Mainly, these two methods preserve a "modified" energy law according to the auxiliary variables instead of the original variables. This inconsistency introduces errors during the computation. In the end, even though the IEQ or SAV schemes preserve the "modified" energy law, they are not necessarily preserving the original energy law, i.e., the energy law for the original PDE models might be violated by the numerical solutions. This is known in the community, but so far, no good remedy is available yet. This motivates our research in this paper. With a novel relaxation step, we effectively penalize the inconstancy between numerical solutions for the auxiliary variables and their continuous definitions. Thus we name this new approach as the relaxed-SAV (RSAV) method. Through the RSAV method, we are able to design novel linear, second-order, unconditionally energy stable schemes, which
keep the advantage of the baseline SAV method and preserve the original energy laws. It turns out that the relaxation approach effectively improves the accuracy and consistency of the SAV method noticeably.

%Our motivation is from one paper \cite{Zhaosub}, one relaxation technique is proposed to improve accuracy and stability for EQ. We will use this technique to develop the relaxed-SAV scheme.

The rest of this paper is organized as follows. In Section 2, we revisit the baseline SAV approach and multiple SAV approaches for the general gradient flow system.  Section 3 proposes the relaxed-SAV and relaxed-MSAV schemes. The energy stability properties of the RSAV and RMSAV methods are proved rigorously. Then, in Section 4, several specific examples and numerical tests are provided to verify the accuracy and effectiveness of the proposed relaxed SAV numerical schemes.  In the end, we give a brief conclusion.

\section{A brief review of the SAV method}
Consider the general gradient flow model
\beq \label{eq:generic-model}
\frac{\partial \phi}{\partial t} = -\cG \frac{\delta \cE}{\delta \phi},
\eeq 
where $\phi$ is the state variable, $\cE$ is the free energy, and $\cG$ is a semi-positive definite operator for dissipative systems (and a skew-symmetric operator for reversible systems or Hamiltonian systems). In the rest of this paper, we consider periodic boundary conditions for simplicity, though all our results can be applied  to models with thermodynamically consistent boundary conditions.

For the general gradient flow model \eqref{eq:generic-model}, it has the following energy law
\beq
\frac{d}{dt} \cE(\phi) = (\frac{\delta \cE}{\delta \phi}, \frac{\partial \phi}{\partial t}) = - \Big( \frac{\delta \cE}{\delta \phi}, \cG \frac{\delta \cE}{\delta \phi} \Big).
\eeq 
When $\cG$ is semi-positive definite, we have $\frac{d\cE}{dt} \leq 0$, and when $\cG$ is a skew-symmetric operator, we have $\frac{d\cE}{dt} = 0$. Here we use the notation, $(f, g) = \displaystyle\int_\Omega fg d\bx$, $\forall f, g\in L^2(\Omega)$. The induced norm will be denoted as $\| f\| =\sqrt{(f,f)}$.

\subsection{Baseline scalar auxiliary variable (SAV) method}
Here we briefly illustrate the scalar auxiliary variable (SAV) method that was first introduced in \cite{ShenXuYang2018}. Following the notations in \cite{ShenXuYang2018}, we start with a simplified free energy
\beq \label{eq:original-energy}
\cE = \int_\Omega \Big( \frac{1}{2} \phi \cL\phi + F(\phi)  \Big) d\bx,
\eeq 
where $\cL$ is a linear operator, and $F$ is the bulk free energy density. Also, we denote the identity operator as $I$ that will be used in the rest of this paper.  Then the gradient flow model in \eqref{eq:generic-model} is specified as
\beq  \label{eq:general-gradient-flow}
\partial_t \phi = -\cG (\cL \phi + F'(\phi)),
\eeq 
with the following energy law
\beq \label{eq:general-gradient-flow-energy-law}
\frac{d \cE}{dt} = \int_\Omega \frac{\delta \cE}{\delta \phi} \frac{\partial \phi}{\partial t} d\bx = -\Big(\cL \phi + F'(\phi), \cG (\cL \phi + F'(\phi)) \Big).
\eeq 

For the SAV method, a scalar auxiliary  variable $q(t)$ is introduced as
\beq \label{eq:SAV-q}
q(t) :=Q(\phi) = \sqrt{ \int_\Omega  (F(\phi) - \frac{1}{2}\gamma_0 \phi^2) d\bx +C_0},
\eeq
where $C_0>0$ is a constant making sure $Q(\phi)$ is well-defined, i.e., $\displaystyle\int_\Omega  (F(\phi) - \frac{1}{2}\gamma_0 \phi^2)d\bx+C_0 >0$. Here $\gamma_0$ is a regularization parameter that was first introduced in \cite{ChenZhaoYangANM2018}. With the scalar auxiliary variable $q(t)$, the gradient flow model \eqref{eq:general-gradient-flow} is reformulated into an equivalent form
\begin{subequations} \label{eq:gradient-flow-SAV}
\begin{align}
& \frac{\partial \phi}{\partial t} = - \cG \Big( \cL\phi +\gamma_0 \phi  + \frac{ q(t)}{Q(\phi)} V(\phi) \Big), \\
& \frac{dq(t)}{dt} = \frac{1}{2Q(\phi)} \int_\Omega V(\phi) \partial_t\phi d\bx,  \quad V(\phi)=F'(\phi)-\gamma_0 \phi.
\end{align}
\end{subequations}

Denote the modified energy $\hat{E}$ as
\beq \label{eq:SAV-energy}
\hat{E} = \int_\Omega \Big( \frac{1}{2}\phi \cL\phi + \frac{1}{2}\gamma_0\phi^2 \Big) d\bx +q^2 - C_0.
\eeq 
The reformulated model \eqref{eq:gradient-flow-SAV} has the following energy law
\begin{subequations}
\begin{align}
\frac{d\hat{E}}{dt} 
&= \int_\Omega \frac{\delta \hat{E}}{\delta \phi}\frac{\partial \phi}{\partial t} d\bx + \frac{\delta \hat{E}}{\delta q}\frac{dq}{dt} \\
&= -\Big(   \cL\phi +\gamma_0 \phi  + \frac{ q(t)}{Q(\phi)} V(\phi) , \cG ( \cL\phi +\gamma_0 \phi  + \frac{ q(t)}{Q(\phi)} V(\phi) )\Big).
\end{align}
\end{subequations}

\begin{rem}
With the SAV transformation, numerical algorithms can be introduced to solve the equivalent model in \eqref{eq:gradient-flow-SAV} that in turn solve the original model in \eqref{eq:general-gradient-flow}, since \eqref{eq:gradient-flow-SAV} and \eqref{eq:general-gradient-flow} are equivalent. 
\end{rem}

Consider the time domain $[0, T]$, and we discretize it into equally distanced meshes $0=t_0<t_1< \cdots<t_N=T$, with $t_i = i \delta t$ and $\delta t=\frac{T}{N}$. Then we use $(\bullet)^{n+1}$ to represent the numerical approximation of $(\bullet)$ at $t_{n+1}$. With these notations, we recall the two second-order SAV schemes in \cite{ShenXuYang2018}. First of all, if we use semi-implicit backward differentiation formula (BDF) for the time discretization, we can get the SAV-BDF2 scheme as below.

\begin{scheme}[Second-order SAV-BDF2 Scheme] \label{sch:SAV-BDF}
\begin{subequations}\label{eq:SAV-BDF}
\begin{align}
& \frac{3 \phi^{n+1}-4\phi^n+\phi^{n-1}}{2\delta t}  = - \cG \mu^{n+1}, \\
& \mu^{n+1}=   \cL\phi^{n+1} + \gamma_0 \phi^{n+1} + \frac{ q^{n+1}}{Q(\overline{\phi}^{n+1})} V(\overline{\phi}^{n+1}), \\
& \frac{ 3q^{n+1} - 4q^{n}+q^{n-1}}{2\delta t} = \int_\Omega \frac{ V(\overline{\phi}^{n+1})}{2Q(\overline{\phi}^{n+1})} \frac{3 \phi^{n+1}-4\phi^n+\phi^{n-1}}{2\delta t}  d\bx,
\end{align}
\end{subequations}
where $\overline{\phi}^{n+1}=\frac{3}{2}{\phi}^{n}-\frac{1}{2}{\phi}^{n-1}$ and $V(\overline{\phi}^{n+1})=F'(\overline{\phi}^{n+1})-\gamma_0\overline{\phi}^{n+1}$.
\end{scheme}

The SAV-BDF2 scheme \ref{sch:SAV-BDF} has the following discrete energy law.
\begin{theorem}\label{Theo:SAV-BDF}
The Scheme \ref{sch:SAV-BDF} is unconditionally energy stable in the sense that \cite{ShenXuYang2019}
\begin{subequations}
\begin{align}
& \frac{1}{4}[({\phi}^{n+1},  (\cL+\gamma_0 I ){\phi}^{n+1})+(2{\phi}^{n+1}-{\phi}^{n},  (\cL+\gamma_0 I)(2{\phi}^{n+1}-{\phi}^{n}))]+({q}^{n+1})^2+(2{q}^{n+1}-{q}^{n})^2\notag\\
&\qquad -\frac{1}{4}[({\phi}^{n},  (\cL+\gamma_0 I ){\phi}^{n})+(2{\phi}^{n}-{\phi}^{n-1},  (\cL+\gamma_0 I)(2{\phi}^{n}-{\phi}^{n-1}))]-({q}^{n})^2-(2{q}^{n}-{q}^{n-1})^2\notag\\
& \qquad\qquad\qquad\qquad \qquad\qquad\qquad\qquad \qquad\qquad\qquad\qquad \le- {\delta t} (\cG \mu^{n+1},\mu^{n+1}).\notag
\end{align}
\end{subequations}
\end{theorem}
Secondly, if we use the semi-implicit Crank-Nicolson method for the time discretization, we will have the SAV-CN scheme as below.
\begin{scheme}[Second-order SAV-CN Scheme] \label{sch:SAV-CN}
\begin{subequations}\label{eq:SAV-CN}
\begin{align}
& \frac{\phi^{n+1}-\phi^n}{\delta t} = - \cG \mu^{n+\frac{1}{2}},\\
& \mu^{n+\frac{1}{2}} = \cL\phi^{n+\frac{1}{2}} + \gamma_0 \phi^{n+\frac{1}{2}} + \frac{ q^{n+\frac{1}{2}}}{Q(\overline{\phi}^{n+\frac{1}{2}})} V(\overline{\phi}^{n+\frac{1}{2}}),\\
& \frac{ q^{n+1} - q^{n}}{\delta t} = \int_\Omega \frac{ V(\overline{\phi}^{n+\frac{1}{2}})}{2Q(\overline{\phi}^{n+\frac{1}{2}})} \frac{ \phi^{n+1}-\phi^n}{\delta t}  d\bx,
\end{align}
\end{subequations}
where $\overline{\phi}^{n+\frac{1}{2}}=\frac{3}{2}{\phi}^{n}-\frac{1}{2}{\phi}^{n-1}$ and $V(\overline{\phi}^{n+\frac{1}{2}})=F'(\overline{\phi}^{n+\frac{1}{2}})-\gamma_0\overline{\phi}^{n+\frac{1}{2}}$.
\end{scheme}

The SAV-CN scheme has the following discrete energy law.
\begin{theorem}\label{Theo:SAV-CN}
The Scheme \ref{sch:SAV-CN} is unconditionally energy stable in the sense that\cite{ShenXuYang2019}

\begin{subequations}
\begin{align}
& \frac{1}{2}({\phi}^{n+1},  (\cL+\gamma_0 I ){\phi}^{n+1})+({q}^{n+1})^2-\frac{1}{2}({\phi}^{n},  (\cL+\gamma_0 I ){\phi}^{n})-({q}^{n})^2=- {\delta t}  (\cG\mu^{n+1},\mu^{n+1}). \notag
\end{align}
\end{subequations}
\end{theorem}

\subsection{Multiple scalar auxiliary variable (MSAV) method}
In the previous subsection, we briefly illustrate the idea of the SAV method. In many scenarios, the energy expression is complicated. Thus, introducing more than one auxiliary variables is necessary. To present the Multiple scalar auxiliary variable (MSAV) method \cite{Chenshen2018}, we consider a more general form of the free energy
\beq \label{eq:general-energy}
\cE = \int_\Omega \Big( \frac{1}{2} \phi \cL\phi + \sum_{i=1}^{k}F_i(\phi)  \Big) d\bx,
\eeq 
where $\cL$ is a linear operator, and $F_i(\phi)$, $i=1,2,\cdots,k$ are the bulk potentials. Then the general gradient flow model in \eqref{eq:generic-model} is specified as
\beq  \label{eq:gradient-flow-M}
\frac{\partial \phi}{\partial t} = -\cG  \Big[ \cL \phi + \sum_{i=1}^k F_i'(\phi) \Big],
\eeq 
%where $\phi$ is the state variable, $\cL$ is a linear operator, $\cG$ is a semi-positive definite operator for dissipative systems (and a skew-symmetric operator for reversible systems or Hamiltonian system).  We consider periodic boundary condition for simplicity, though all our results apply  to models with thermodynamically consistent boundary conditions.
which has the following energy law
\beq
\frac{d}{dt} \cE(\phi) = (\frac{\delta \cE}{\delta \phi}, \frac{\partial \phi}{\partial t}) = - \Big( \Big[ \cL \phi + \sum_{i=1}^k F_i'(\phi) \Big], \cG \Big[ \cL \phi + \sum_{i=1}^k F_i'(\phi) \Big] \Big).
\eeq 
When $\cG$ is semi-positive definite, we have $\frac{d\cE}{dt} \leq 0$, and when $\cG$ is a skew-symmetric operator, we have $\frac{d\cE}{dt} = 0$. For the problem in \eqref{eq:gradient-flow-M}, multiple scalar auxiliary variables are introduced as 
\beq \label{eq:MSAV-q}
q_i(t) :=Q_i(\phi) = \sqrt{ \int_\Omega  (F_i(\phi) - \frac{1}{2}\gamma_i \phi^2) d\bx +C_i}, \quad i=1,2,\cdots,k.
\eeq
Here $C_i$ are positive constants that make sure $q_i(t)$ are well defined. And $\gamma_i$ are the regularization constants \cite{ChenZhaoYangANM2018}. 
With the scalar auxiliary variables $q_i(t)$, the gradient flow model \eqref{eq:gradient-flow-M} can be reformulated into an equivalent form
\begin{subequations} \label{eq:general-gradient-flow-MSAV}
\begin{align}
& \frac{\partial \phi}{\partial t} = - \cG \Big( \cL\phi + \sum_{i=1}^k \gamma_i \phi  +\sum_{i=1}^{k} \frac{ q_i(t)}{Q_i(\phi)} V_i(\phi) \Big), \\
& \frac{d q_j(t)}{dt} = \frac{1}{2Q_j(\phi)} \int_\Omega V_j(\phi) \partial_t\phi d\bx, \quad  j=1,2,\cdots,k,
\end{align}
\end{subequations}
where $V_i(\phi)=F_i'(\phi)-\gamma_i \phi$,  $i= 1,...,k$.

With the introduction of the multiple scalar auxiliary variables in \eqref{eq:MSAV-q}, we can get the modified free energy as:
\beq \label{eq:MSAV-energy}
\hat{E} = \int_\Omega \Big( \frac{1}{2}\phi \cL\phi + \sum_{i=1}^k \frac{\gamma_i}{2}\phi^2 \Big) d\bx +\sum_{i=1}^{k}(q_i^2 - C_i).
\eeq 
For the reformulated model \eqref{eq:general-gradient-flow-MSAV}, it has the following energy law
\begin{subequations}
\begin{align}
&\frac{d\hat{E}}{dt} = \int_\Omega \frac{\delta \hat{E}}{\delta \phi}\frac{\partial \phi}{\partial t} d\bx + \sum_{i=1}^{k}\frac{\delta \hat{E}}{\delta q_i}\frac{dq_i}{dt}\notag \\
&= -\Big( \Big[ \cL\phi + \sum_{i=1}^k \gamma_i \phi  +\sum_{i=1}^{k} \frac{ q_i(t)}{Q_i(\phi)} V_i(\phi) \Big]  , \cG \Big[ \cL\phi + \sum_{i=1}^k \gamma_i \phi  +\sum_{i=1}^{k} \frac{ q_i(t)}{Q_i(\phi)} V_i(\phi) \Big] \Big).\notag
\end{align}
\end{subequations}

Similarly, second-order numerical schemes can be designed for the reformulated model in \eqref{eq:general-gradient-flow-MSAV}. In particular, the following two schemes can be easily obtained.

\begin{scheme}[Second-order MSAV-BDF2 Scheme]
\begin{subequations}\label{eq:MSAV-BDF}
\begin{align}
& \frac{3 \phi^{n+1}-4\phi^n+\phi^{n-1}}{2\delta t}  = - \cG \mu^{n+1}, \\
& \mu^{n+1}=   \cL\phi^{n+1} + \sum_{i=1}^k \gamma_i \phi^{n+1} + \sum_{i=1}^{k}\frac{ q_i^{n+1}}{Q_i(\overline{\phi}^{n+1})} V_i(\overline{\phi}^{n+1}), \\
& \frac{ 3q_j^{n+1} - 4q_j^{n}+q_j^{n-1}}{2\delta t} = \int_\Omega  \frac{V_j(\overline{\phi}^{n+1})}{2Q_j(\overline{\phi}^{n+1})} \frac{3 \phi^{n+1}-4\phi^n+\phi^{n-1}}{2\delta t}  d\bx, \quad j=1,2,\cdots,k,
\end{align}
\end{subequations}
where $\overline{\phi}^{n+1}=\frac{3}{2}{\phi}^{n}-\frac{1}{2}{\phi}^{n-1}$ and $V_i(\overline{\phi}^{n+1})=F_i'(\overline{\phi}^{n+1})-\gamma_i\overline{\phi}^{n+1}$, $i=1,...,k$.
\end{scheme}

\begin{scheme}[Second-order MSAV-CN Scheme]
\begin{subequations}\label{eq:MSAV-CN}
\begin{align}
& \frac{\phi^{n+1}-\phi^n}{\delta t} = - \cG \mu^{n+\frac{1}{2}}, \\
& \mu^{n+\frac{1}{2}} = \cL\phi^{n+\frac{1}{2}} + \sum_{i=1}^k \gamma_i  \phi^{n+\frac{1}{2}} + \sum_{i=1}^{k}\frac{ q_i^{n+\frac{1}{2}}}{Q_i(\overline{\phi}^{n+\frac{1}{2}})} V_i(\overline{\phi}^{n+\frac{1}{2}}),\\
& \frac{ q_j^{n+1} - q_j^{n}}{\delta t} = \int_\Omega \frac{ V_j(\overline{\phi}^{n+\frac{1}{2}})}{2Q_j(\overline{\phi}^{n+\frac{1}{2}})} \frac{ \phi^{n+1}-\phi^n}{\delta t}  d\bx, \quad j=1,2,\cdots, k,
\end{align}
\end{subequations}
where $\overline{\phi}^{n+\frac{1}{2}}=\frac{3}{2}{\phi}^{n}-\frac{1}{2}{\phi}^{n-1}$ and $V_i(\overline{\phi}^{n+\frac{1}{2}})=F_i'(\overline{\phi}^{n+\frac{1}{2}})-\gamma_i\overline{\phi}^{n+\frac{1}{2}}$, $i=1,...,k$. 
\end{scheme}

The MSAV-BDF2 scheme has the following discrete energy law.
\begin{theorem}\label{Theo:MSAV-BDF}
The Scheme \ref{eq:MSAV-BDF} is unconditionally energy stable in the sense that \cite{Chenshen2018}
\begin{subequations}
\begin{align}
&  \frac{1}{4}[({\phi}^{n+1},  (\cL+ \sum_{i=1}^k \gamma_i  I ){\phi}^{n+1})+(2{\phi}^{n+1}-{\phi}^{n},  (\cL+ \sum_{i=1}^k \gamma_i I)(2{\phi}^{n+1}-{\phi}^{n}))]\notag\\
&\quad  -\frac{1}{4}[({\phi}^{n},  (\cL+ \sum_{i=1}^k\gamma_i  I ){\phi}^{n})+(2{\phi}^{n}-{\phi}^{n-1},  (\cL+ \sum_{i=1}^k\gamma_i I)(2{\phi}^{n}-{\phi}^{n-1}))]\notag\\
& \qquad\quad
+\sum_{i=1}^{k}[({q_i}^{n+1})^2+(2{q_i}^{n+1}-{q_i}^{n})^2] - \sum_{i=1}^{k}[({q_i}^{n})^2+(2{q_i}^{n}-{q_i}^{n-1})^2]   \notag\\
& \qquad\qquad\qquad \qquad\qquad\qquad\qquad \qquad\qquad\qquad  \le- {\delta t}  (\cG\mu^{n+1},\mu^{n+1}).\notag
\end{align}
\end{subequations}
\end{theorem}

Meanwhile, the MSAV-CN scheme has the following discrete energy law.
\begin{theorem}\label{Theo:MSAV-CN}
The Scheme \ref{eq:MSAV-CN} is unconditionally energy stable in the sense that\cite{Chenshen2018}
\[\frac{1}{2}({\phi}^{n+1},  (\cL+\sum_{i=1}^k \gamma_i I ){\phi}^{n+1})+\sum_{i=1}^{k}({{q}_i}^{n+1})^2-\frac{1}{2}({\phi}^{n},  (\cL+\sum_{i=1}^k \gamma_i I ){\phi}^{n})-\sum_{i=1}^{k}({{q}_i}^{n})^2=- {\delta t} (\cG\mu^{n+\frac{1}{2}},\mu^{n+\frac{1}{2}}).\]
%\begin{subequations}
%\begin{align}
%& \frac{1}{2}({\phi}^{n+1},  (\cL+\sum_{i=1}^k \gamma_i I %){\phi}^{n+1})+\sum_{i=1}^{k}({{q}_i}^{n+1})^2-\frac{1}{2}({\phi}^{n%},  (\cL+\sum_{i=1}^k \gamma_i I %){\phi}^{n})-\sum_{i=1}^{k}({{q}_i}^{n})^2\notag\\
%&\qquad\qquad\qquad\qquad\qquad \qquad\qquad\qquad\qquad \qquad=- %{\delta t}  (\cG\mu^{n+\frac{1}{2}},\mu^{n+\frac{1}{2}}).\notag
%\end{align}
%\end{subequations}
\end{theorem}

\section{Our Remedy: the relaxed SAV (RSAV) method}

Notice the definition in \eqref{eq:SAV-q} tells us $\frac{q(t)}{Q(\phi)}=1$. Hence, we can observe the two energies \eqref{eq:original-energy} and \eqref{eq:SAV-energy} are equivalent in the PDE level. Meanwhile, the two PDE models \eqref{eq:general-gradient-flow} and \eqref{eq:gradient-flow-SAV} are equivalent. However, after temporal discretization, the numerical results of $q(t)$ and $Q(\phi)$ are not equal anymore, which means the discrete energies of \eqref{eq:original-energy} and \eqref{eq:SAV-energy} are not necessarily equivalent anymore. The major issue is that $q^{n+1}$ is no longer equal to $Q(\phi^{n+1})$ numerically. Thus, an energy stable scheme that satisfies the modified energy law in \eqref{eq:SAV-energy} does not necessarily satisfy the original energy law in \eqref{eq:original-energy}. 

To fix the inconsistency issue for $q^{n+1}$ and $Q(\phi^{n+1})$ (that are supposed to be equal as introduced in \eqref{eq:SAV-q}), we propose a relaxation technique to penalize the difference between $q^{n+1}$ and $Q(\phi^{n+1})$. As will be clear in the following sections, the RSAV method introduces negligible extra computational cost, but it inherits all the baseline SAV method's good properties.

\subsection{Baseline RSAV method}
To start with a simple case, we introduce the RSAV method to solve the problem in \eqref{eq:gradient-flow-SAV}. If we utilize the semi-implicit BDF2 time marching method, we have the following RSAV-BDF2 scheme.

\begin{scheme}[Second-order RSAV-BDF2 Scheme]  \label{scheme:RSAV-BDF2}
We can update $\phi^{n+1}$ via the following two steps:
\begin{itemize}
\item Step 1. Calculate the intermediate solution ($\phi^{n+1}$, $\tilde{q}^{n+1}$) from the baseline SAV method.
\begin{subequations}\label{eq:RSAV-BDF2_step1}
\begin{align}
& \frac{3 \phi^{n+1}-4\phi^n+\phi^{n-1}}{2\delta t} =\mu^{n+1}, \\
&\mu^{n+1}=- \cG \Big(  \cL\phi^{n+1} + \gamma_0 \phi^{n+1} + \frac{ \tilde{q}^{n+1}}{Q(\overline{\phi}^{n+1})} V(\overline{\phi}^{n+1})\Big), \\
& \frac{ 3\tilde{q}^{n+1} - 4q^{n}+q^{n-1}}{2\delta t} = \int_\Omega \frac{ V(\overline{\phi}^{n+1})}{2Q(\overline{\phi}^{n+1})} \frac{3 \phi^{n+1}-4\phi^n+\phi^{n-1}}{2\delta t}  d\bx,
\end{align}
\end{subequations}
Where $\overline{\phi}^{n+1}=\frac{3}{2}{\phi}^{n}-\frac{1}{2}{\phi}^{n-1}$ and $V(\overline{\phi}^{n+1})=F'(\overline{\phi}^{n+1})-\gamma_0\overline{\phi}^{n+1}$.
\item Step 2. Update the scalar auxiliary variable $q^{n+1}$ via a relaxation step as 
\beq\label{new_r_BDF}
q^{n+1} = \xi_0 \tilde{q}^{n+1} + (1-\xi_0) Q(\phi^{n+1}), \quad \xi_0 \in \cV.
\eeq
Here, $\cV$ is a set defined by $\cV = \cV_1 \cap \cV_2$, where
\begin{subequations}
\begin{align}
& \cV_1= \{\xi | \xi \in [0, 1] \}, \\
& \cV_2 = \Big\{ \xi | ({q}^{n+1})^2+(2{q}^{n+1}-{q}^{n})^2-((\tilde{q}^{n+1})^2+(2\tilde{q}^{n+1}-{q}^{n})^2)  \nonumber \\
& \qquad \qquad  \le  {\delta t} \eta  (\cG\mu^{n+1},\mu^{n+1}), \quad q^{n+1} = \xi \tilde{q}^{n+1} + (1-\xi) Q(\phi^{n+1}) \Big\}.
\end{align}
\end{subequations}
Here, $\eta \in [0,1]$ is an artificial parameter that can be manually assigned.
\end{itemize}
\end{scheme}

Several vital observations for Scheme \ref{scheme:RSAV-BDF2} are given as follows.
\begin{rem}
We emphasis that the set $\cV$ in \eqref{new_r_BDF} is non-empty by noticing $1 \in \cV$.
\end{rem}

\begin{rem}
Scheme \ref{scheme:RSAV-BDF2} is second-order accurate in time. In particular, the relaxation step in \eqref{new_r_BDF} doesn't not affect the order of accuracy in time. Notice $\tilde{q}^{n+1} = Q(\phi(\bx, t_{n+1})) + O(\delta t^2)$ and $Q(\phi^{n+1}) = Q(\phi(\bx, t_{n+1})) + O(\delta t^2)$. Hence 
$$
q^{n+1} =  \xi_0 \tilde{q}^{n+1} + (1-\xi_0) Q(\phi^{n+1})= Q(\phi(\bx, t_{n+1})) + O(\delta t^2).
$$
\end{rem}

\begin{rem}[Optimal choice for $\xi_0$]
Here we explain the optimal choice for the relaxation parameter $\xi_0$ in \eqref{new_r_BDF}. $\xi_0$ can be chosen as a solution of the following optimization problem,
\beq
\xi_0 = \min_{\xi \in [0, 1]} \xi, \quad \mbox{ s.t.  }  ({q}^{n+1})^2+(2{q}^{n+1}-{q}^{n})^2-((\tilde{q}^{n+1})^2+(2\tilde{q}^{n+1}-{q}^{n})^2)\le  {\delta t} \eta  (\cG\mu^{n+1},\mu^{n+1}),
\eeq 
with $q^{n+1} = \xi\tilde{q}^{n+1} + (1-\xi)Q(\phi^{n+1})$. This can be simplified as
\beq \label{eq:BDF-optimal-xi0}
\xi_0 = \min_{\xi \in [0, 1]} \xi, \quad \mbox{ s.t.  }  a\xi^2 + b\xi + c \leq 0,
\eeq 
where the coefficients are
\begin{subequations}
\begin{align}
&a=5(\tilde{q}^{n+1}-Q(\phi^{n+1}))^2,\notag \\
&b=2(\tilde{q}^{n+1}-Q(\phi^{n+1}))(5Q(\phi^{n+1})-2q^n),\notag \\
&c=(Q(\phi^{n+1}))^2+(2Q(\phi^{n+1})-q^n)^2-(\tilde{q}^{n+1})^2-(2\tilde{q}^{n+1}-q^n)^2-{\delta t}\eta (\cG \mu^{n+1},\mu^{n+1}).\notag 
\end{align}
\end{subequations}
Notice the fact $\delta t \eta (\cG \mu^{n+1}, \mu^{n+1}) \geq 0$, and $a+b+c \leq 0$. Given $a\neq 0$, the optimization problem in \eqref{eq:BDF-optimal-xi0} can be solved as
$$\xi_0 = \max\{0, \frac{-b -\sqrt{b^2-4ac}}{2a}\}.$$
\end{rem}

% where $\xi_0 = \max\{0, \frac{-b -\sqrt{b^2-4ac}}{2a}\}$,with the variables 
% \begin{subequations}
% \begin{align}
% &a=\frac{5}{2}(\tilde{q}^{n+1}-Q(\phi^{n+1})^2,\notag \\
% &b=(\tilde{q}^{n+1}-Q(\phi^{n+1}))Q(\phi^{n+1})+(\tilde{q}^{n+1}-Q(\phi^{n+1}))(2Q(\phi^{n+1})-q^n),\notag \\
% &c=\frac{1}{2}[(Q(\phi^{n+1}))^2+(2Q(\phi^{n+1})-q^n)^2-(\tilde{q}^{n+1})^2-(2\tilde{q}^{n+1}-q^n)^2]-{\delta t}\eta (\cG \mu^{n+1},\mu^{n+1}),\notag 
% \end{align}
% \end{subequations}

\begin{theorem}
The Scheme \ref{scheme:RSAV-BDF2} is unconditionally energy stable.
\end{theorem}

\begin{proof}
According to the Theorem \ref{Theo:SAV-BDF}, we could get
\begin{subequations}
\begin{align}
& \frac{1}{4}[({\phi}^{n+1},  (\cL+\gamma_0 I ){\phi}^{n+1})+(2{\phi}^{n+1}-{\phi}^{n},  (\cL+\gamma_0 I)({\phi}^{n+1}-{\phi}^{n}))]+(\tilde{q}^{n+1})^2+(2\tilde{q}^{n+1}-{q}^{n})^2 \notag \\
&\qquad -\frac{1}{4}[({\phi}^{n},  (\cL+\gamma_0 I ){\phi}^{n})+(2{\phi}^{n}-{\phi}^{n-1},  (\cL+\gamma_0 I)({\phi}^{n}-{\phi}^{n-1}))]-({q}^{n})^2-(2{q}^{n}-{q}^{n-1})^2\notag \\
& \qquad\qquad\qquad\qquad \qquad\qquad\qquad\qquad \qquad\qquad\qquad\qquad \le- {\delta t} (\cG \mu^{n+1},\mu^{n+1}),\notag 
\end{align}
\end{subequations}
for the first step of the scheme \ref{scheme:RSAV-BDF2}.
At the same time,  we could get
\beq
(({q}^{n+1})^2+(2{q}^{n+1}-{q}^{n})^2)-((\tilde{q}^{n+1})^2+(2\tilde{q}^{n+1}-{q}^{n})^2)\leq  {\delta t} \eta  (\cG\mu^{n+1},\mu^{n+1}),
\eeq 
from \eqref{new_r_BDF}.
Adding the above two equations together, we could have
\begin{subequations}
\begin{align}
& \frac{1}{4}[({\phi}^{n+1},  (\cL+\gamma_0 I ){\phi}^{n+1})+(2{\phi}^{n+1}-{\phi}^{n},  (\cL+\gamma_0 I)({\phi}^{n+1}-{\phi}^{n}))]+({q}^{n+1})^2+(2{q}^{n+1}-{q}^{n})^2 \notag \\
&\qquad -\frac{1}{4}[({\phi}^{n},  (\cL+\gamma_0 I ){\phi}^{n})+(2{\phi}^{n}-{\phi}^{n-1},  (\cL+\gamma_0 I)({\phi}^{n}-{\phi}^{n-1}))]-({q}^{n})^2-(2{q}^{n}-{q}^{n-1})^2\notag \\
& \qquad\qquad\qquad\qquad \qquad\qquad\qquad\qquad \qquad\qquad \le- {\delta t}(1-\eta) (\cG \mu^{n+1},\mu^{n+1})\le0,\notag 
\end{align}
\end{subequations}
since $1-\eta \ge 0$. This completes the proof.
\end{proof}

\begin{scheme}[Second-order RSAV-CN Scheme] \label{scheme:RSAV-CN}
We update $\phi^{n+1}$ via the following two steps:
\begin{itemize}
\item Step 1. Calculate the intermediate solution ($\phi^{n+1}$, $\tilde{q}^{n+1}$) using the baseline SAV-CN method as below.
\begin{subequations}\label{eq:RSAV-CN_step1}
\begin{align}
& \frac{\phi^{n+1}-\phi^n}{\delta t} = - \cG \mu^{n+\frac{1}{2}}, \\
& \mu^{n+\frac{1}{2}} = \cL\phi^{n+\frac{1}{2}} + \gamma_0 \phi^{n+\frac{1}{2}} + \frac{ \tilde{q}^{n+\frac{1}{2}}}{Q(\overline{\phi}^{n+\frac{1}{2}})} V(\overline{\phi}^{n+\frac{1}{2}}),\\
& \frac{ \tilde{q}^{n+1} - q^{n}}{\delta t} = \int_\Omega \frac{ V(\overline{\phi}^{n+\frac{1}{2}})}{2Q(\overline{\phi}^{n+\frac{1}{2}})} \frac{ \phi^{n+1}-\phi^n}{\delta t}  d\bx.
\end{align}
\end{subequations}
\item Step 2. Update the scalar auxiliary variable $q^{n+1}$ as
\beq \label{eq:RSAV-CN-step2}
q^{n+1} = \xi_0 \tilde{q}^{n+1} + (1-\xi_0) Q(\phi^{n+1}), \quad  \xi_0 \in \cV,
\eeq
with the feasible set $\cV$ defined as $\cV=\cV_1\cap \cV_2$, where 
\begin{subequations}
\begin{align}
& \cV_1= \{\xi | \xi \in [0, 1] \}, \\
& \cV_2 = \Big\{ \xi | ({q}^{n+1})^2- (\tilde{q}^{n+1})^2  \le  {\delta t} \eta  (\cG\mu^{n+\frac{1}{2}},\mu^{n+\frac{1}{2}}), \quad q^{n+1} = \xi \tilde{q}^{n+1} + (1-\xi) Q(\phi^{n+1}) \Big\}.
\end{align}
\end{subequations}
\end{itemize}
\end{scheme}

Similarly, we have the following critical observations for Scheme \ref{scheme:RSAV-CN}.
\begin{rem}
The set $\cV$ in \eqref{eq:RSAV-CN-step2} is non-empty, given that $1 \in \cV$.
\end{rem}
\begin{rem}
The Scheme \ref{scheme:RSAV-CN} is second-order accurate in time. Because the relaxation step in \eqref{eq:RSAV-CN-step2} doesn't not affect the order of accuracy. Notice $\tilde{q}^{n+1} = Q(\phi(\bx, t_{n+1})) + O(\delta t^2)$ and $Q(\Phi^{n+1}) = Q(\phi(\bx, t_{n+1})) + O(\delta t^2)$. Hence 
$$
q^{n+1} =  \xi_0 \tilde{q}^{n+1} + (1-\xi_0) Q(\phi^{n+1})= Q(\phi(\bx, t_{n+1})) + O(\delta t^2).
$$
\end{rem}
\begin{rem}[Optimal choice for $\xi_0$]
Here we elaborate the optimal choice for the relaxation parameter $\xi_0$.  We can choose $\xi_0$ as the solution of the following optimization problem
\beq \label{eq:optimization-CN}
\xi_0 = \min_{\xi \in [0, 1]} \xi, \quad \mbox{ s.t.  }   (q^{n+1})^2 - (\tilde{q}^{n+1})^2 \leq \delta t \eta (\mu^{n+\frac{1}{2}}, \cG\mu^{n+\frac{1}{2}}).
\eeq 
This can be simplified as
\beq
\xi_0 = \min_{\xi \in [0, 1]} \xi, \quad \mbox{ s.t.  }  a\xi^2 + b\xi + c \leq 0,
\eeq 
where the coefficients are
\begin{subequations}
\begin{align}
& a = ( \tilde{q}^{n+1} - Q(\phi^{n+1}))^2, \quad b = 2\big(\tilde{q}^{n+1}-Q(\phi^{n+1})\big)Q(\phi^{n+1}), \\
& c = [Q(\phi^{n+1})]^2 - (\tilde{q}^{n+1})^2 - \delta t \eta  (\mu^{n+\frac{1}{2}}, \cG\mu^{n+\frac{1}{2}}).
\end{align}
\end{subequations}
Notice $a+b+c<0$. Given $a \neq 0$, the solution to \eqref{eq:optimization-CN} is given as 
$$
\xi_0 = \max\{0, \frac{-b -\sqrt{b^2-4ac}}{2a}\}.
$$
\end{rem}

% where $\xi_0 = \max\{0, \frac{-b -\sqrt{b^2-4ac}}{2a}\}$,with the variables 
% \begin{subequations}
% \begin{align}
% & a = \frac{1}{2}( \tilde{q}^{n+1} - Q(\phi^{n+1}))^2, \quad b = \tilde{q}^{n+1} Q(\phi^{n+1}) - [Q(\phi^{n+1})]^2, \\
% & c = \frac{1}{2} [Q(\phi^{n+1})]^2 - (\tilde{q}^{n+1})^2 - \delta t \eta  (\mu^{n+\frac{1}{2}}, \cG\mu^{n+\frac{1}{2}}).
% \end{align}
% \end{subequations}

\begin{theorem}
The Scheme \ref {scheme:RSAV-CN} is unconditionally energy stable.
\end{theorem}

\begin{proof}
For the step in \ref{eq:RSAV-CN_step1}, thanks to the Theorem \ref{Theo:SAV-CN}, we could get
\beq  \label{eq:RSAV-CN-proof-1}
\frac{1}{2}({\phi}^{n+1},  (\cL+\gamma_0 I ){\phi}^{n+1})+(\tilde{q}^{n+1})^2-\frac{1}{2}({\phi}^{n},  (\cL+\gamma_0 I ){\phi}^{n})-({q}^{n})^2=- {\delta t}  (\cG\mu^{n+\frac{1}{2}},\mu^{n+\frac{1}{2}}).
\eeq 
From \eqref{eq:RSAV-CN-step2}, we know
\beq \label{eq:RSAV-CN-proof-2}
({q}^{n+1})^2-(\tilde{q}^{n+1})^2\le  {\delta t} \eta  (\cG\mu^{n+\frac{1}{2}},\mu^{n+\frac{1}{2}}).
\eeq 

Adding two equations \eqref{eq:RSAV-CN-proof-1} and \eqref{eq:RSAV-CN-proof-2}, and acknowledging the inequality $1-\eta \ge 0$, we could arrive at
\begin{subequations}
\begin{align}
&\frac{1}{2}({\phi}^{n+1},  (\cL+\gamma_0 I ){\phi}^{n+1})+({q}^{n+1})^2-\frac{1}{2}({\phi}^{n},  (\cL+\gamma_0 I ){\phi}^{n})-({q}^{n})^2\notag \\
& \qquad\qquad\qquad\qquad \qquad\qquad\qquad\qquad \qquad\qquad \le- {\delta t}(1-\eta)  (\cG\mu^{n+\frac{1}{2}},\mu^{n+\frac{1}{2}})\le0.\notag 
\end{align}
\end{subequations}
 This completes the proof.
\end{proof}

\subsection{The relaxed MSAV approach}
In a similar manner, we introduce the relaxation technique to the MASV method to fix the inconsistency issues between $q_i(t)$ and $Q_i(\phi)$ after discretization. The two second-order MSAV schemes could be improved as follows.

\begin{scheme}[Second-order RMSAV-BDF2 Scheme]  \label{scheme:RMSAV-BDF2}
We update $\phi^{n+1}$ via the following two steps:
\begin{itemize}
\item Step 1, Calculate the intermediate solution ($\phi^{n+1}$, $\tilde{q}^{n+1}$) using the MSAV-BDF2 method as below.
\begin{subequations}\label{eq:RMSAV-BDF2_step1}
\begin{align}
& \frac{3 \phi^{n+1}-4\phi^n+\phi^{n-1}}{2\delta t}  = - \cG \mu^{n+1}, \\
& \mu^{n+1}=   \cL\phi^{n+1} + \sum_{i=1}^k \gamma_i  \phi^{n+1} + \sum_{i}^{k}\frac{\tilde{q_i}^{n+1}}{Q_i(\overline{\phi}^{n+1})} V_i(\overline{\phi}^{n+1}), \\
& \frac{ 3\tilde{q_j}^{n+1} - 4q_j^{n}+q_j^{n-1}}{2\delta t} = \int_\Omega  \frac{V_j(\overline{\phi}^{n+1})}{2Q_j(\overline{\phi}^{n+1})} \frac{3 \phi^{n+1}-4\phi^n+\phi^{n-1}}{2\delta t}  d\bx, \quad j=1,2,\cdots, k,
\end{align}
\end{subequations}
where $\overline{\phi}^{n+1}=\frac{3}{2}{\phi}^{n}-\frac{1}{2}{\phi}^{n-1}$ and $V_i(\overline{\phi}^{n+1})=F_i'(\overline{\phi}^{n+1})-\gamma_i \overline{\phi}^{n+1}$, $i=1,2,\cdots, k$.
\item Step 2, update the $k$ scalar auxiliary variables as 
\beq\label{RMSAV_new_r_BDF}
q_i^{n+1} = \xi_0 \tilde{q_i}^{n+1} + (1-\xi_0) Q_i(\phi^{n+1}), \quad i=1,2,\cdots, k, \quad \xi_0 \in \cV,
\eeq
where the feasible set $\cV$ is defined as $\cV=\cV_1\cap \cV_2$ with
\begin{subequations}
\begin{align}
& \cV_1=\{\xi| \xi \in [0, 1] \}, \\
& \cV_2 = \Big\{ \xi | \sum_{i=1}^k \Big[  ({q}_i^{n+1})^2+(2{q}_i^{n+1}-{q}_i^{n})^2-((\tilde{q}_i^{n+1})^2+(2\tilde{q}_i^{n+1}-{q}_i^{n})^2) \Big] \nonumber \\
&   \le  {\delta t} \eta  (\cG\mu^{n+1},\mu^{n+1}), \quad q_i^{n+1} = \xi \tilde{q}_i^{n+1} + (1-\xi) Q(\phi^{n+1}), i=1,2,\cdots,k \Big\}, \quad \eta \in [0, 1].
\end{align}
\end{subequations}
\end{itemize}

\begin{rem}[Optimal choice for $\xi_0$]
Similarly, we can propose an optimal choice for the relaxation parameter $\xi_0$ as the solution of an optimization problem. And in the end, 
$$
\xi_0 = \max\{0, \frac{-b -\sqrt{b^2-4ac}}{2a}\},
$$
with the coefficients given by
\begin{subequations}
\begin{align}
&a=5\sum_{i=1}^{k}(\tilde{q_i}^{n+1}-Q_i(\phi^{n+1}))^2,\notag \\
&b=\sum_{i=1}^{k}2(\tilde{q_i}^{n+1}-Q_i(\phi^{n+1}))(5Q_i(\phi^{n+1})-2q_i^n),\notag \\
&c=\sum_{i=1}^{k}[(Q_i(\phi^{n+1}))^2+(2Q_i(\phi^{n+1})-q_i^n)^2-(\tilde{q_i}^{n+1})^2-(2\tilde{q_i}^{n+1}-q_i^n)^2]-{\delta t}\eta  (\cG\mu^{n+1},\mu^{n+1}).\notag 
\end{align}
\end{subequations}
\end{rem}

\end{scheme}
\begin{theorem}
The Scheme \ref{scheme:RMSAV-BDF2} is unconditionally energy stable.
\end{theorem}

\begin{proof}
For the first step of the scheme \ref{scheme:RMSAV-BDF2}, according to the Theorem \ref{Theo:MSAV-BDF}, We could get
\begin{subequations}
\begin{align}
&  \frac{1}{4}[({\phi}^{n+1},  (\cL+ \sum_{i=1}^k \gamma_i  I ){\phi}^{n+1})+(2{\phi}^{n+1}-{\phi}^{n},  (\cL+ \sum_{i=1}^k \gamma_i I)(2{\phi}^{n+1}-{\phi}^{n}))]\notag\\
&\quad  -\frac{1}{4}[({\phi}^{n},  (\cL+ \sum_{i=1}^k \gamma_i  I ){\phi}^{n})+(2{\phi}^{n}-{\phi}^{n-1},  (\cL+ \sum_{i=1}^k \gamma_i  I)(2{\phi}^{n}-{\phi}^{n-1}))]\notag\\
& \qquad\qquad\quad +\sum_{i=1}^{k}[({\tilde{q}_i}^{n+1})^2+(2{\tilde{q}_i}^{n+1}-{q_i}^{n})^2] - \sum_{i=1}^{k}[({q_i}^{n})^2+(2{q_i}^{n}-{q_i}^{n-1})^2]   \notag\\
& \qquad\qquad\qquad \qquad\qquad\qquad\qquad \qquad\qquad\qquad\quad  \le- {\delta t}  (\cG\mu^{n+1},\mu^{n+1}).\notag
\end{align}
\end{subequations}
At the same time, we know from \eqref{RMSAV_new_r_BDF}, we could know
\beq
\sum_{i=1}^{k}[({q_i}^{n+1})^2+(2{q_i}^{n+1}-{q}^{n})^2]-\sum_{i=1}^{k}[({\tilde{q}_i}^{n+1})^2+(2{\tilde{q}_i}^{n+1}-{q_i}^{n})^2]\le  {\delta t} \eta  (\cG\mu^{n+1},\mu^{n+1}).
\eeq 
Adding the above two equations together, we could have
\begin{subequations}
\begin{align}
&  \frac{1}{4}[({\phi}^{n+1},  (\cL+ \sum_{i=1}^k \gamma_i I ){\phi}^{n+1})+(2{\phi}^{n+1}-{\phi}^{n},  (\cL+ \sum_{i=1}^k \gamma_i  I)({\phi}^{n+1}-{\phi}^{n}))]\notag\\
&\quad  -\frac{1}{4}[({\phi}^{n},  (\cL+ \sum_{i=1}^k \gamma_i  I ){\phi}^{n})+(2{\phi}^{n}-{\phi}^{n-1},  (\cL+ \sum_{i=1}^k \gamma_i I)({\phi}^{n}-{\phi}^{n-1}))]\notag\\
& \qquad\qquad +\sum_{i=1}^{k}[({{q}_i}^{n+1})^2+(2{{q}_i}^{n+1}-{q_i}^{n})^2] - \sum_{i=1}^{k}[({q_i}^{n})^2+(2{q_i}^{n}-{q_i}^{n-1})^2]   \notag\\
& \qquad\qquad\qquad\qquad \qquad\qquad\qquad\qquad \quad \le- {\delta t}(1-\eta)  (\cG\mu^{n+1},\mu^{n+1})\le0, \notag 
\end{align}
\end{subequations}
by using the fact $1-\eta \ge 0$. This completes the proof.

\end{proof}

Then, if we use the semi-implicit CN time discretization, we can obtain the the following second-order scheme.
\begin{scheme}[Second-order RMSAV-CN Scheme] \label{scheme:RMSAV-CN}
We update $\phi^{n+1}$ via the following two steps:
\begin{itemize}
\item Step 1. Calculate the intermediate solution
\begin{subequations}\label{eq:RMSAV-CN_step1}
\begin{align}
& \frac{\phi^{n+1}-\phi^n}{\delta t} = - \cG \mu^{n+\frac{1}{2}}, \\
& \mu^{n+\frac{1}{2}} = \cL\phi^{n+\frac{1}{2}} + \sum_{i=1}^k \gamma_i \phi^{n+\frac{1}{2}} + \sum_{i=1}^{k}\frac{ \tilde{q}_i^{n+\frac{1}{2}}}{Q_i(\overline{\phi}^{n+\frac{1}{2}})} V_i(\overline{\phi}^{n+\frac{1}{2}}),\\
& \frac{ \tilde{q}_j^{n+1} - q_j^{n}}{\delta t} = \int_\Omega \frac{ V_j(\overline{\phi}^{n+\frac{1}{2}})}{2Q_j(\overline{\phi}^{n+\frac{1}{2}})} \frac{ \phi^{n+1}-\phi^n}{\delta t}  d\bx, \quad j=1,2,\cdots,k.
\end{align}
\end{subequations}
\item Step 2. Update the scalar auxiliary variable as
\beq \label{eq:RMSAV-CN-step2}
{q_i}^{n+1} = \xi_0 {\tilde{q}_i}^{n+1} + (1-\xi_0) Q_i(\phi^{n+1}), \quad i=1,2,\cdots, k, \quad \xi_0 \in \cV,
\eeq
with the feasible set $V$ defined as $\cV=\cV_1\cap \cV_2$, where 
\begin{subequations}
\begin{align}
& \cV_1= \{\xi | \xi \in [0, 1] \}, \\
& \cV_2 = \Big\{ \xi | \sum_{i=1}^k \Big[ ({q}_i^{n+1})^2- (\tilde{q}_i^{n+1})^2 \Big]  \le  {\delta t} \eta  (\cG\mu^{n+\frac{1}{2}},\mu^{n+\frac{1}{2}}), \nonumber \\
& \qquad \qquad q_i^{n+1} = \xi \tilde{q}_i^{n+1} + (1-\xi) Q(\phi^{n+1}), \quad i=1,2,\cdots,k \Big\}, \quad \eta \in [0, 1].
\end{align}
\end{subequations}
\end{itemize}
\end{scheme}

\begin{rem}
Similarly, the optimal choice for the relaxation parameter can be calculated as $\xi_0 = \max\{0, \frac{-b -\sqrt{b^2-4ac}}{2a}\}$ with the coefficients given by
\begin{subequations}
\begin{align}
& a = \sum_{i=1}^{k}( {\tilde{q}_i}^{n+1} - Q_i(\phi^{n+1}))^2, \quad b =\sum_{i=1}^{k} 2\big({\tilde{q}_i}^{n+1}-Q_i(\phi^{n+1})\big)Q_i(\phi^{n+1}), \\
& c = \sum_{i=1}^{k} [Q_i(\phi^{n+1})]^2 - \sum_{i=1}^{k}({\tilde{q}_i}^{n+1})^2 - \delta t \eta  (\mu^{n+\frac{1}{2}}, \cG\mu^{n+\frac{1}{2}}).
\end{align}
\end{subequations}
\end{rem}

% \begin{rem}
% The Scheme \ref{scheme:RMSAV-CN} is second-order accurate in time. Since the relaxation step in \eqref{eq:RMSAV-CN-step2} doesn't not affect the order of accuracy. Notice ${\tilde{q}_i}^{n+1} = Q_i(\phi(\bx, t_{n+1})) + O(\delta t^2)$, $Q_i(\Phi^{n+1}) = Q_i(\phi(\bx, t_{n+1})) + O(\delta t^2)$, $i=1, ... ,k$. Hence 
% $$
% {q_i}^{n+1} =  \xi_0 {\tilde{q}_i}^{n+1} + (1-\xi_0) Q_i(\phi^{n+1})= Q_i(\phi(\bx, t_{n+1})) + O(\delta t^2).
% $$
% \end{rem}
% \begin{rem}
% As the same as relaxed-CN scheme,  the relaxation step. $\xi_0$ is a solution of the optimization problem,
% \beq
% \xi_0 = \min_{\xi \in [0, 1]} \xi, \quad \mbox{ s.t.  }  \sum_{i=1}^{k} ({q_i}^{n+1})^2 - \sum_{i=1}^{k}({\tilde{q}_i}^{n+1})^2 \leq \delta t \eta (\mu^{n+\frac{1}{2}}, \cG\mu^{n+\frac{1}{2}}).
% \eeq 
% This can be simplified as
% \beq
% \xi_0 = \min_{\xi \in [0, 1]} \xi, \quad \mbox{ s.t.  }  a\xi^2 + b\xi + c \leq 0,
% \eeq 
% where the coefficients are
% \begin{subequations}
% \begin{align}
% & a = \sum_{i=1}^{k}( {\tilde{q}_i}^{n+1} - Q_i(\phi^{n+1}))^2, \quad b =\sum_{i=1}^{k} ({\tilde{q}_i}^{n+1} Q_i(\phi^{n+1}) -2 [Q_i(\phi^{n+1})]^2), \\
% & c = \sum_{i=1}^{k} ([Q_i(\phi^{n+1})]^2 - ({\tilde{q}_i}^{n+1})^2) - \delta t \eta  (\mu^{n+\frac{1}{2}}, \cG\mu^{n+\frac{1}{2}}).
% \end{align}
% \end{subequations}
% \end{rem}

\begin{theorem}
The Scheme \ref {scheme:RMSAV-CN} is unconditionally energy stable.
\end{theorem}

\begin{proof}
For the first step of \eqref{eq:RMSAV-CN_step1}, due to the Theorem \ref{Theo:MSAV-CN}, we could get
\beq 
({\phi}^{n+1},  (\cL+\sum_{i=1}^k \gamma_i I ){\phi}^{n+1})+\sum_{i=1}^{k}({\tilde{q}_i}^{n+1})^2-({\phi}^{n},  (\cL+\sum_{i=1}^k \gamma_i I ){\phi}^{n})-\sum_{i=1}^{k}({{q}_i}^{n})^2 =- {\delta t}  (\cG\mu^{n+\frac{1}{2}},\mu^{n+\frac{1}{2}}).\notag
\eeq 
From \eqref{eq:RMSAV-CN-step2}, we know
\beq
\sum_{i=1}^{k}({{q}_i}^{n+1})^2-\sum_{i=1}^{k}({\tilde{q}_i}^{n+1})^2\le  {\delta t} \eta  (\cG\mu^{n+\frac{1}{2}},\mu^{n+\frac{1}{2}}).
\eeq 
Adding above two equations, and using $1-\eta \ge 0$, we could arrive at
\[({\phi}^{n+1},  (\cL+\sum_{i=1}^k \gamma_i I ){\phi}^{n+1})+\sum_{i=1}^{k}({\tilde{q}_i}^{n+1})^2-({\phi}^{n},  (\cL+\sum_{i=1}^k \gamma_i I ){\phi}^{n})-\sum_{i=1}^{k}({{q}_i}^{n})^2\le- {\delta t}(1-\eta)  (\cG\mu^{n+\frac{1}{2}},\mu^{n+\frac{1}{2}})\le 0.\]
%\begin{subequations}
%\begin{align}
%& ({\phi}^{n+1},  (\cL+\sum_{i=1}^k \gamma_i I ){\phi}^{n+1})+\sum_{i=1}^{k}({\tilde{q}_i}^{n+1})^2-({\phi}^{n},  (\cL+\sum_{i=1}^k \gamma_i I ){\phi}^{n})-\sum_{i=1}^{k}({{q}_i}^{n})^2\notag\\
%& \qquad\qquad\qquad\qquad \qquad\qquad\qquad\qquad \qquad\qquad \le- {\delta t}(1-\eta)  (\cG\mu^{n+\frac{1}{2}},\mu^{n+\frac{1}{2}})\le 0.\notag 
%\end{align}
%\end{subequations}
This completes the proof.
\end{proof}

\section{Numerical results}
In this section, we implement the proposed numerical algorithms and apply them to several classical phase-field models that include the Allen-Cahn (AC) equation, the Cahn-Hilliard (CH) equation, the Molecular Beam Epitaxy (MBE) model, the phase-field crystal (PFC) model and the diblock copolymer model. 

For simplicity, we only consider the phase-field models with periodic boundary conditions. However, we emphasize that our proposed algorithms apply to other thermodynamically consistent boundary conditions that satisfy the energy dissipation laws. Given the periodic boundary conditions, we use the Fourier pseudo-spectral method for spatial discretization. Let $N_x,N_y$ be two positive even integers. The spatial domain $\Omega = [0,L_x]\times[0,L_y]$ is uniformly partitioned with mesh size $h_x = L_{x}/N_{x}, h_y = L_{y}/N_{y}$ and 
$$
\Omega_{h} =\left\{(x_{j},y_{k})|x_{j} = j h_x, y_{k} = kh_y,~0\leq j\leq
N_{x}-1,0\leq k\leq N_{y}-1\right\}.
$$
The details for spatial discretization are omitted. Interested readers can refer to our previous work \cite{ChenZhaoGongCICP2019}. Also, given the BDF2 and CN schemes are both second-order accurate, we only compare the baseline SAV-CN scheme and the RSAV-CN scheme in this paper.

\subsection{Allen-Cahn equation}
In the first example, we consider the Allen-Cahn equation. Consider the free energy
$\cE = \int_\Omega \frac{\varepsilon^2}{2}|\nabla \phi|^2 + \frac{1}{4} (\phi^2 -1)^2 d\bx$,  with mobility operator $\cG = 1$, the general gradient flow model in \eqref{eq:generic-model} reduces to the corresponding Allen-Cahn equation
\beq
\partial_t \phi = -\lambda (-\varepsilon^2 \Delta \phi + \phi^3 - \phi).
\eeq 
In the SAV formulation, we introduce the scalar auxiliary variable
$$
q(t):= Q(\phi(\bx,t)) = \sqrt{\int_\Omega\frac{1}{4} (\phi^2 - 1-\gamma_0)^2 d\bx + C}.
$$
Then the SAV reformulated equations read as
\begin{subequations} \label{eq:AC-SAV-2}
\begin{align}
&\partial_t \phi = -\lambda \Big[-\varepsilon^2 \Delta \phi + \gamma_0 \phi + \frac{q(t)}{Q(\phi)}V(\phi)\Big],  \quad V(\phi) = \phi(\phi^2 - 1-\gamma_0),\\
& \frac{d}{dt}q(t) = \int_\Omega \frac{ V(\phi)}{2Q(\phi)} \partial_t\phi d\bx.
\end{align}
\end{subequations}

We verify that the relaxed SAV-CN scheme is second-order accurate in time. Consider the domain $\Omega=[0, 1]^2$, and we pick the smooth initial condition 
\beq \label{eq:initial-condition}
\phi(x,y,t=0)=0.01 \cos(2\pi x) \cos(2\pi y),
\eeq
and set the model parameters: $\varepsilon=0.01$ and $\lambda=1$. To solve the AC equation in \eqref{eq:AC-SAV-2}, we use uniform meshes $N_x=N_y=128$, and numerical parameters $C_0=1$, $\eta =0.95$, and $\gamma_0=1$. 

Given the analytical solutions are unknown, we calculate the error as the difference between the numerical solutions using the current time step and the numerical solutions using the adjacent finer time step.  The  numerical errors in $L^2$ norm with various time steps are summarized in Figure \ref{fig:AC-mesh-refinement}. A second-order convergence for the numerical solutions of $\phi$ and $q$ are both observed.

\begin{figure}
\center
\subfigure[Time step refinement test for $\phi$]{\includegraphics[width=0.4\textwidth]{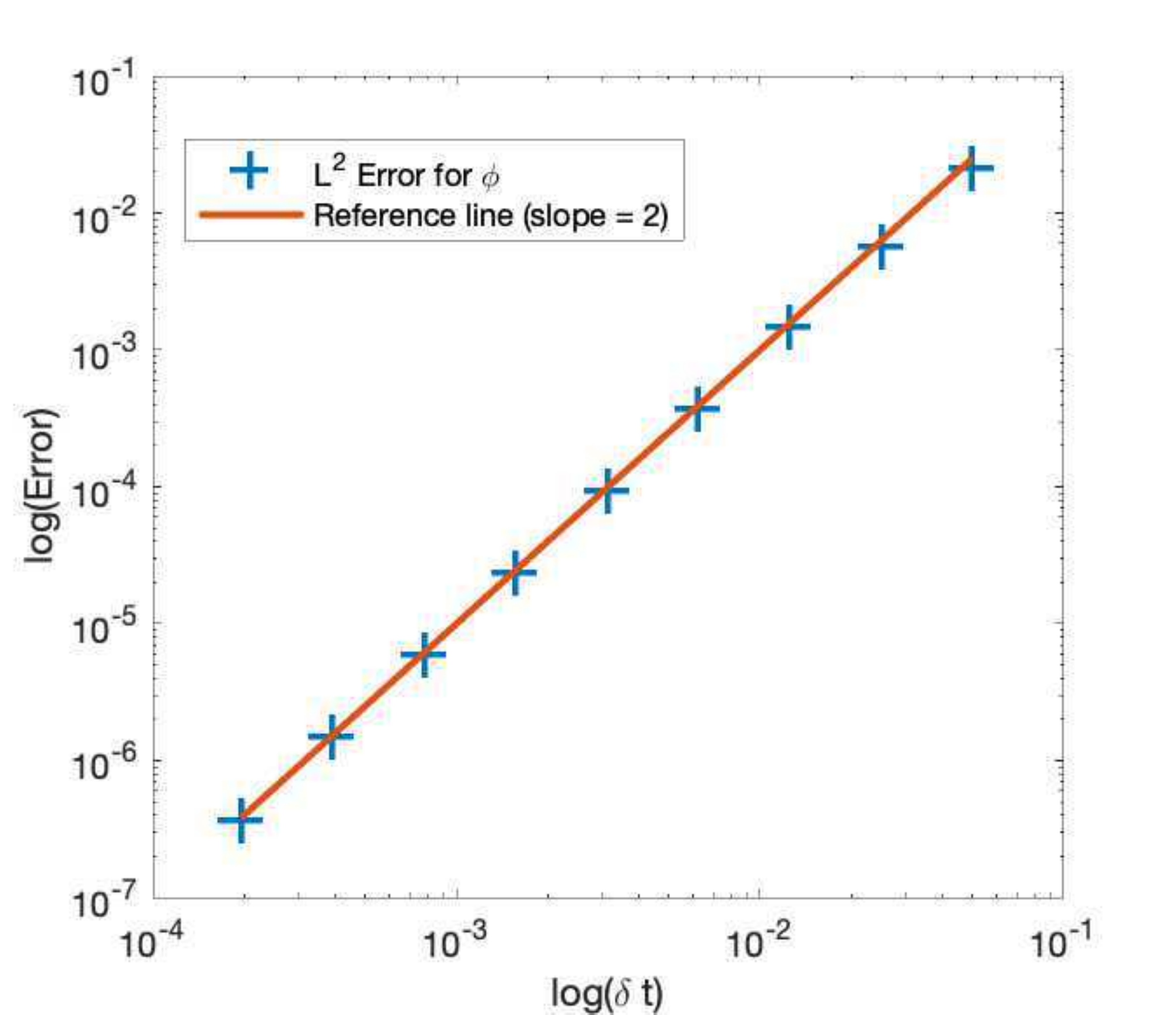}}
\subfigure[Time step refinement test for $q$]{\includegraphics[width=0.4\textwidth]{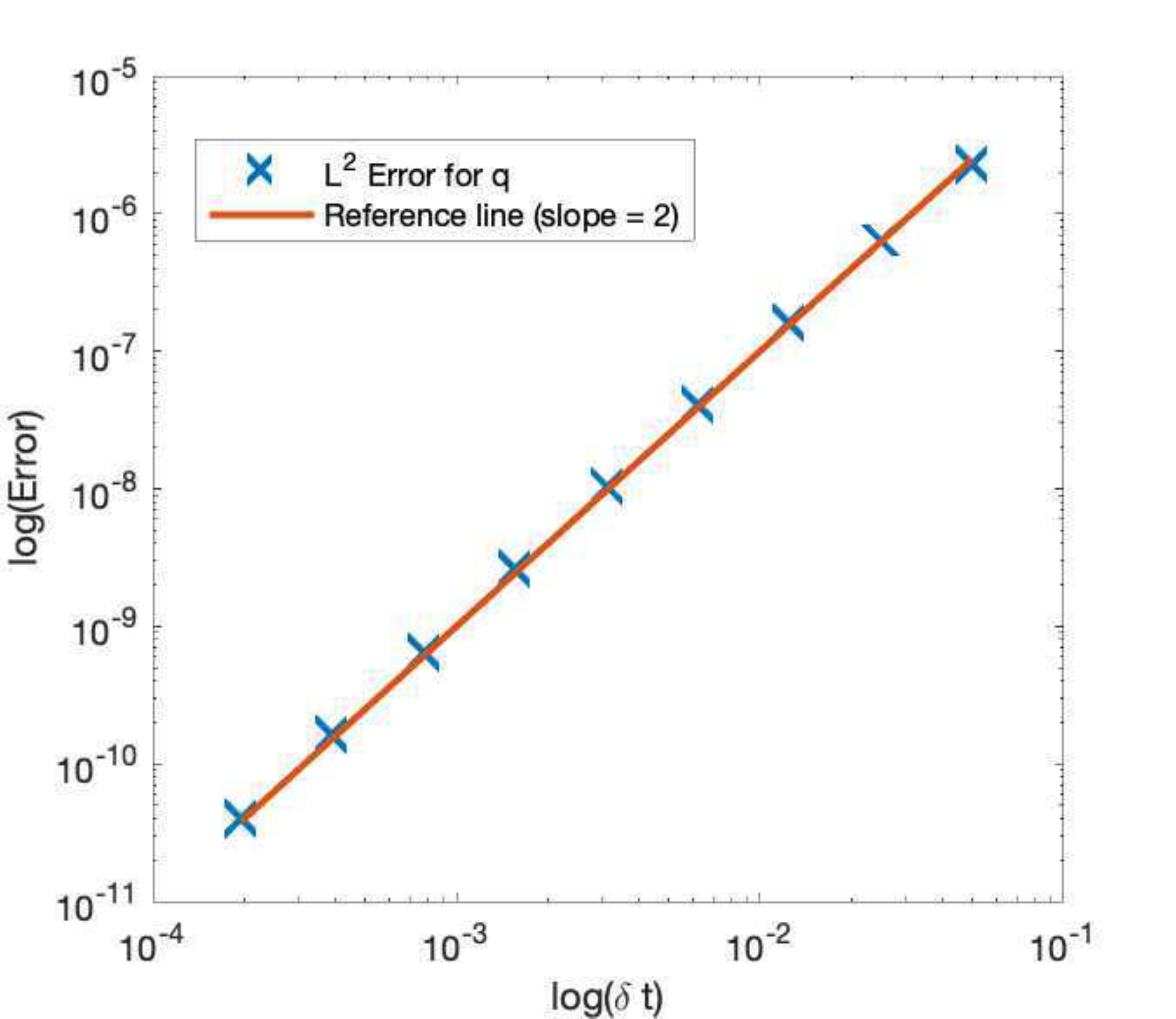}}
\caption{Time step mesh refinement tests of RSAV-CN scheme for solving the AC equation. This figure indicates that the proposed RSAV-CN algorithm is second-order accuracy in time when solving the AC equation.}
\label{fig:AC-mesh-refinement}
\end{figure}

Next, we conduct a detailed comparison between the baseline SAV-CN scheme and the RSAV-CN scheme. We use the same parameters with the example above. Consider the domain $\Omega=[0, L_x] \times[0, L_y]$,  and choose the initial condition as
\begin{subequations} \label{eq:initial-star}
\begin{align}
& \phi(x, y) = \tanh \frac{1.5 + 1.2 \cos(6\theta) - 2\pi r}{ \sqrt{2} \varepsilon}, \\
& \theta= \arctan \frac{y-0.5L_y}{x-0.5L_x}, \quad  r = \sqrt{ (x-\frac{L_x}{2})^2 + (y-\frac{L_y}{2})^2}.
\end{align}
\end{subequations}
In this example, we set $L_x=L_y=1$.
This initial condition represents a star shape in 2D, as shown in Figure \ref{fig:AC-star-example}(a). The comparison of calculated numerical energies for the AC equation is summarized in Figure \ref{fig:AC-compare}(a). We observe that the RSAV-CN method is more accurate than the baseline SAV-CN method, even with a larger time step. In addition, the numerical errors between $q^{n+1}$ and $Q(\phi^{n+1})$ are shown in Figure \ref{fig:AC-compare}(b). We observe that the RSAV method can effectively reduce the numerical errors between $q^{n+1}$ and $Q(\phi^{n+1})$. This is essential for preserving the consistency between the modified energy and the original energy after temporal discretization.

\begin{figure}
\center
\subfigure[Log-log plot for the energy evolution]{\includegraphics[width=0.4\textwidth]{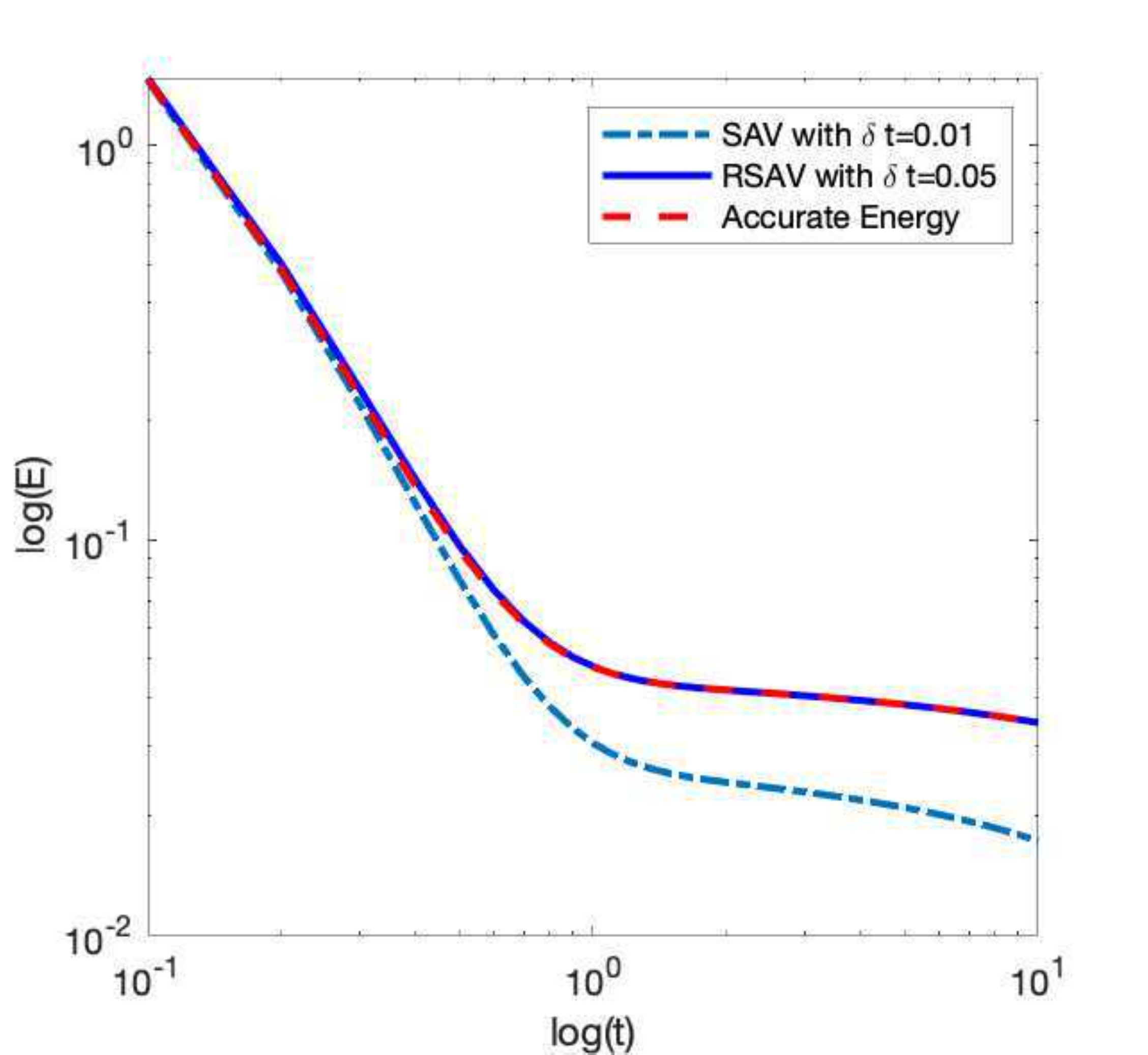}}
\subfigure[Numerical results of $q(t) - Q(\phi(\bx,t))$]{\includegraphics[width=0.4\textwidth]{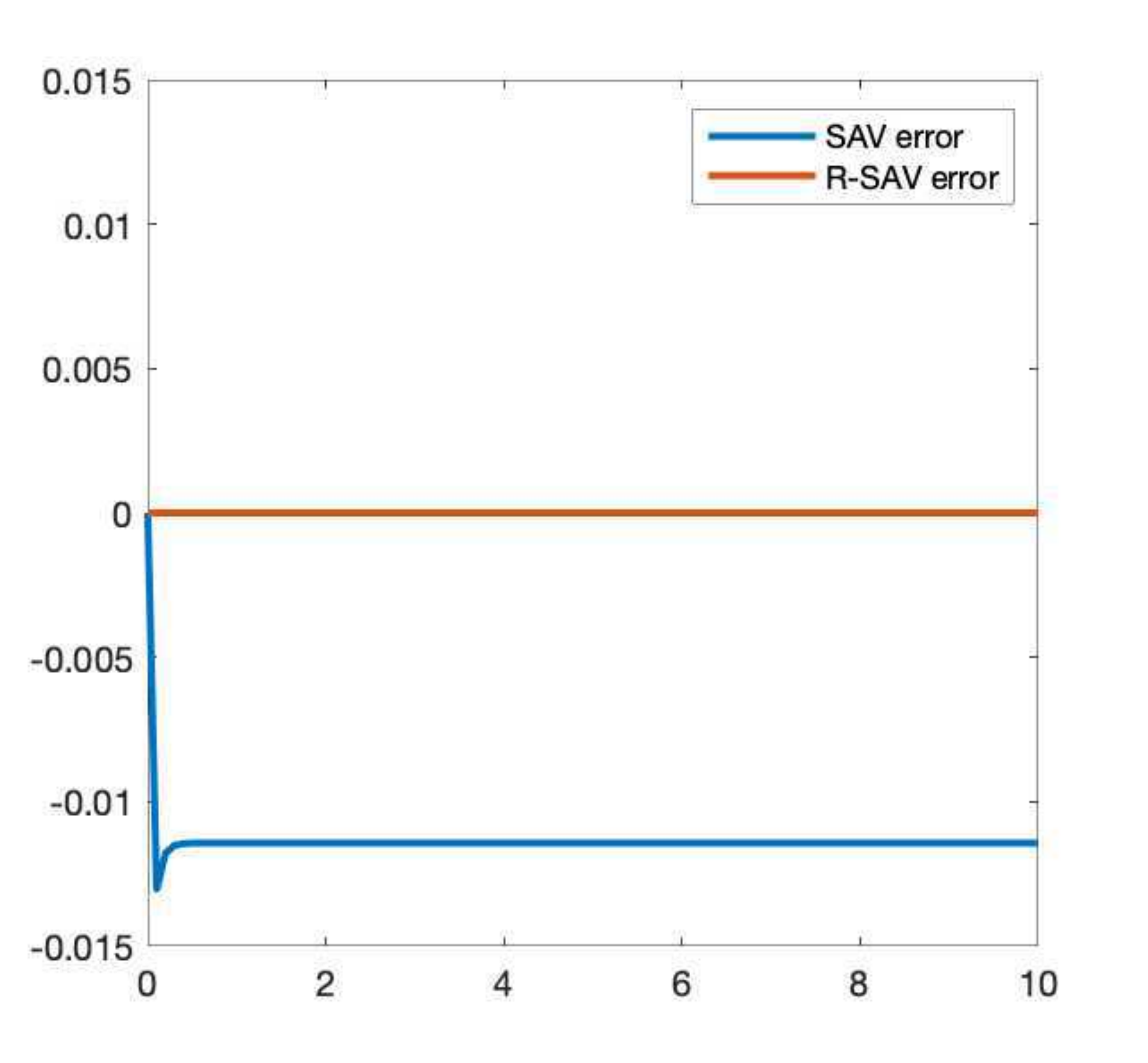}}
\caption{A comparison between the baseline SAV method and the relaxed SAV method in solving the Allen-Cahn equation. In (a) the numerical energies using the baseline SAV and the relaxed SAV are shown. The RSAV-CN scheme provides accurate result even with a much larger time step than the baseline SAV-CN scheme. In (b), the numerical results for $q(t)-Q(\phi(\bx,t))$ are shown, where we observe that the baseline SAV introduces numerical errors for $q(t)-Q(\phi(\bx,t))$, but the relaxed SAV method shows better consistency between $q^{n+1}$ and $Q(\phi^{n+1})$.}
\label{fig:AC-compare}
\end{figure}

Then long-time simulation is carried out using the RSAV-CN method with a time step $\delta t=0.01$.  The profiles of $\phi$ at various times are summarized in Figure \ref{fig:AC-star-example}. The star shape first relaxes to a round shape and keeps shrinking afterward. This is due to the Allen-Cahn model doesn't preserve the total volume of the phase variable. In the next section, we use the same initial condition to simulate dynamics driven by the Cahn-Hilliard equation, for which the total volume is preserved.

\begin{figure}
\center
\subfigure[Profiles of $\phi$ at $t=0, 10, 50$]{
\includegraphics[width=0.24\textwidth]{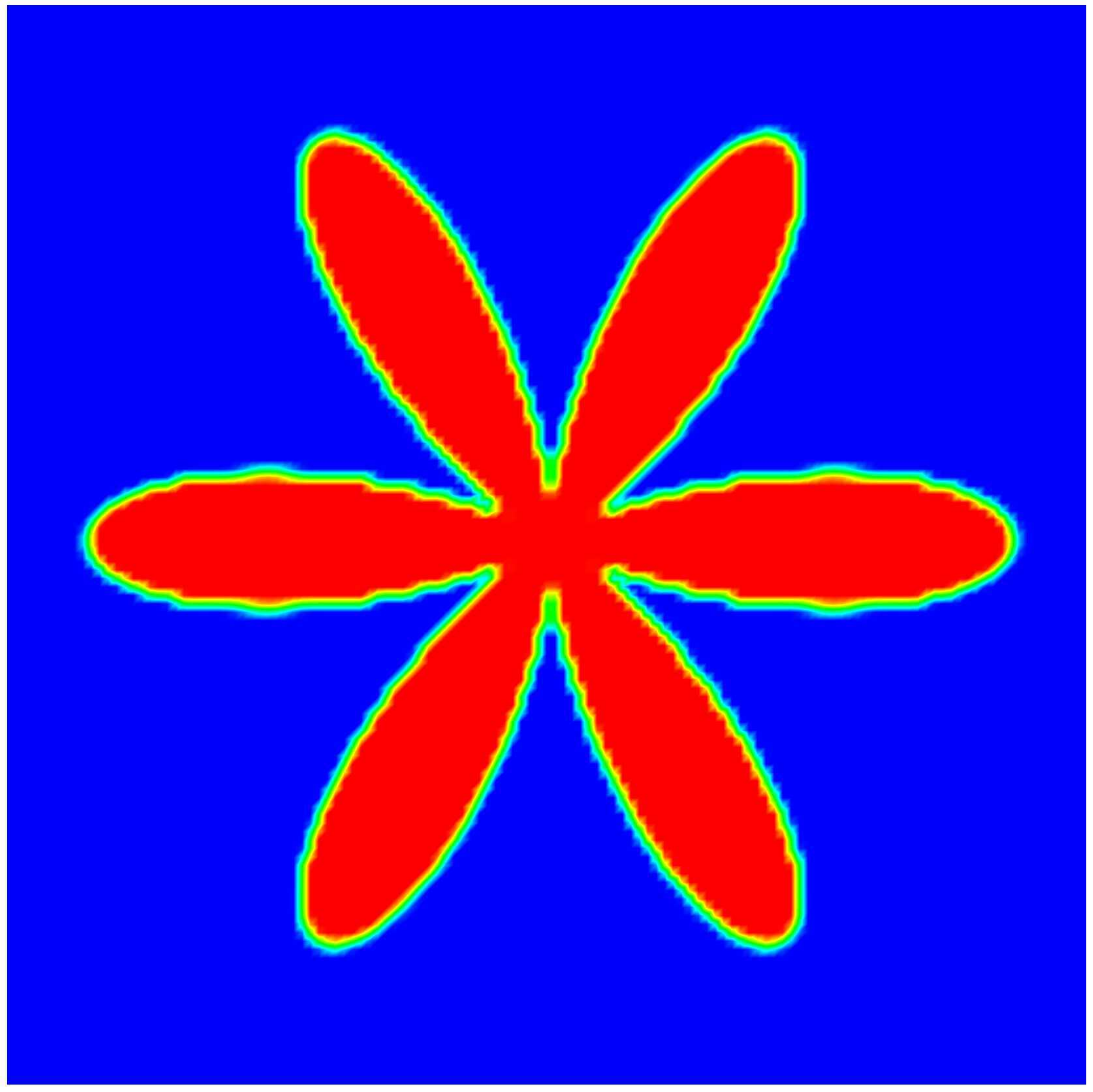}
\includegraphics[width=0.24\textwidth]{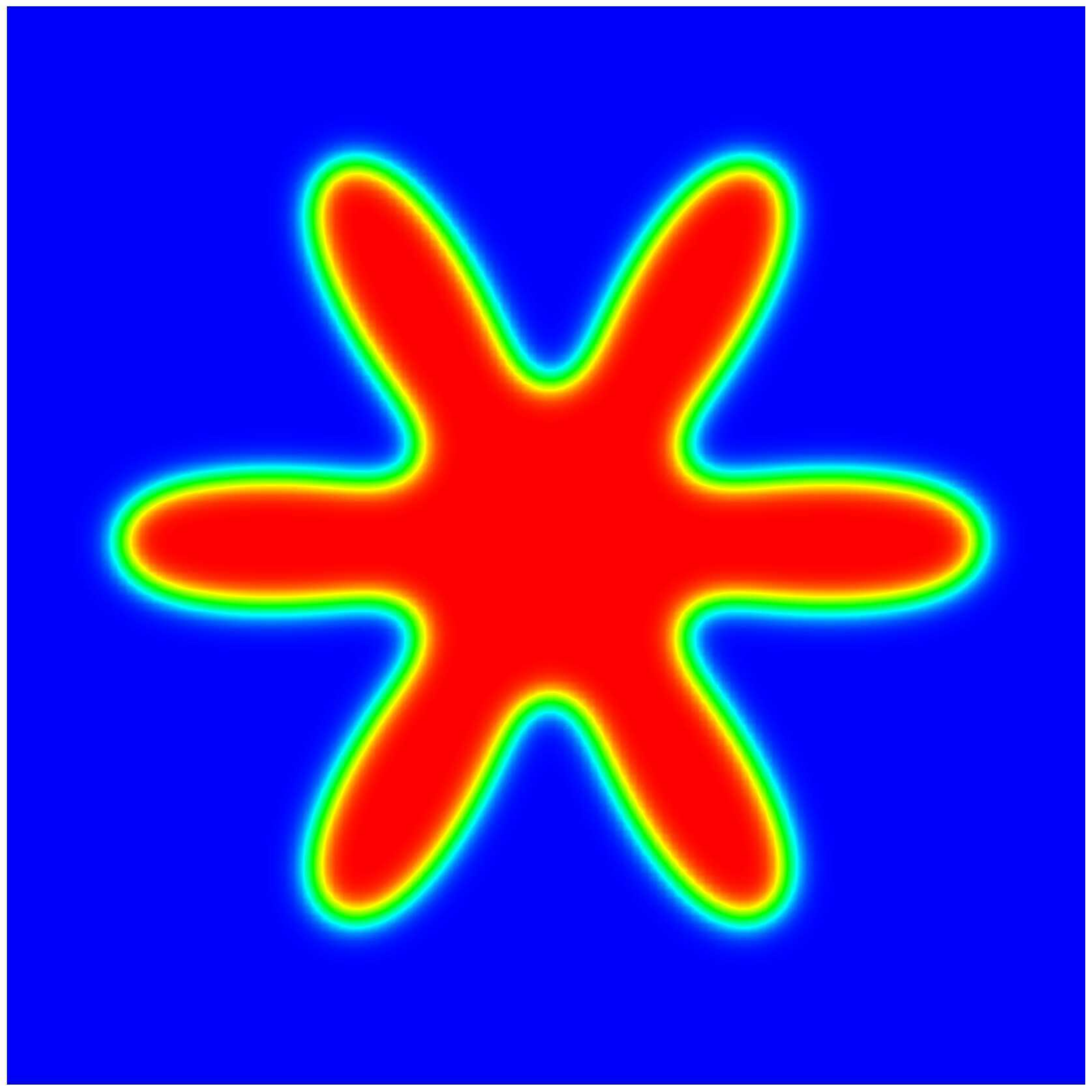}
\includegraphics[width=0.24\textwidth]{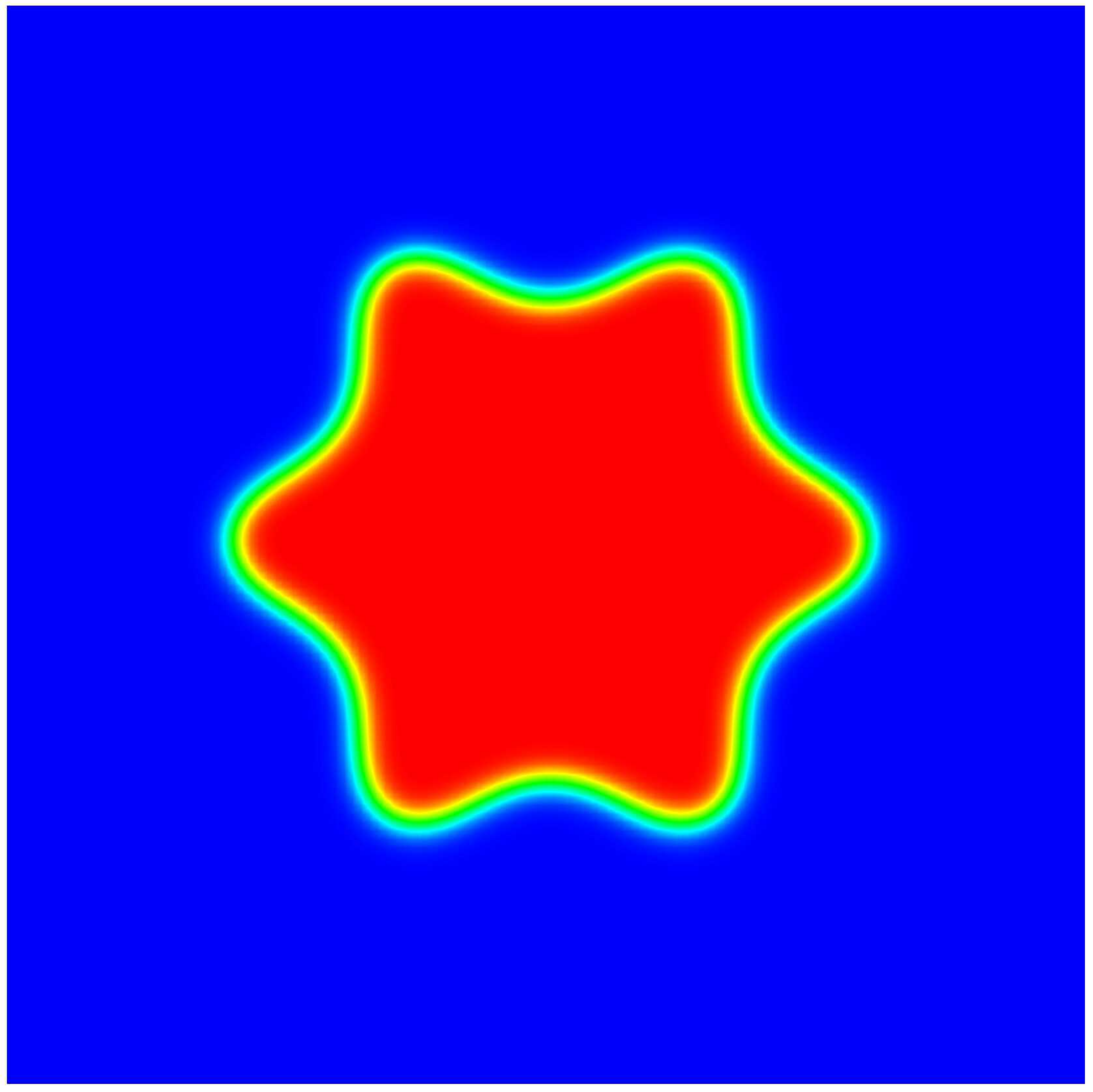}
}

\subfigure[Profiles of $\phi$ at $t=100, 150, 200$]{
\includegraphics[width=0.24\textwidth]{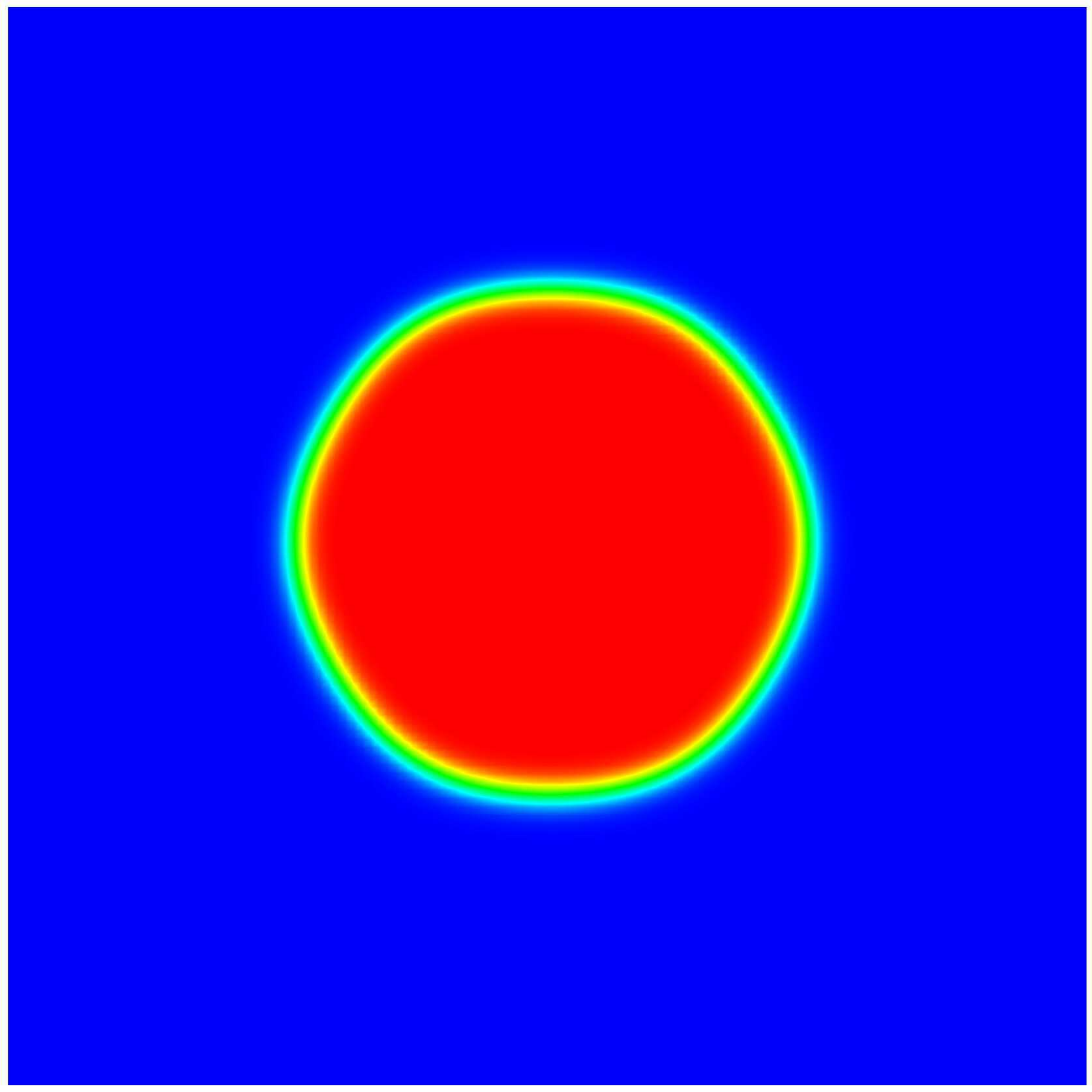}
\includegraphics[width=0.24\textwidth]{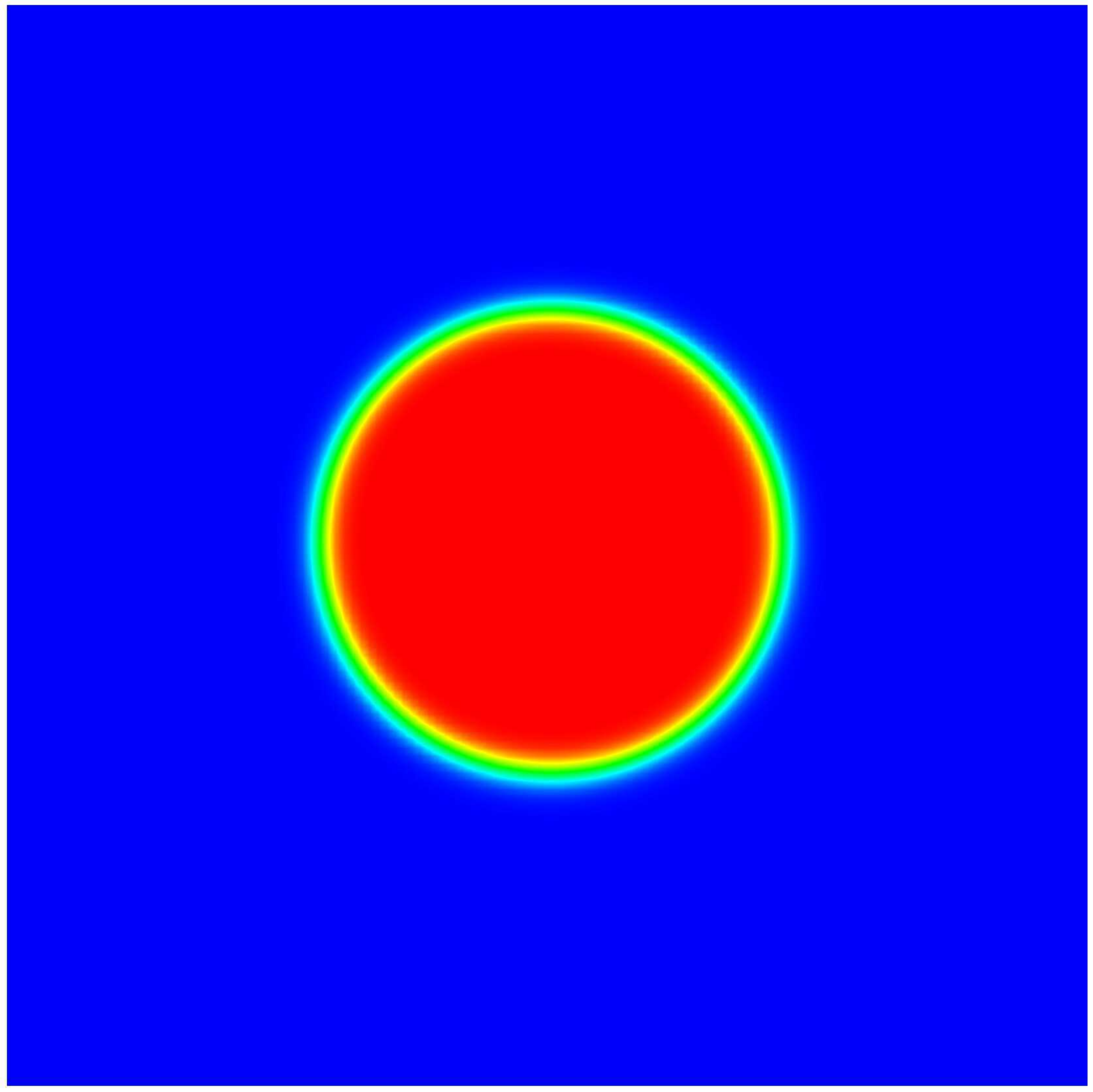}
\includegraphics[width=0.24\textwidth]{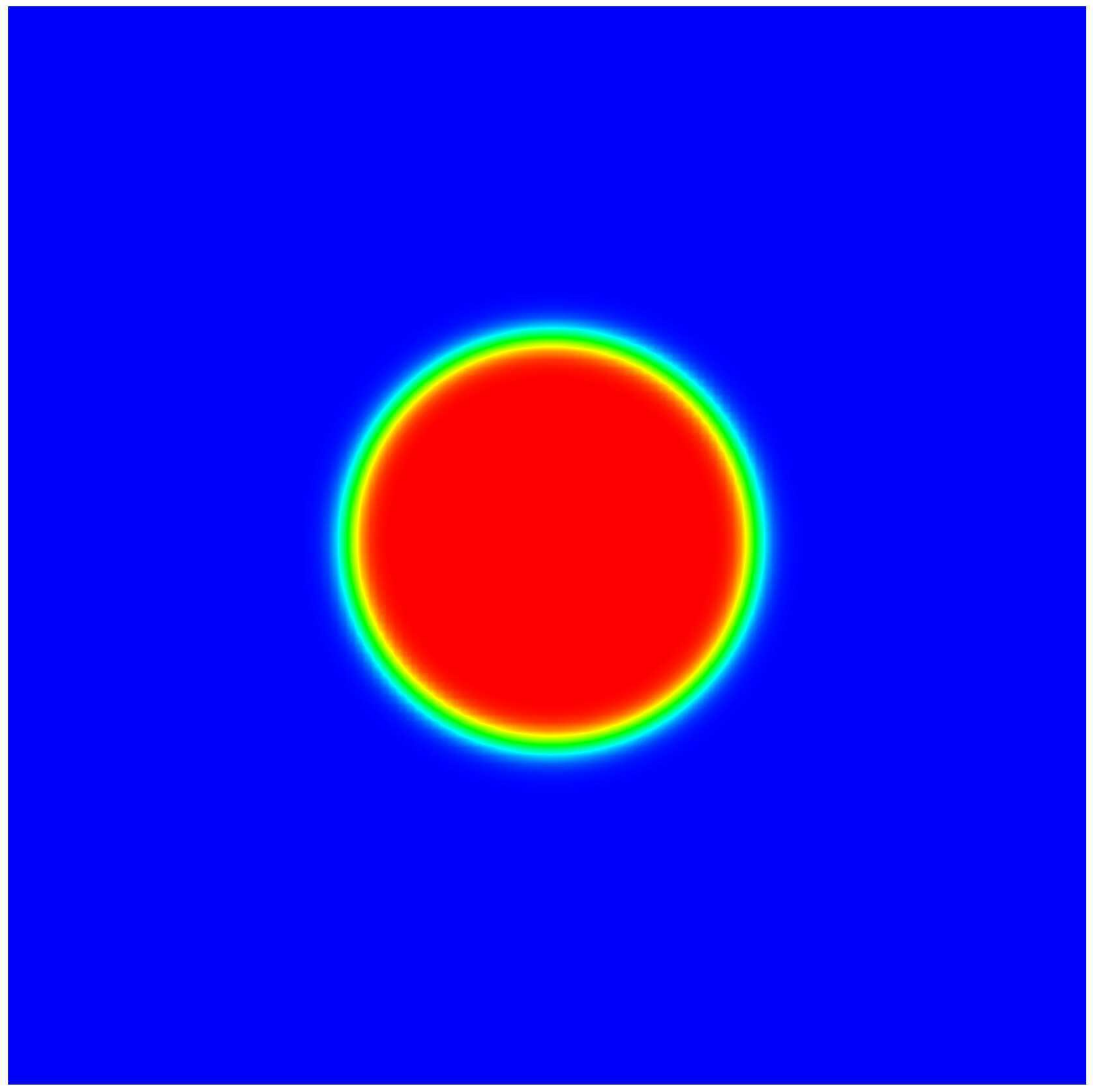}
}
\caption{Dynamics driven by the Allen-Cahn equation. The initial condition is a star-shaped profile in (a). And the profiles of $\phi$ at various times are shown.}
\label{fig:AC-star-example}
\end{figure}

\subsection{Cahn-Hilliard Equation}
In the second example, we consider the well-known Cahn-Hilliard equation with a double-well potential. Mainly, consider the free energy
$\cE = \int_\Omega \frac{\varepsilon^2}{2}|\nabla \phi|^2 + \frac{1}{4} (\phi^2 -1)^2 d\bx$, and mobility $\cG = -\lambda \Delta$. The general gradient flow model in \eqref{eq:generic-model} reduces to the corresponding  Cahn-Hilliard equation
\begin{subequations}
\begin{align}
&\partial_t \phi = \lambda \Delta \mu, \\
&\mu = - \varepsilon^2 \Delta \phi + \phi^3 - \phi.
\end{align}
\end{subequations}

In the SAV formulation, we introduce the scalar auxiliary variable
$$
q(t) : =Q(\phi(\bx, t))= \sqrt{\int_\Omega\frac{1}{4} (\phi^2 - 1-\gamma_0)^2 d\bx + C}.
$$
Then the reformulated equations are obtained as
\begin{subequations}
\begin{align}
&\partial_t \phi =\lambda \Delta \mu, \\
& \mu = - \varepsilon^2 \Delta \phi + \gamma_0 \phi + \frac{q(t)}{Q(\phi)}V(\phi), \quad V(\phi) = \phi(\phi^2-1-\gamma_0), \\
& \frac{d}{dt}q(t) = \int_\Omega \frac{ V(\phi)}{2Q(\phi)} \partial_t\phi d\bx.
\end{align}
\end{subequations}

First of all, we conduct the mesh refinement test to check the order of temporal convergence. We use the same initial condition as the AC case in \eqref{eq:initial-condition} for the time mesh refinement tests. We consider the domain $\Omega=[0, 1]^2$ and model parameters $\lambda=0.01$, $\varepsilon=0.01$. To solve the problem, we choose the numerical parameters $C_0=1$, $\gamma_0 =4$, $\eta = 0.95$, and $N_x=N_y=128$. Then we calculate the numerical solutions to $t=0.5$ with various time steps. The $L^2$ errors for numerical solutions (using strategies explained in the AC case) are calculated. The results are summarized in Figure \ref{fig:CH-mesh}. A second-order temporal convergence for both $\phi$ and $q$ is observed.

\begin{figure}
\center
\subfigure[Temporal mesh refinement test for $\phi$]{\includegraphics[width=0.4\textwidth]{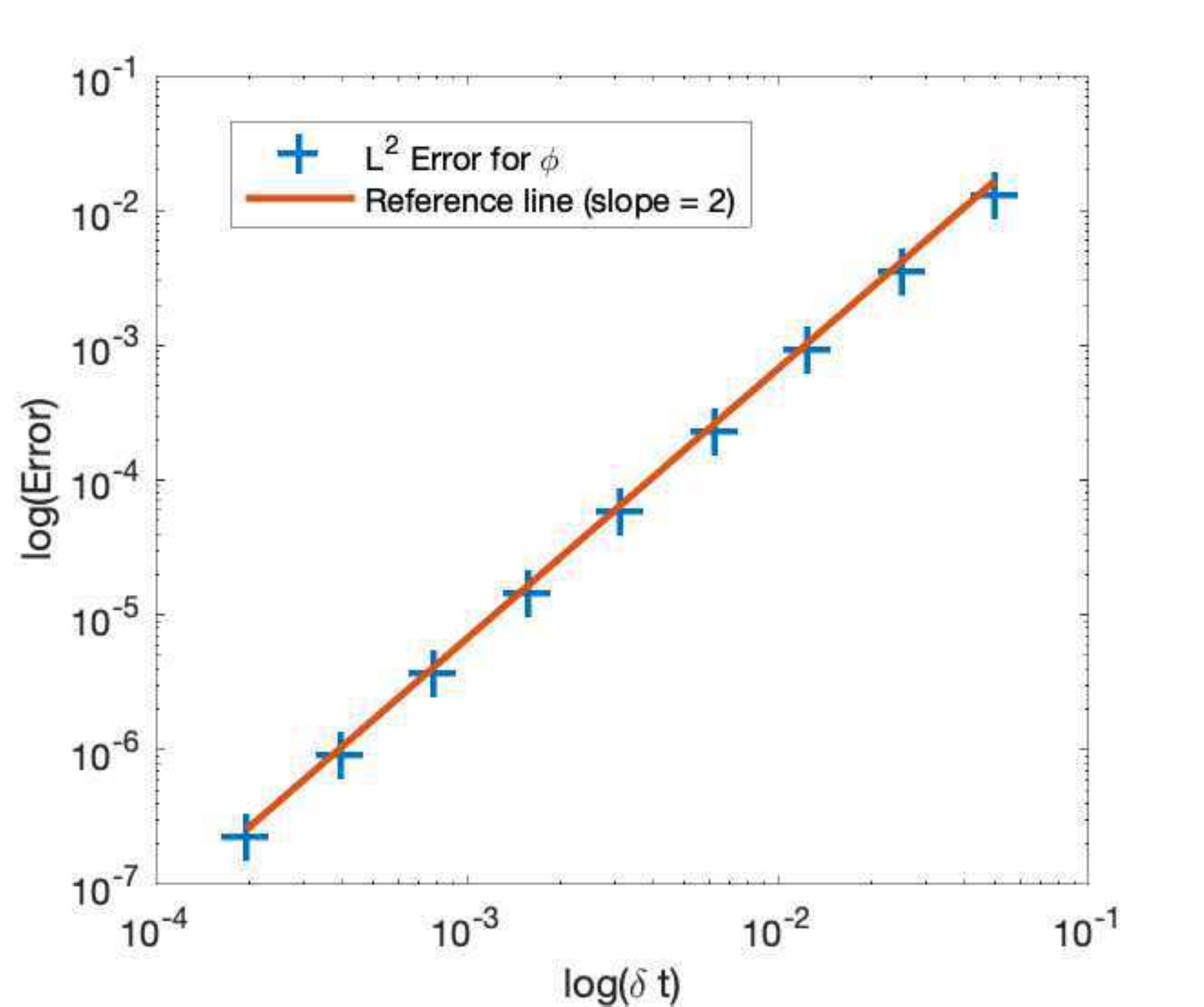}}
\subfigure[Temporal mesh refinement test for $q$]{\includegraphics[width=0.4\textwidth]{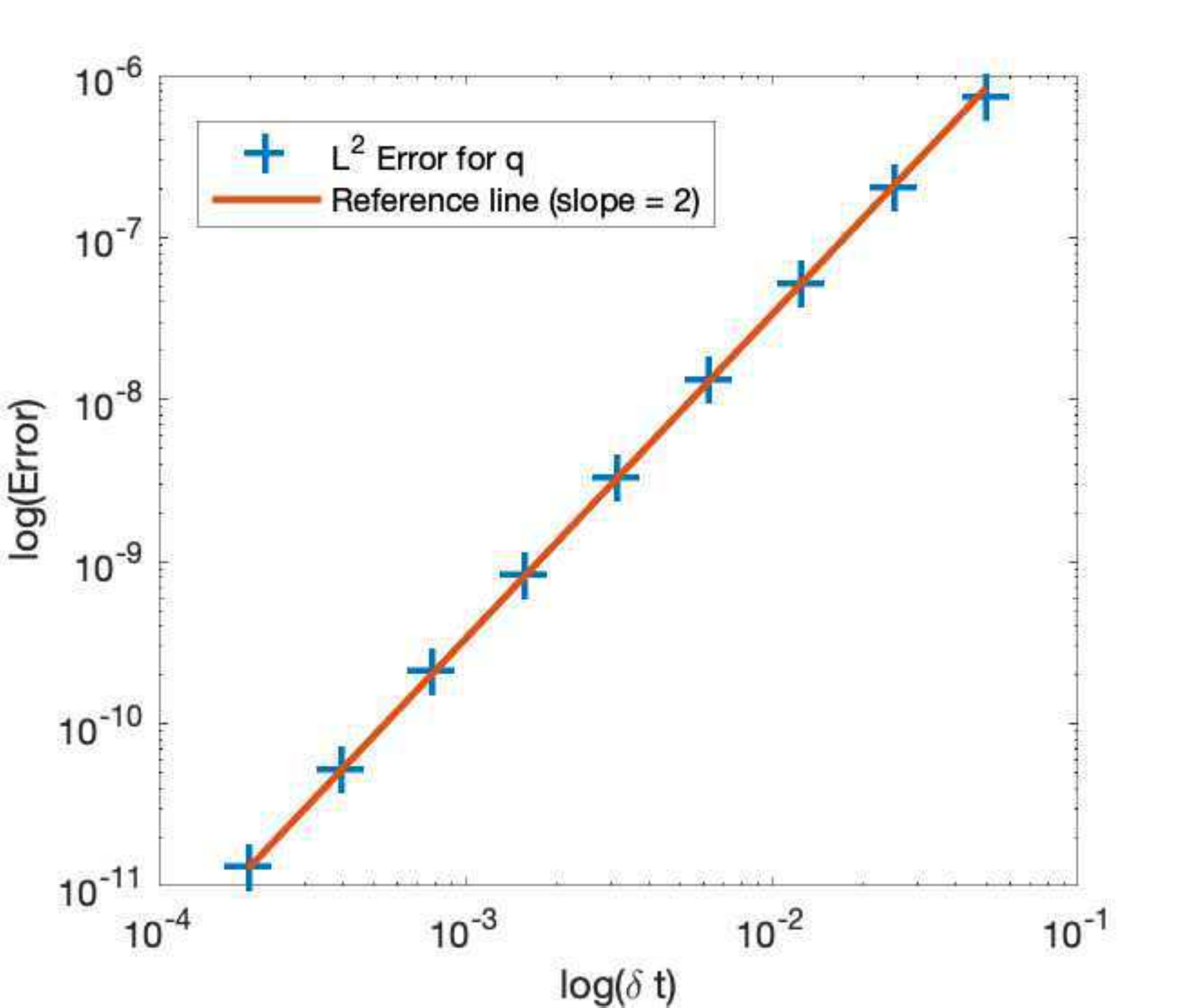}}
\caption{Time step mesh refinement tests of RSAV-CN method for solving the CH equation. This figure indicates that the proposed RSAV-CN algorithm is second-order accuracy in time when solving the CH equation.}
\label{fig:CH-mesh}
\end{figure}

After we verify that the RSAV-CN scheme is second-order accurate, we compare the accuracy of the baseline SAV-CN scheme and the RSAV-CN scheme for solving the Cahn-Hilliard equation. We use the same initial conditions as in \eqref{eq:initial-star}.  The model parameters used are $\lambda=0.1$, $\varepsilon=0.01$, $\gamma_0=4$. And to solve the problem, we choose the numerical parameters $C_0=1$, $\eta=0.95$, and $N_x=N_y=128$. The comparison of calculated energies are shown in Figure \ref{fig:CH-compare}(a), and the numerical errors for $q(t) - Q(\phi(\bx, t))$) using both the baseline SAV-CN scheme and the relaxed SAV-CN schemes are shown in Figure \ref{fig:CH-compare}(b). We observe that the relaxation step improves the accuracy significantly. In addition, the relaxation guarantees the consistency of numerical solution $q^{n+1}$ with its original definition $Q(\phi^{n+1})$, which indicates the numerical consistency of modified energy and the original energy.

\begin{figure}
\center
\subfigure[Log-log plot for the energy evolution]{\includegraphics[width=0.4\textwidth]{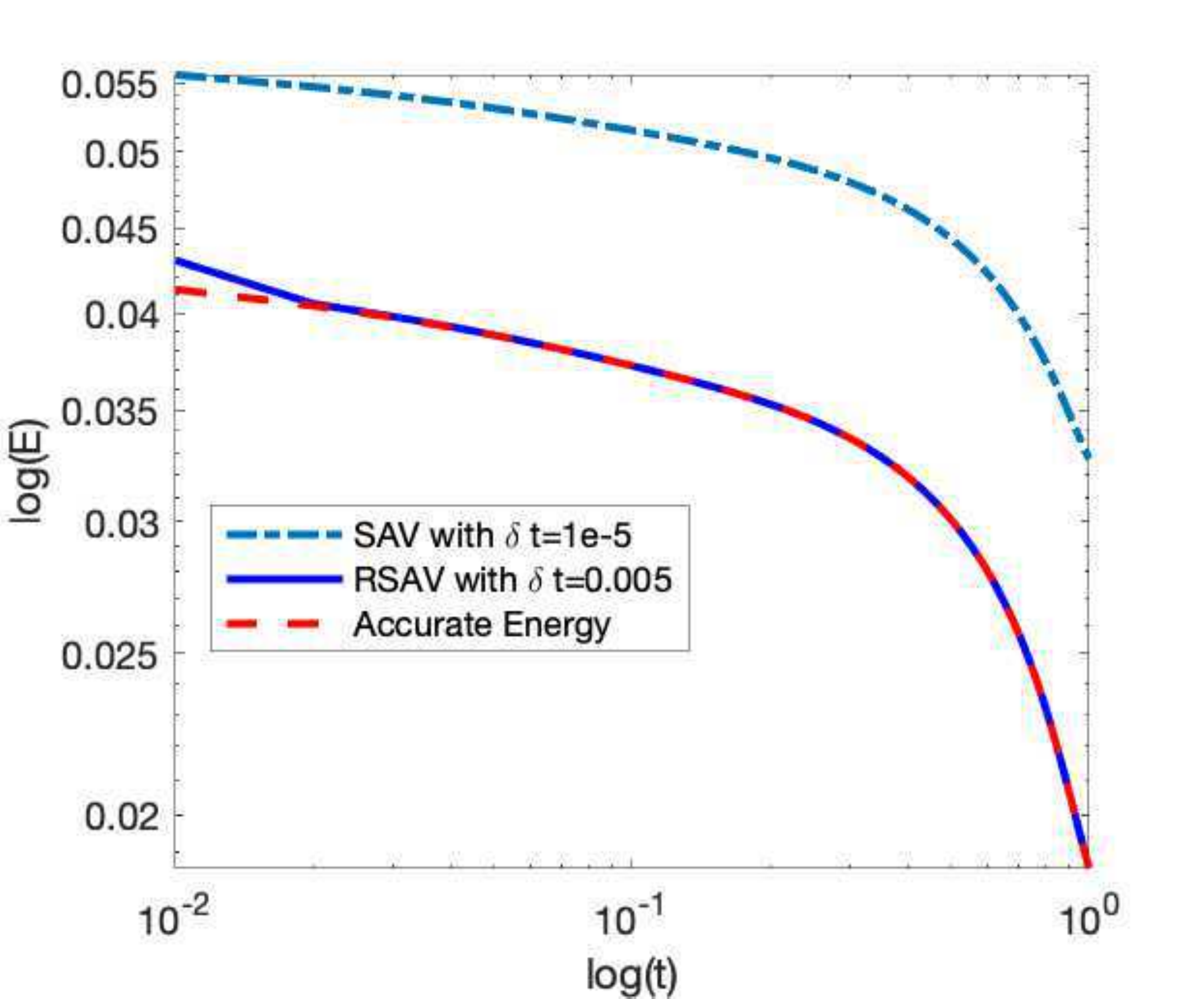}}
\subfigure[Numerical results of $q(t) - Q(\phi(\bx,t))$]{\includegraphics[width=0.4\textwidth]{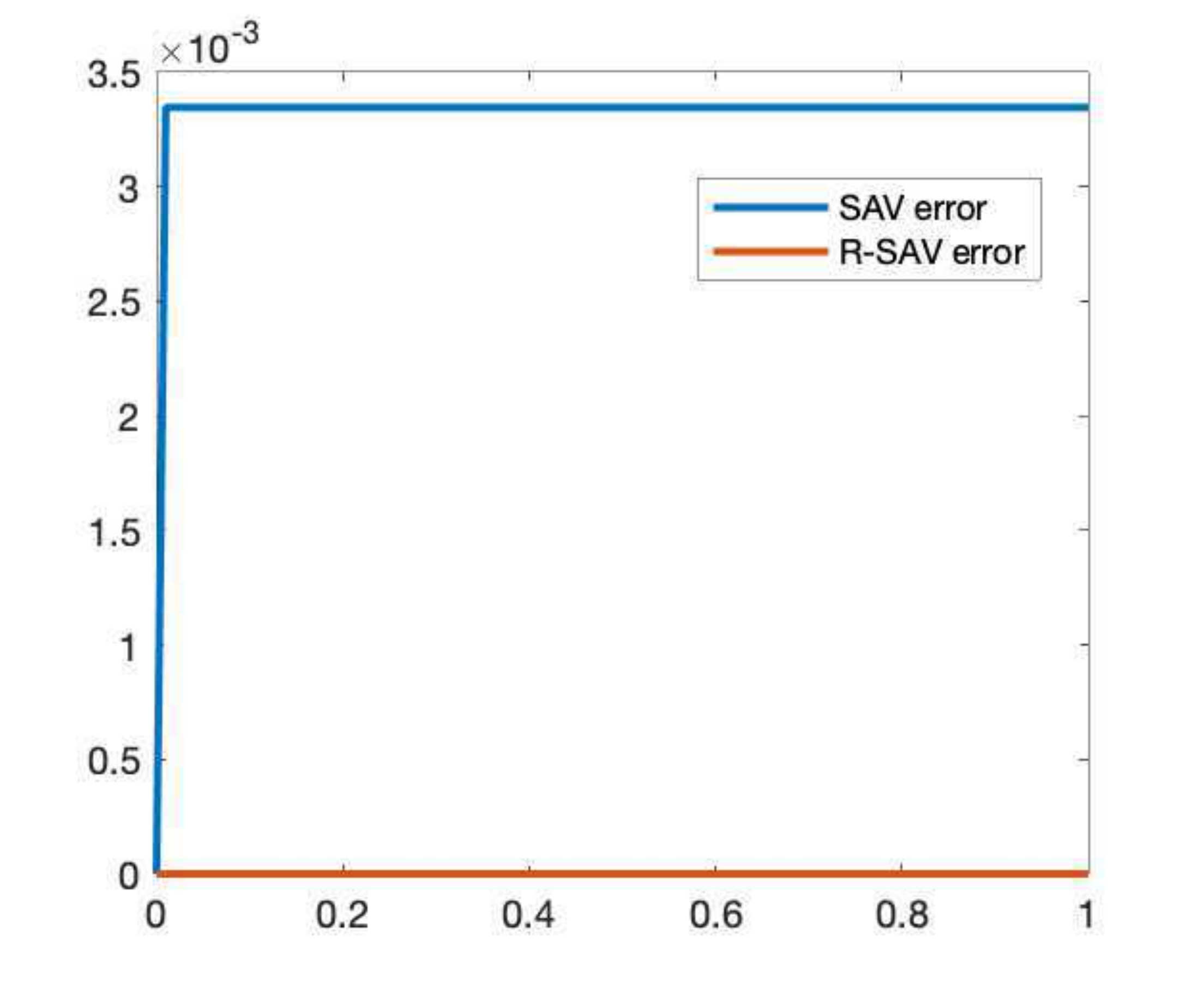}}
\caption{A comparison between the baseline SAV method and the relaxed SAV method for solving the Cahn-Hilliard model. In (a) the numerical energies using the baseline SAV and the relaxed SAV are shown. The RSAV method provides accurate result even with larger time step than the baseline SAV method. In (b), the numerical results for $q(t)-Q(\phi(\bx,t))$ are shown, where we observe that the baseline SAV introduces numerical errors for $q(t)-Q(\phi(\bx,t))$, but the relaxed SAV has properly relaxed the error close to 0.}
\label{fig:CH-compare}
\end{figure}

The profiles of $\phi$ at various time slots using the RSAV-CN scheme with a time step $\delta t=0.001$ are summarized in Figure \ref{fig:CH-example}. The initial profile of $\phi$ is a star shape shown in (a), and it smooths out into a disk.

\begin{figure}
\center
\subfigure[profiles of $\phi$ at $t=0,0.1, 0.5$]{
\includegraphics[width=0.24\textwidth]{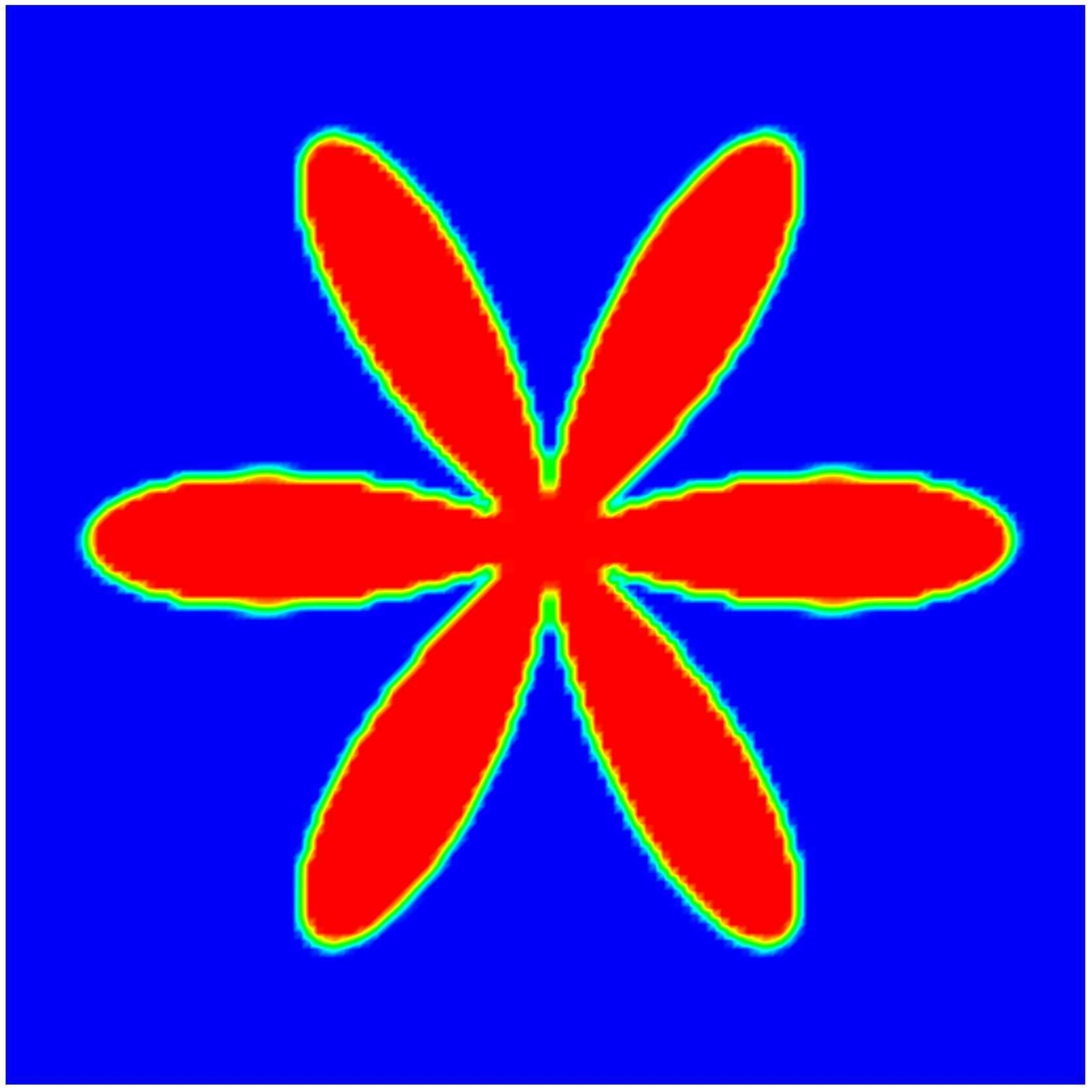}
\includegraphics[width=0.24\textwidth]{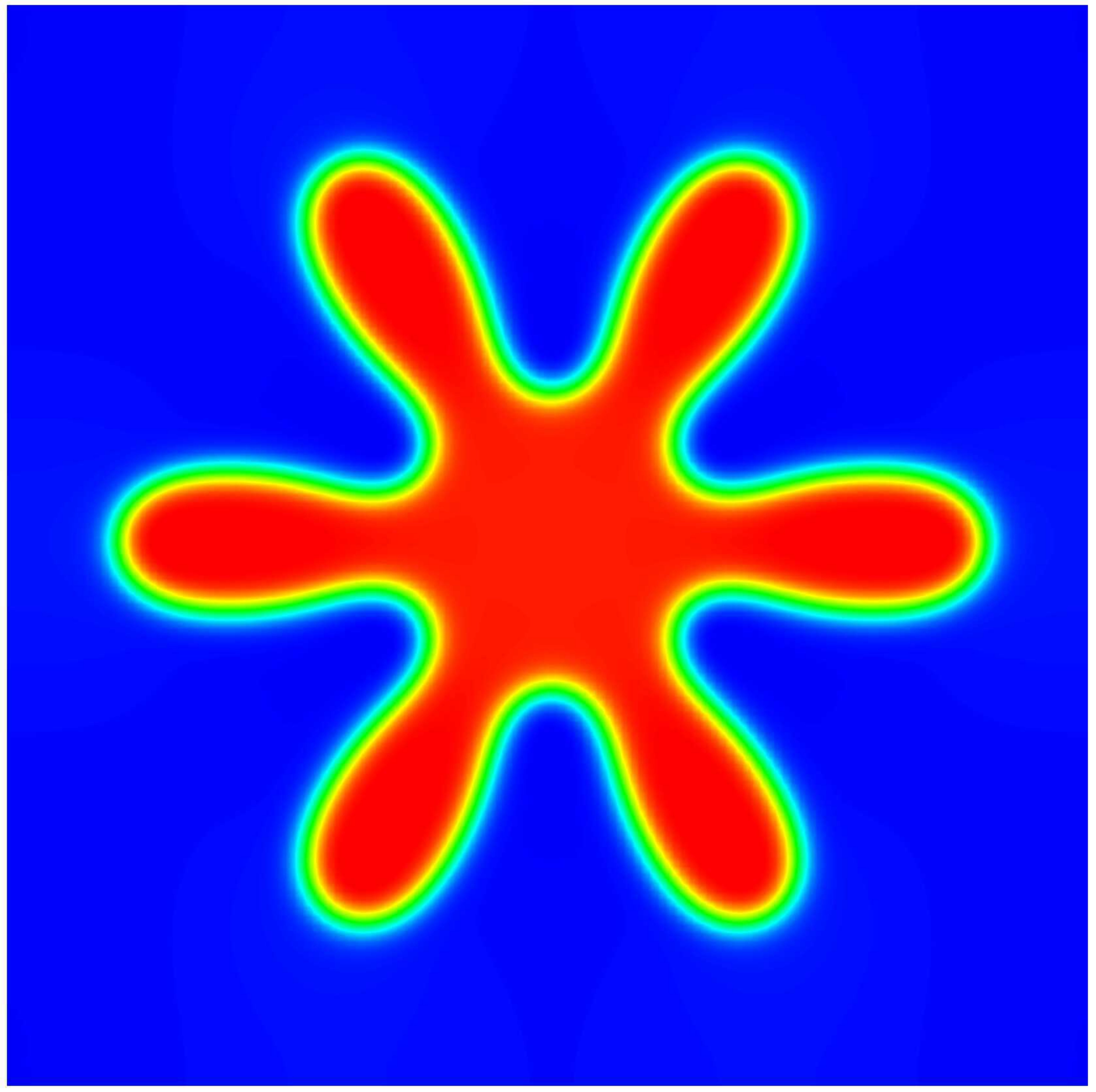}
\includegraphics[width=0.24\textwidth]{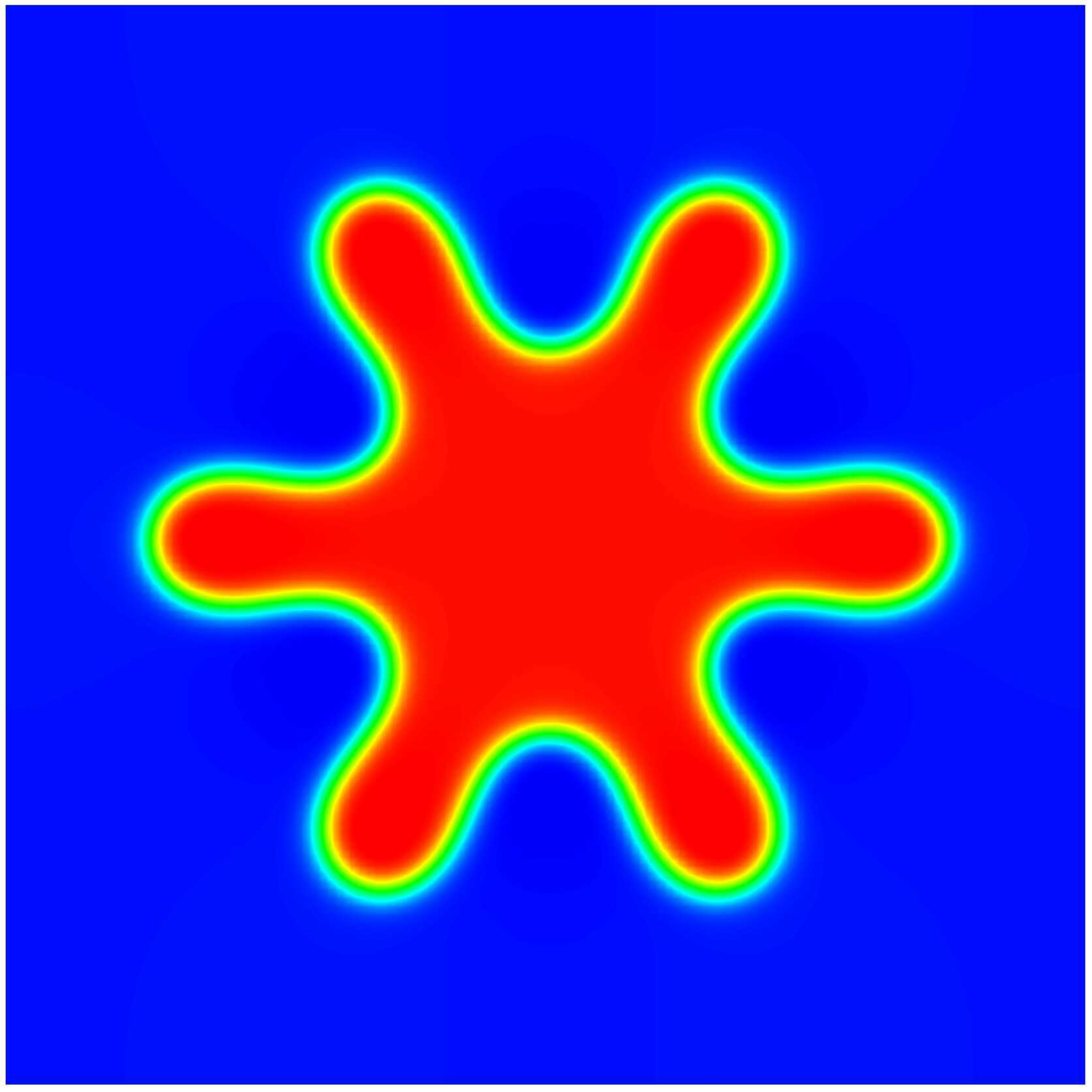}
}

\subfigure[profiles of $\phi$ at $t=1,1.2, 3$]{
\includegraphics[width=0.24\textwidth]{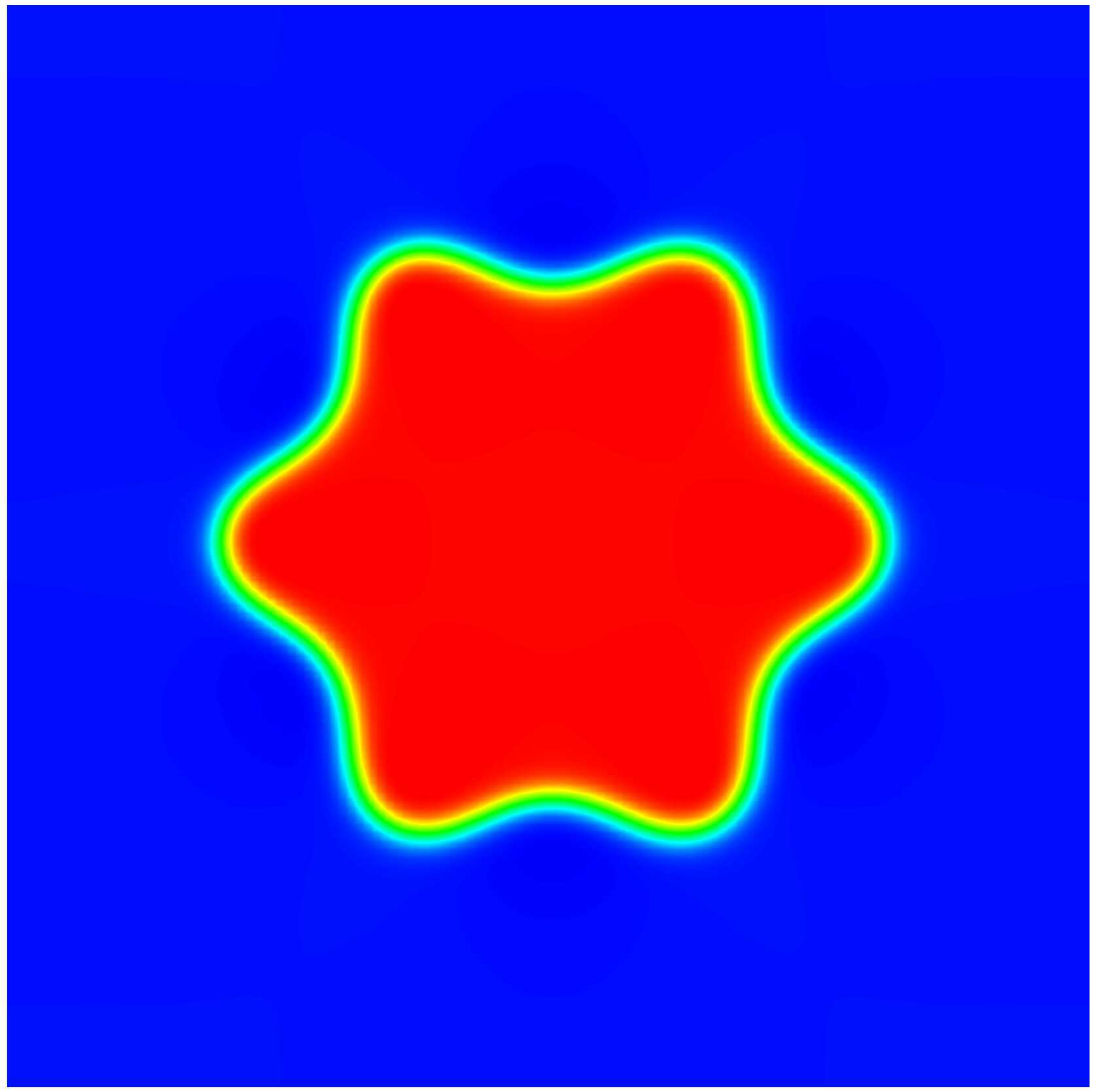}
\includegraphics[width=0.24\textwidth]{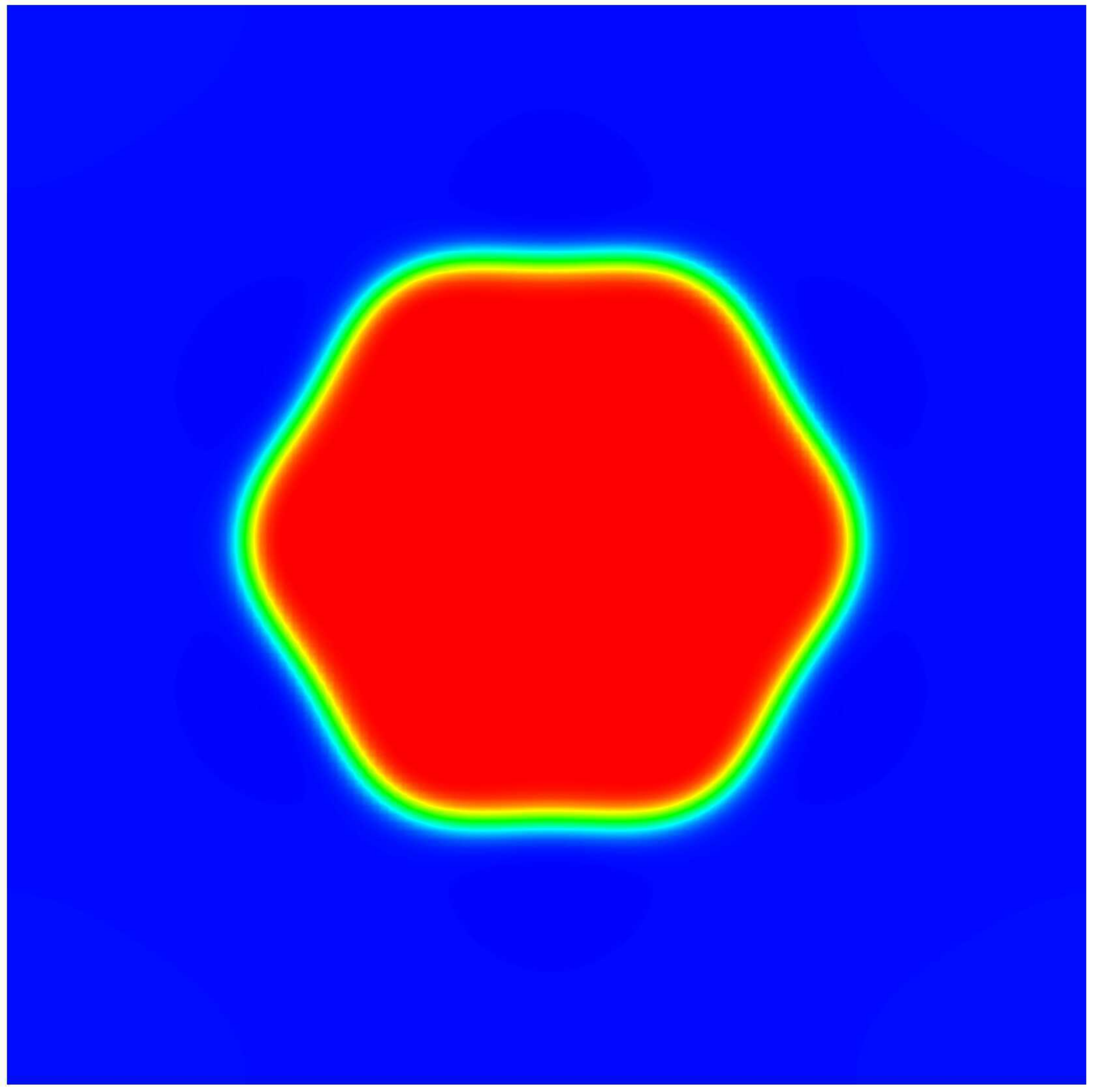}
\includegraphics[width=0.24\textwidth]{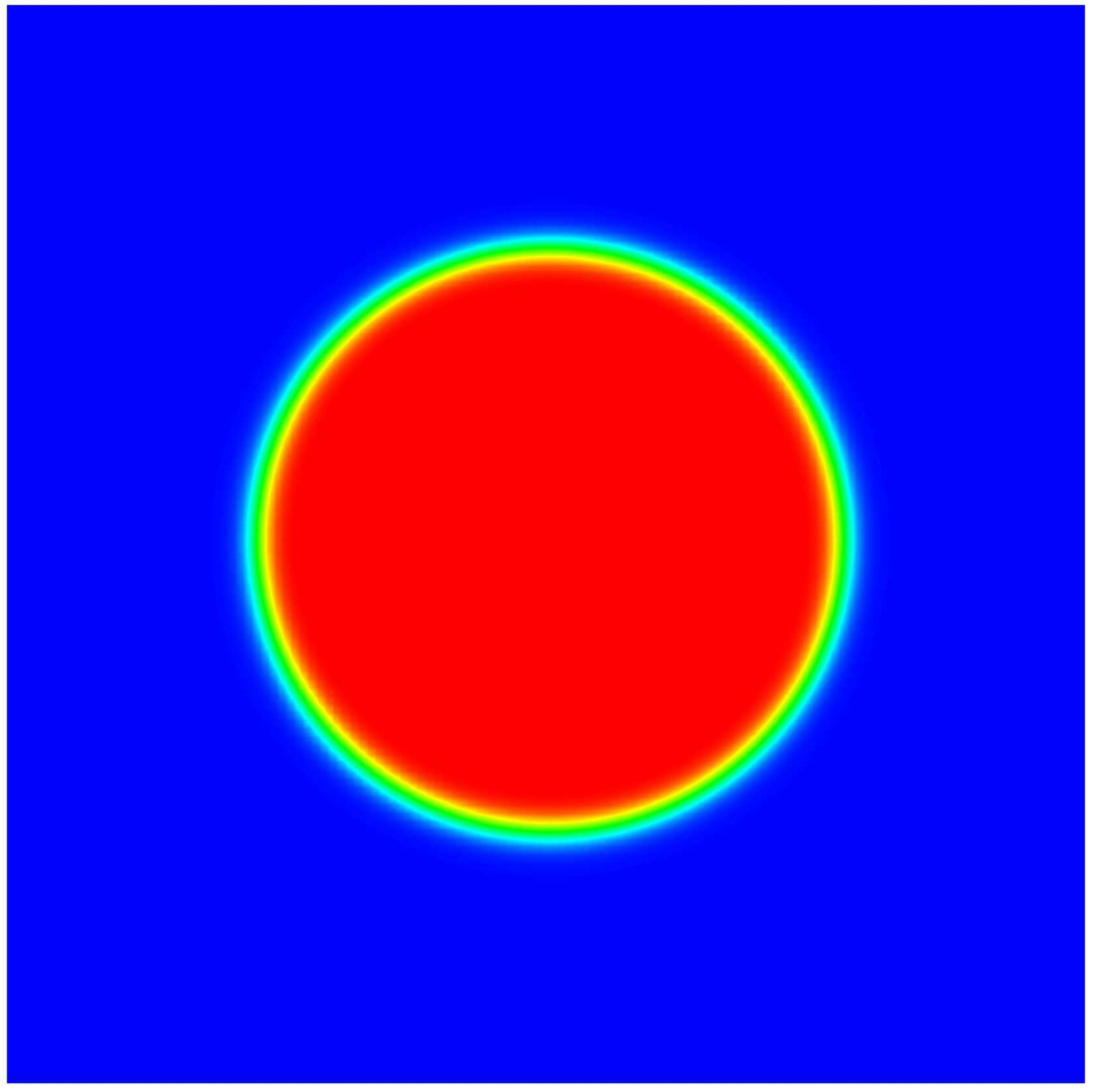}
}
\caption{Dynamics driven by the Cahn-Hilliard equation. The profiles of $\phi$ at various times are shown. }
\label{fig:CH-example}
\end{figure}

Next, we investigate the coarsening dynamics driven by the Cahn-Hilliard equation. We consider the domain $\Omega=[0, 4]^2$, and set $\lambda=0.1$, $\varepsilon=0.01$. Set the initial condition as 
$\phi(x, y, t=0) = \hat{\phi}_0 + 0.05 rand(x, y)$,
where $rand(x,y)$ generates rand numbers between $-1$ and 1, and $\hat{\phi}_0$ is a constant. We use the relaxed SAV-CN scheme to solve it with meshes $N_x=N_y=512$, model parameters $\gamma_0=4$, $C_0=1$, $\eta = 0.95$ and time step $\delta t= 0.001$. The results are summarized in Figure  \ref{fig:CH-coarsening}. It is observed that when the volume difference between two phases are small, saying in Figure \ref{fig:CH-coarsening}(a), the spinodal decomposition dynamics takes place; when the volume difference between two phases are larger, saying in Figure \ref{fig:CH-coarsening}(b), the nucleation dynamics takes place.

\begin{figure}
\center
\subfigure[$\hat{\phi}_0 =0$, $\phi$ at $t=0.1, 0.5, 5, 50$]{
\includegraphics[width=0.242\textwidth]{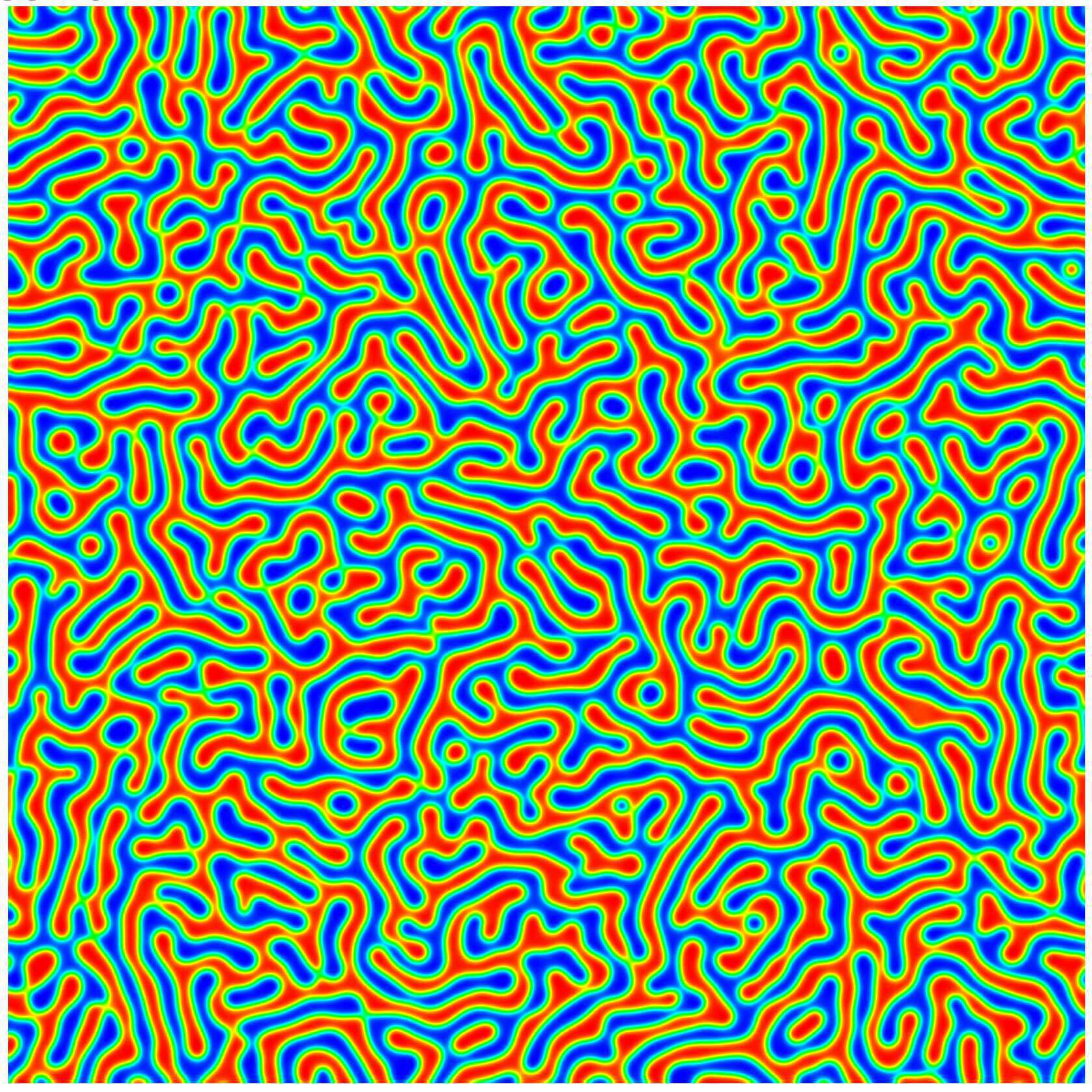}
\includegraphics[width=0.242\textwidth]{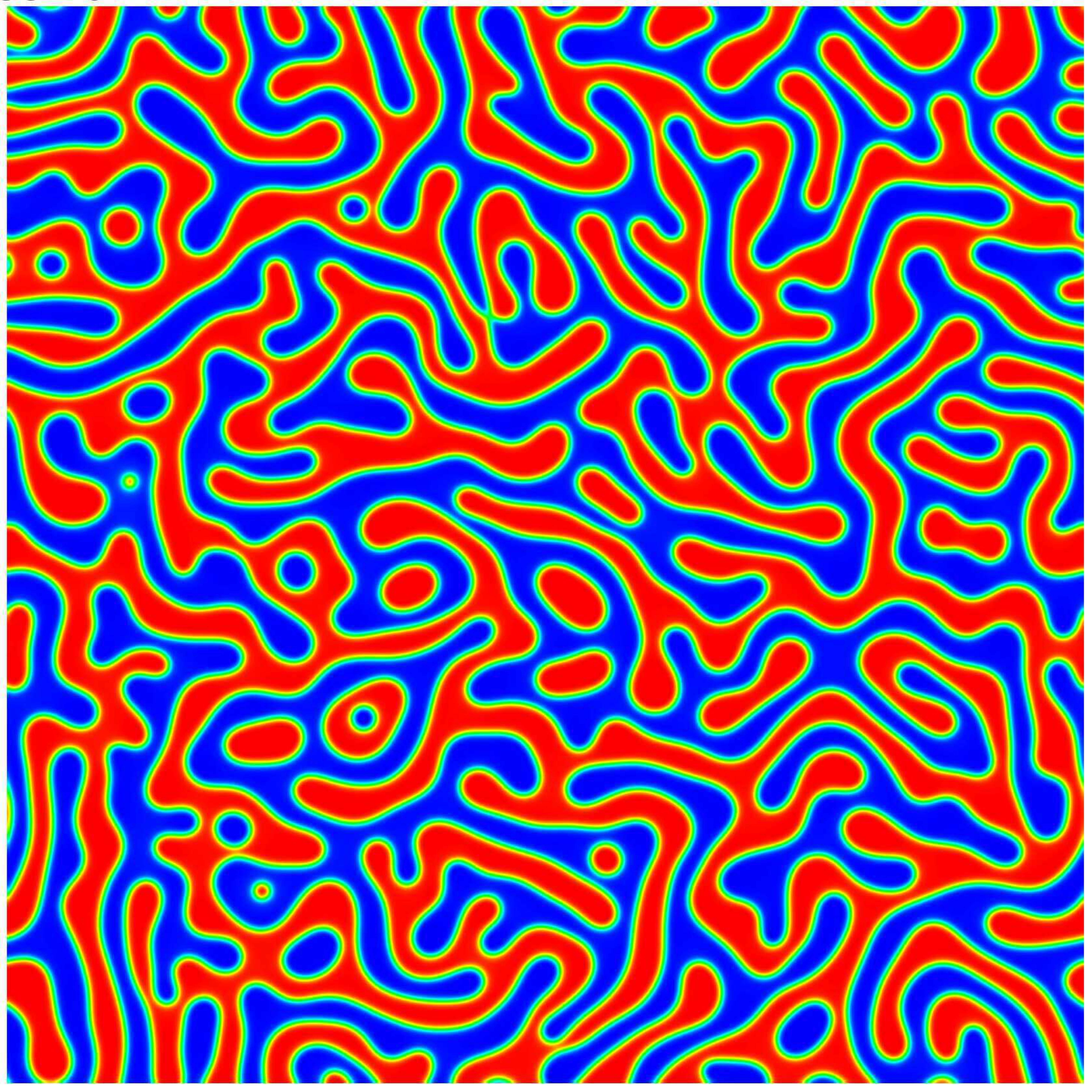}
\includegraphics[width=0.242\textwidth]{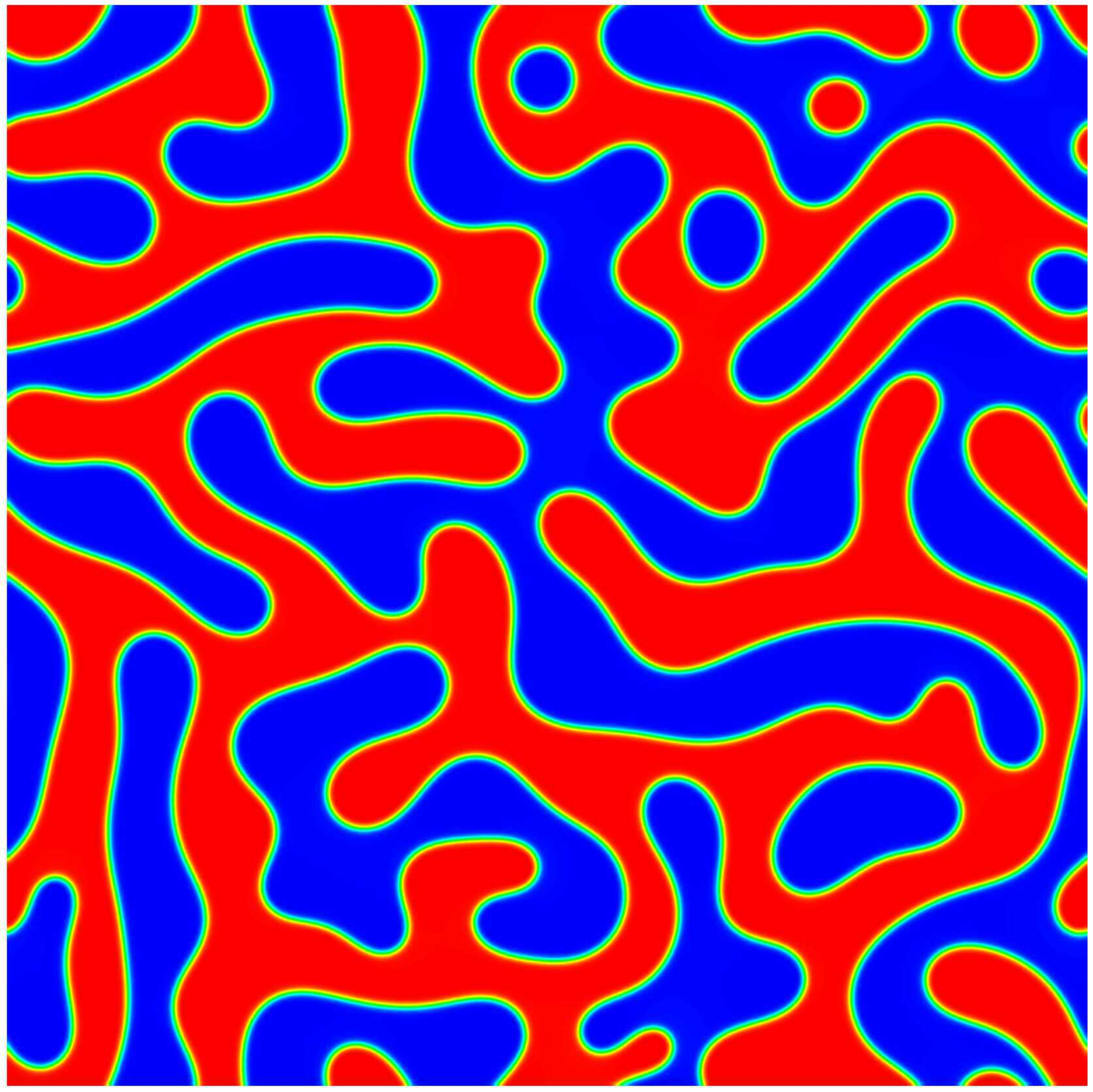}
\includegraphics[width=0.242\textwidth]{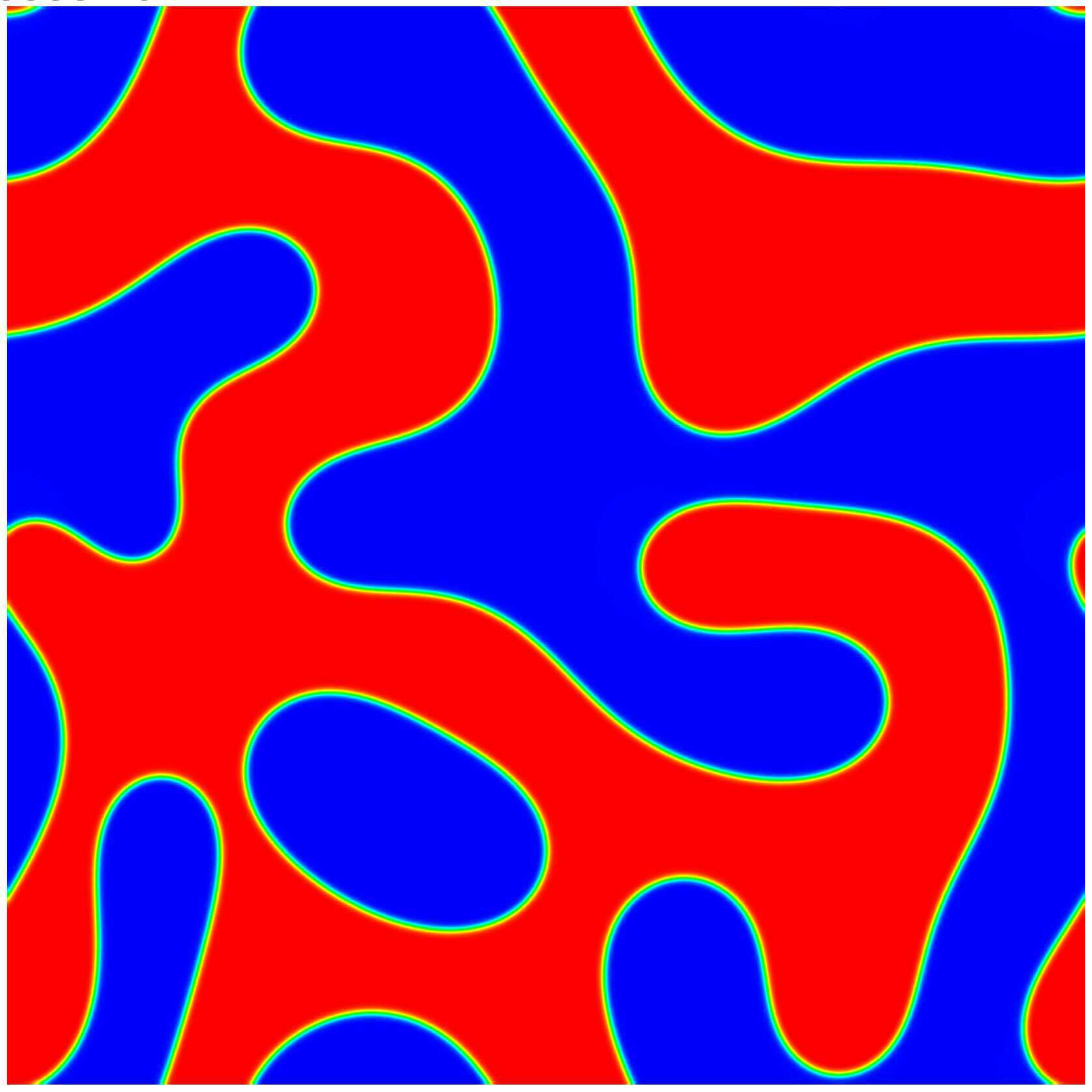}
}

\subfigure[$\hat{\phi}_0=0.4$, $\phi$ at $t=0.1, 0.5, 5, 50$]{
\includegraphics[width=0.242\textwidth]{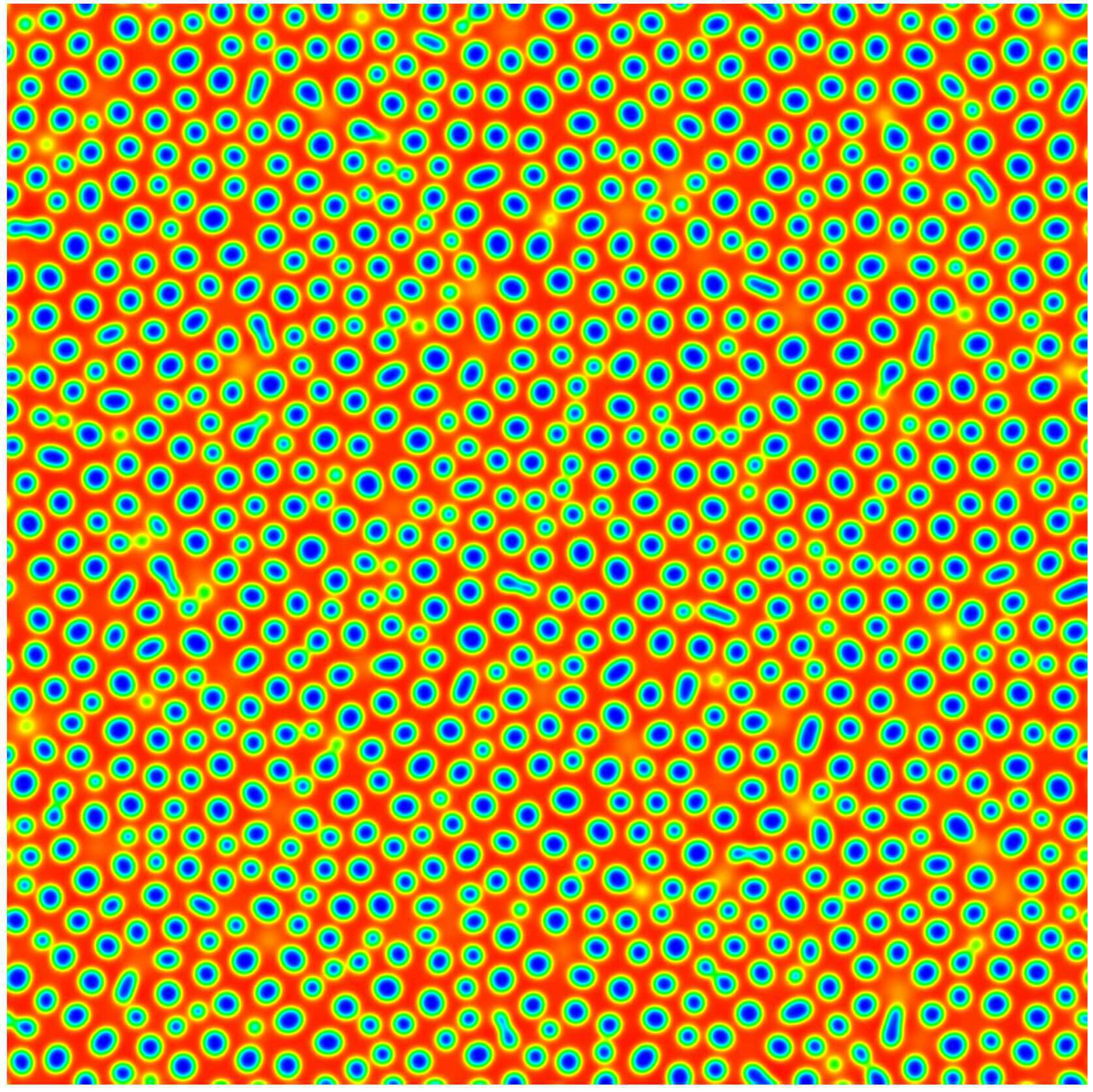}
\includegraphics[width=0.242\textwidth]{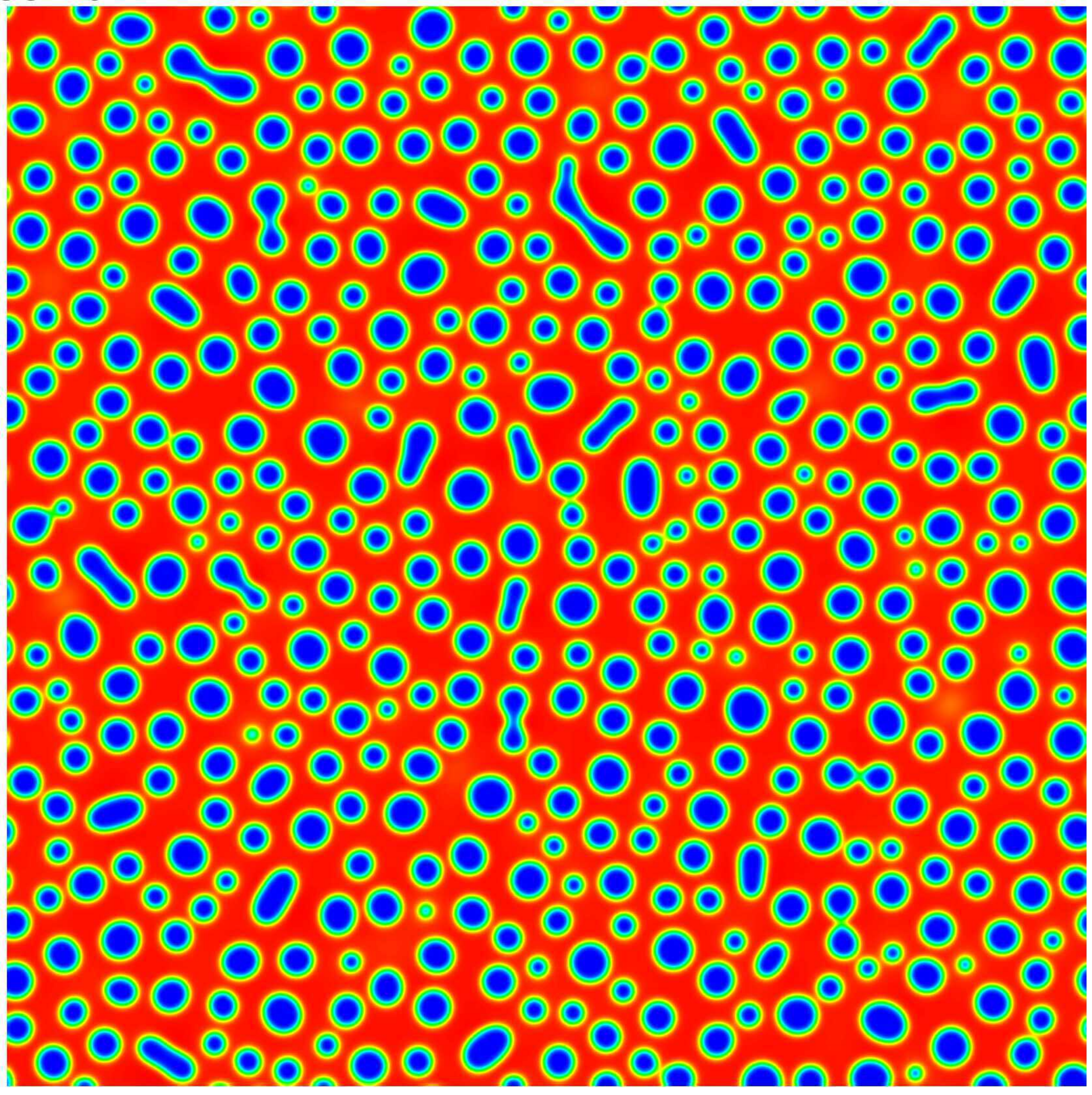}
\includegraphics[width=0.242\textwidth]{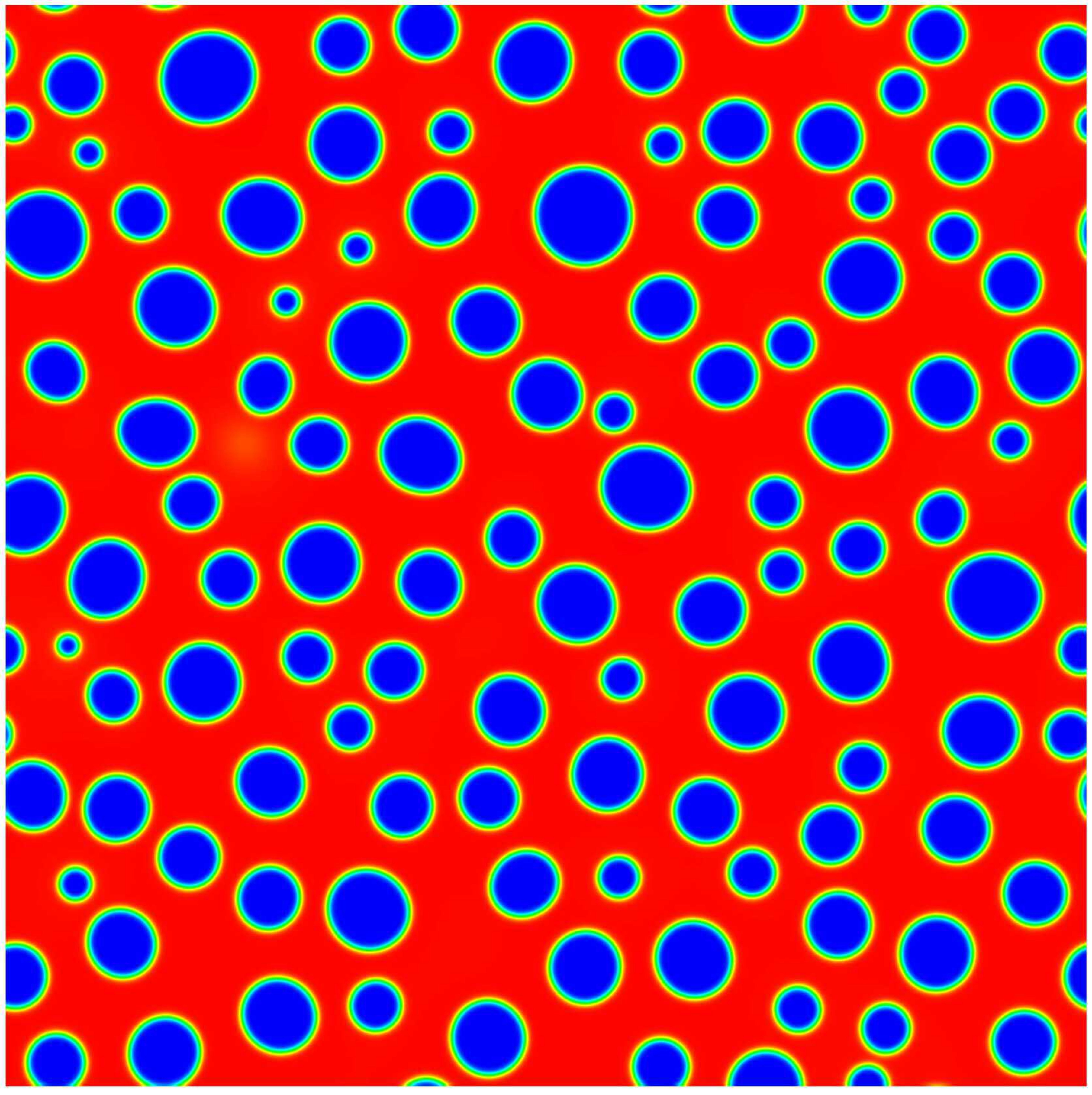}
\includegraphics[width=0.242\textwidth]{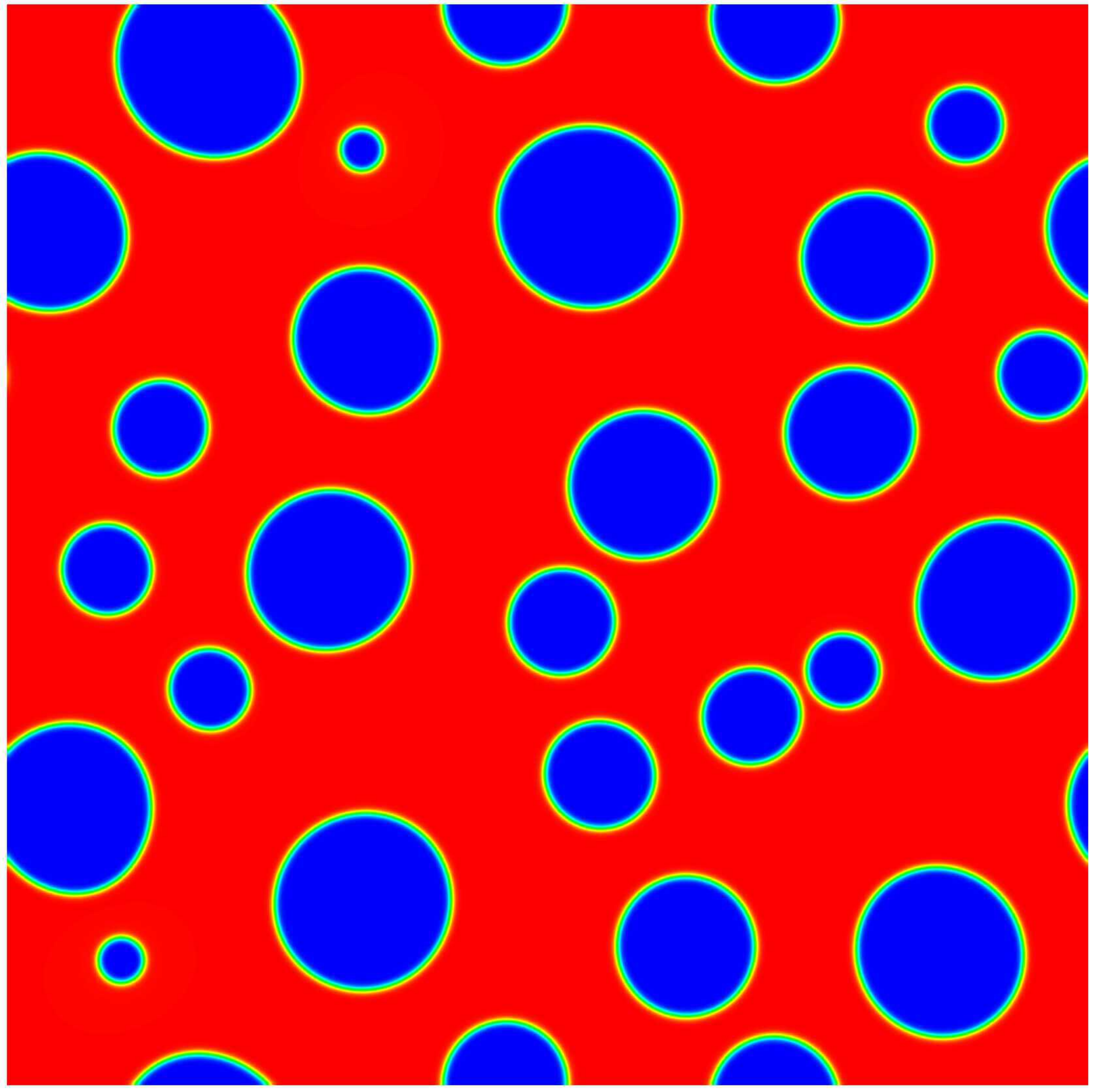}
}
\caption{Coarsening dynamics driven by the Cahn-Hilliard equation. In (a), we set  $\hat{\phi}_0 =0$, and the profiles of $\phi$ at $t=0.1, 0.5,, 5, 50$ are shown; (b) we set  $\hat{\phi}_0=0.4$, and the profiles of $\phi$ at $t=0.1, 0.5, 5, 50$ are shown.}
\label{fig:CH-coarsening}
\end{figure}

\subsection{Molecular beam epitaxy model with slope selection}
In the next example, we consider the molecular beam epitaxy (MBE) model with slope selection. Given $\phi$ denoting the MBE thickness, the free energy is defined as $\cE=\int_\Omega \frac{\varepsilon^2}{2}(\Delta \phi)^2 + \frac{1}{4}(|\nabla \phi|^2 - 1)^2 d\bx$, with the mobility operator, $\cG = 1$, the general gradient flow model in \eqref{eq:generic-model} is specified as the MBE model with slope selection, which reads as
\beq
\partial_t \phi = - \varepsilon^2 \Delta^2 \phi + \nabla \cdot (|\nabla \phi|^2 - 1) \nabla \phi).
\eeq 
With the similar idea as the previous examples, we introduce the scalar auxiliary variable
\beq
q(t) : =Q(\phi(\bx, t)) = \sqrt{\frac{1}{4} \int_\Omega (|\nabla \phi|^2 - 1 - \gamma_0)^2 d\bx + C_0},
\eeq  
then the reformulated model reads as
\begin{subequations}
\begin{align}
& \partial_t \phi = -\varepsilon^2 \Delta^2 \phi + \gamma_0\Delta \phi + \nabla \cdot ( \frac{ q(t)}{Q(\phi)} \nabla \phi), \\
& \frac{d}{dt} q(t) = \int_\Omega \frac{\nabla \phi \cdot \nabla \partial_t  \phi}{2Q(\phi)}d\bx.
\end{align}
\end{subequations}

As a routine, we test the temporal convergence of the relaxed SAV-CN scheme for the MBE model. Consider the domain $\Omega=[0, 1]^2$, and the model parameter $\epsilon=0.1$. We use the same initial condition as before, i.e. $\phi(x, y, t=0)= 0.01 \cos(2\pi x) \cos(2\pi y)$.  We choose $N_x=N_y=128$,  $\gamma_0=4$, $C_0=0$, and $\eta=0.95$. The numerical errors at $t=0.5$ are calculated and summarized in Figure \ref{fig:MBE-mesh}. A second-order convergence for both $\phi$ and $q$ are observed, when the time step is not too large.

\begin{figure}
\center
\subfigure[time mesh refinement test for $\phi$]{\includegraphics[width=0.4\textwidth]{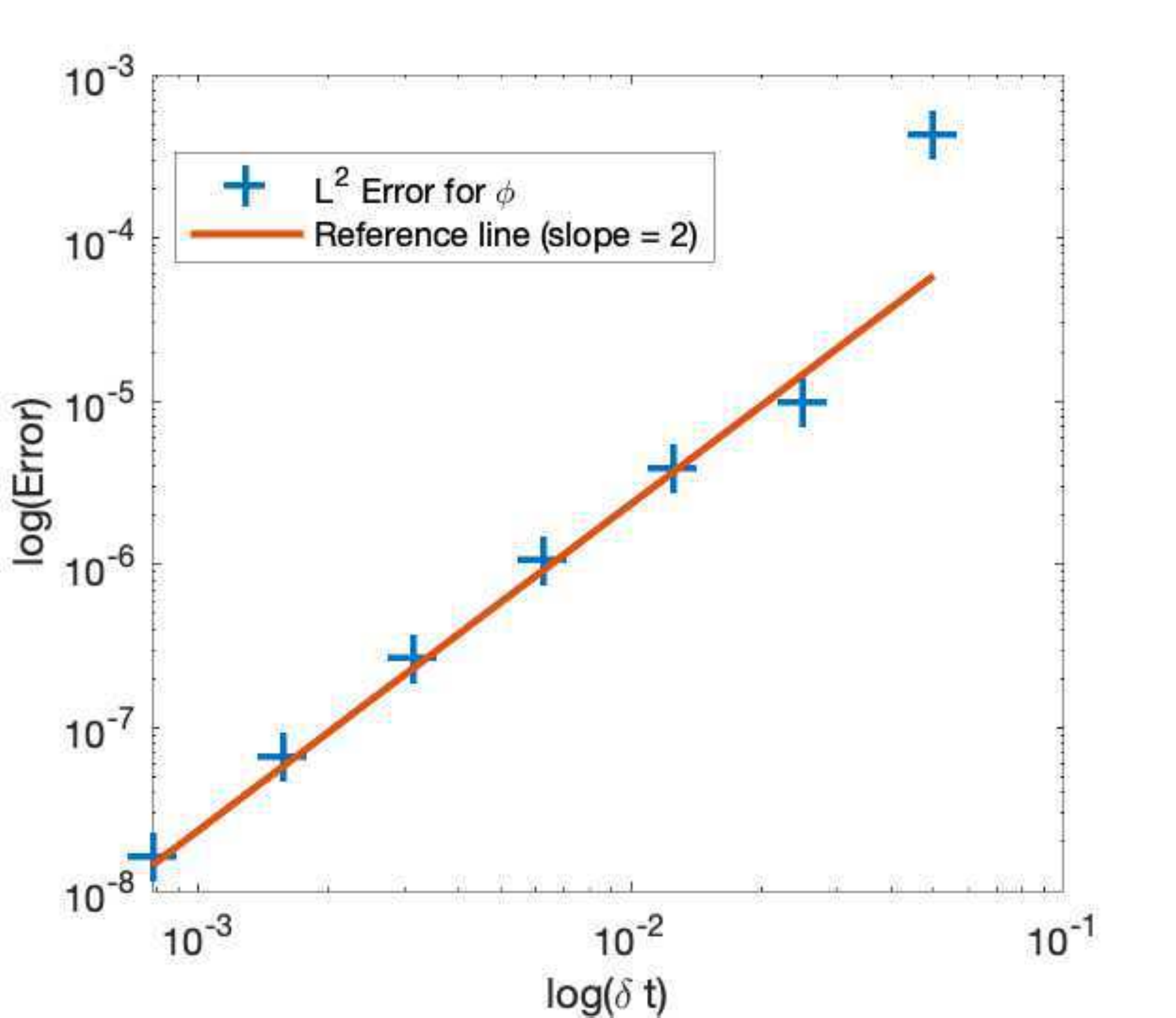}}
\subfigure[time mesh refinement test for $q$]{\includegraphics[width=0.4\textwidth]{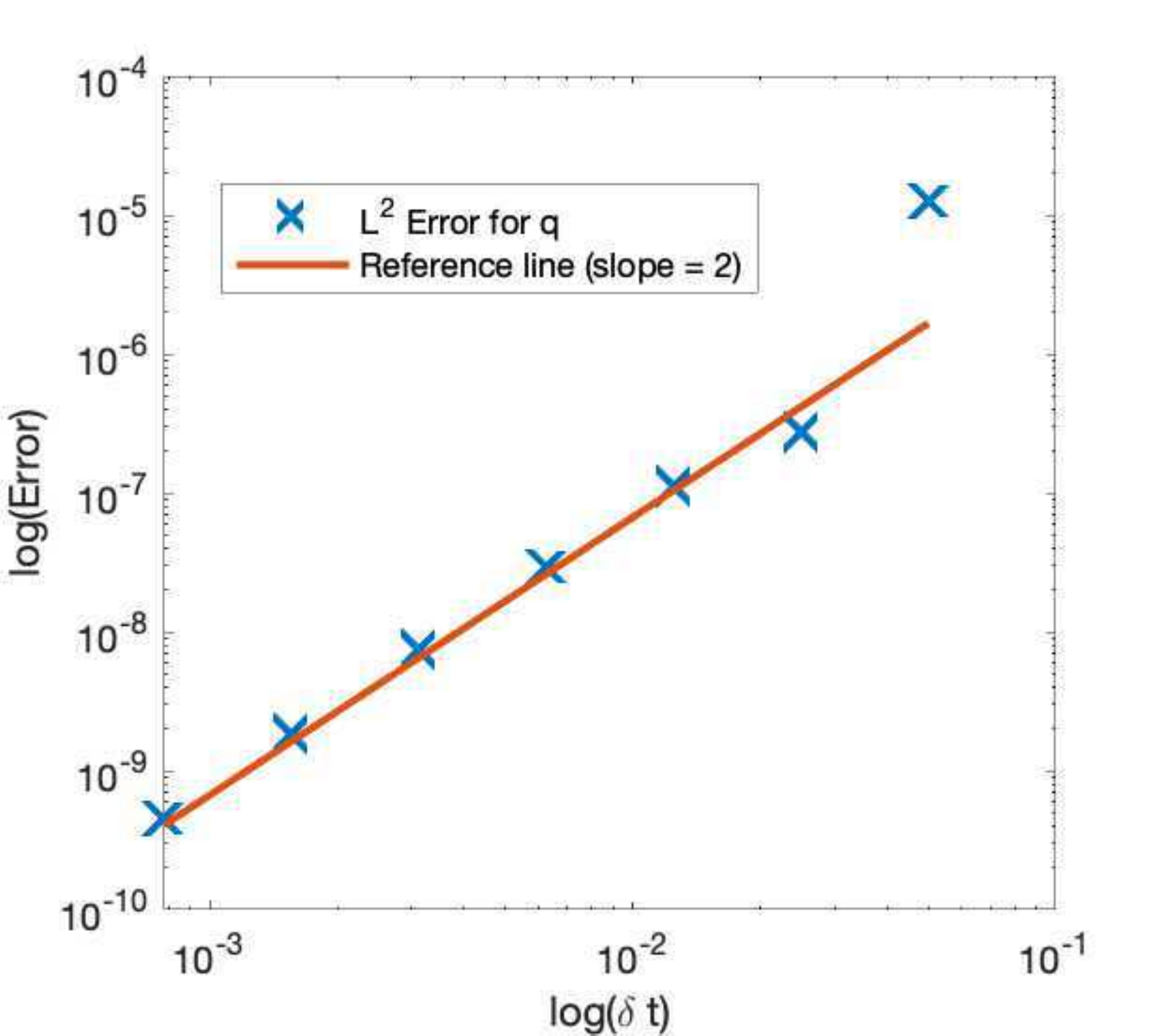}}
\caption{Time step mesh refinement tests of RSAV-CN method for solving the MBE equation. This figure indicates that the proposed RSAV-CN algorithm is second-order accuracy in time when solving the MBE equation.}
\label{fig:MBE-mesh}
\end{figure}

Then, we compare the accuracy between the baseline SAV-CN method and the relaxed SAV-CN method for solving the MBE model. We use the classical benchmark problem for the MBE model. Mainly we consider the domain $\Omega=[0, 2\pi]^2$, with $\epsilon^2 = 0.1$. We solve the problem with $N_x=N_y=128$, $\gamma_0=4$, $C_0=1$, $\eta=0.95$. The numerical comparisons between the baseline SAV-CN scheme and the relaxed SAV-CN scheme are summarized in Figure \ref{fig:MBE-compare}. We observe that the relaxation step increases the numerical accuracy and guarantees the numerical consistency between $q^{n+1}$ and $Q(\phi^{n+1})$. Here we emphasize that the relaxation step is computationally negligible. 

\begin{figure}
\center
\subfigure[Numerical energies]{\includegraphics[width=0.4\textwidth]{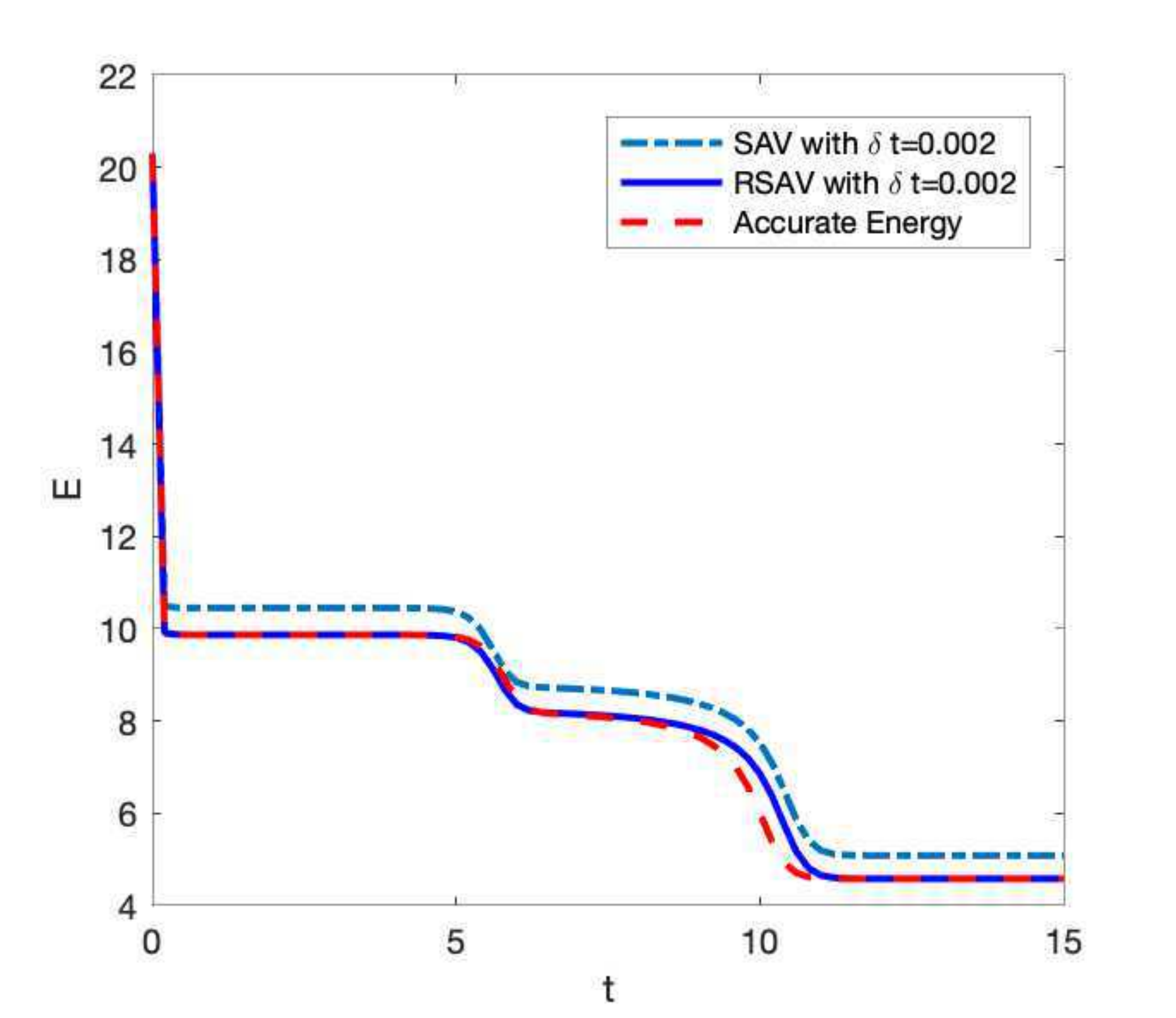}}
\subfigure[Numerical errors for $q(t)-Q(\phi(\bx,t))$]{\includegraphics[width=0.4\textwidth]{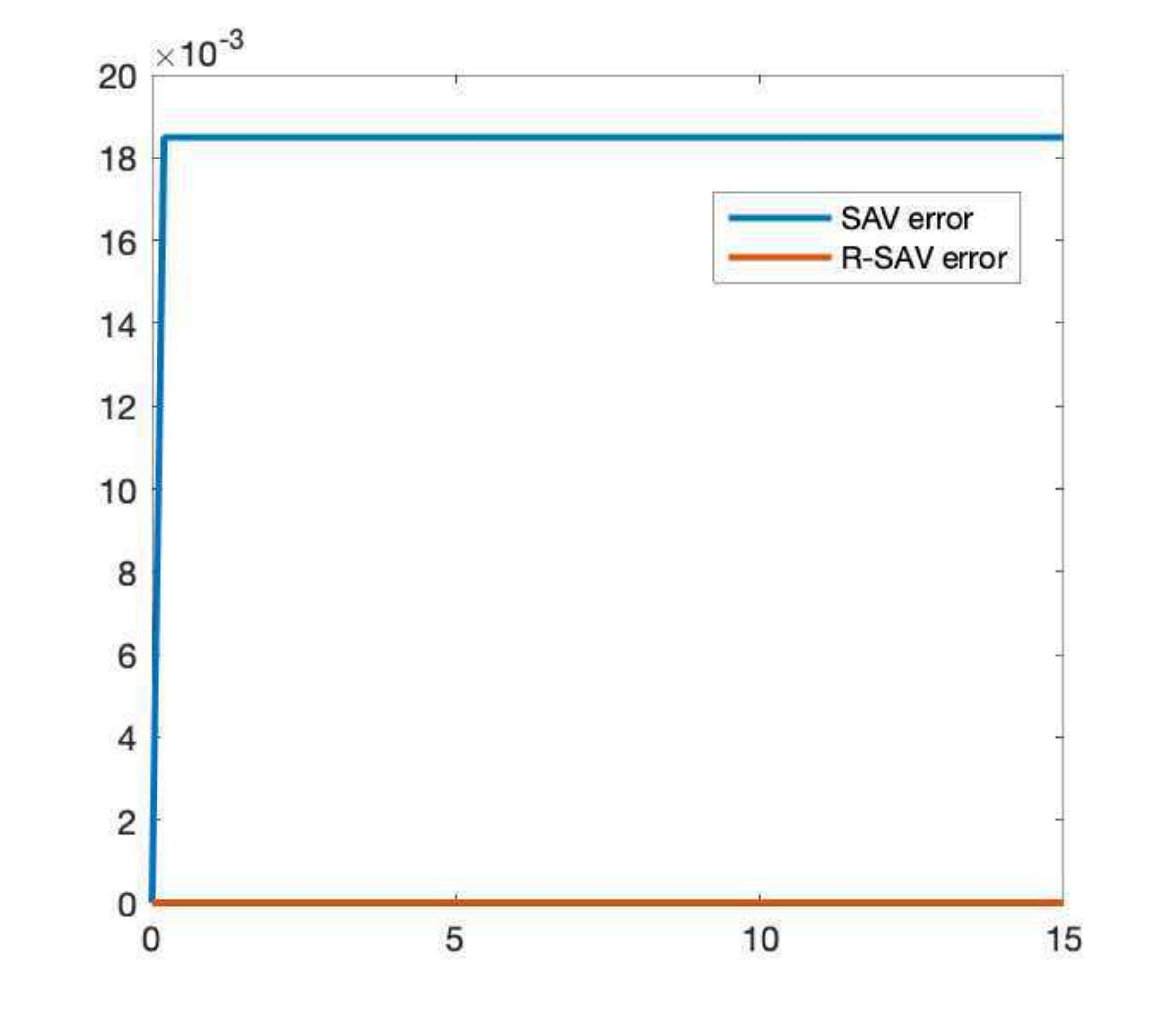}}
\caption{A comparison between the baseline SAV-CN method and the RSAV-CN method for solving the MBE model. In (a) the comparisons of numerical energies are shown. The RSAV method provides more accurate result. In (b), the numerical results for $q(t)-Q(\phi(\bx,t))$ are shown, where we observe that the baseline SAV introduces numerical errors for $q(t)-Q(\phi(\bx,t))$, but the relaxed SAV gauntnesses the consistency of $q(t)$ and $Q(\phi(\bx, t)$ numerically.}
\label{fig:MBE-compare}
\end{figure}

\subsection{Other phase-field models}
Thanks to the SAV method's generality, the relaxed SAV method also applies to various dissipative PDE models, and particularly thermodynamically consistent phase field models. We skip some details of using the RSAV method to other models due to space limitation but focus on two more specific applications: (1) the phase-field crystal model; and (2) the diblock copolymer model.

First of all, we consider its application to the phase field crystal (PFC) model. Consider the free energy
$\cE=\int_\Omega \frac{1}{2}  \phi (a_0+\Delta)^2 \phi + \frac{1}{4}\phi^4 -\frac{b_0}{2} \phi^2 d\bx$, and the mobility operator $\cG = -\lambda \Delta$, where $a_0$ and $b_0$ are model parameters. The general gradient flow model in \eqref{eq:generic-model} is reduced to the PFC model, which reads as
\begin{subequations}
\begin{align}
&\partial_t \phi = \lambda \Delta \mu, \\
& \mu =  -(a_0+ \Delta)^2 \phi + \phi^3 -b_0 \phi.
\end{align}
\end{subequations}
If we introduce the scalar auxiliary variable $q(t) :=Q(\phi(\bx, t))= \sqrt{\frac{1}{4} \int_\Omega (\phi^2 - b_0 -\gamma_0)^2d\bx + C_0}$, we get the reformulated model
\begin{subequations}
\begin{align}
& \partial_t \phi = \lambda \Delta \mu, \\
& \mu = -(a_0+\Delta )^2 \phi + \gamma_0 \phi + \frac{q(t)}{Q(\phi)}V(\phi), \quad   V(\phi) = \phi (\phi^2 - b_0-\gamma_0), \\
& \frac{d}{dt} q(t) = \int_\Omega \frac{V(\phi)}{2Q(\phi)} \partial_t \phi d\bx.
\end{align}
\end{subequations}

We verify that the RSAV scheme shows second-order convergence in time when solving the PFC model. The results are not shown to save space. Then we compare the accuracy between the baseline SAV-CN scheme and the relaxed SAV-CN scheme.  We consider the domain $\Omega=[0, 400]^2$, and set $a_0=1$, $b_0=0.325$, $\lambda=1$.  To solve the PFC model, we use the numerical  parameters $\gamma_0=1$, $C_0=1$, $N_x=N_y=512$, $\eta=0.95$. The initial condition is chosen as shown in Figure \ref{fig:PFC-example}(a).  The numerical comparisons between the two schemes are summarized in Figure \ref{fig:PFC-compare}. The RSAV-CN scheme shows more accurate results. Most importantly, the results obtained from the RSAV-CN method show the numerical consistency between the modified energy and the original energy, since it guarantees the consistency of $q(t)$ with $Q(\phi(\bx, t))$ as shown in Figure \ref{fig:PFC-compare}(b).

\begin{figure}
\center
\subfigure[Numerical energies]{\includegraphics[width=0.4\textwidth]{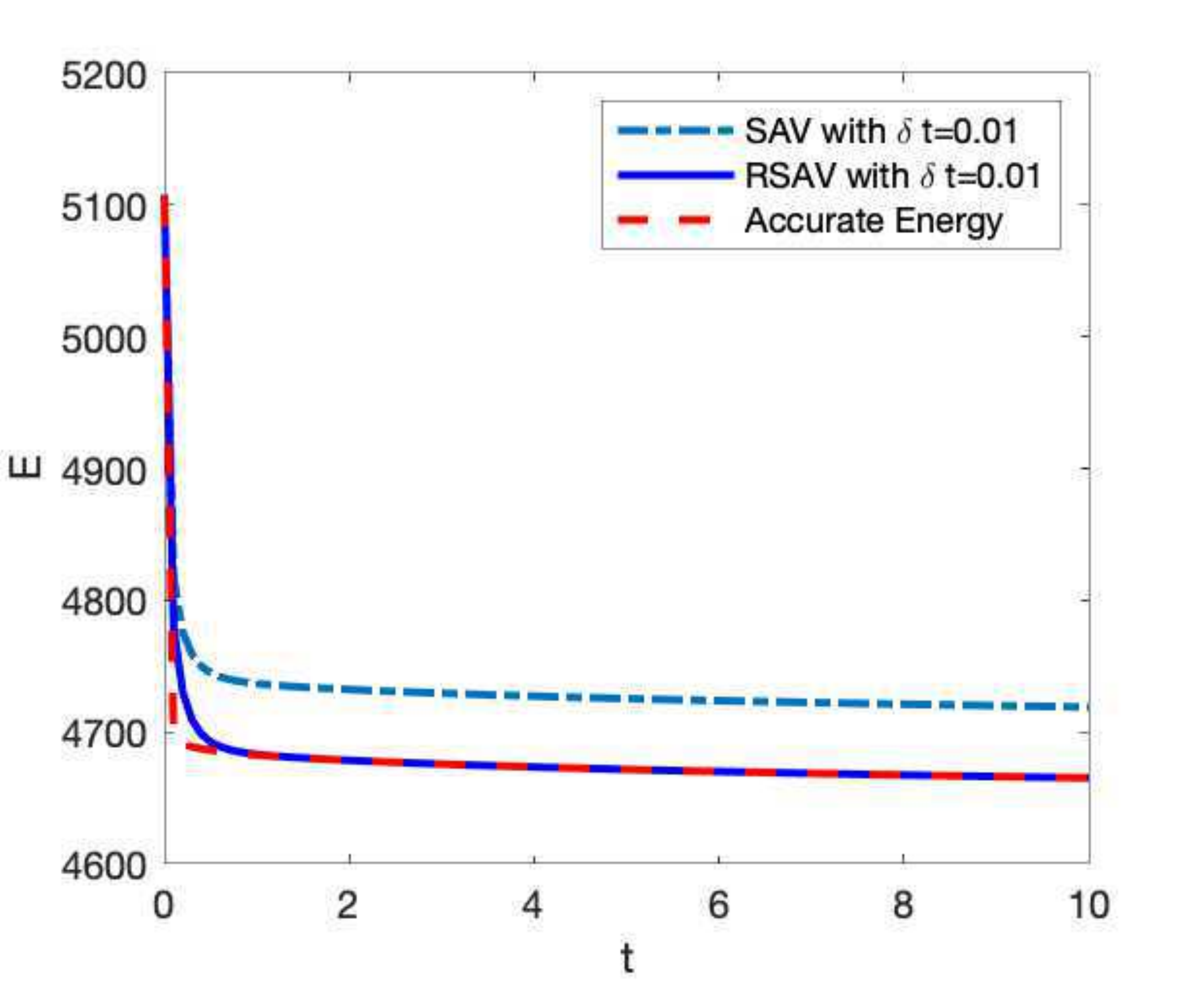}}
\subfigure[Numerical errors for $q(t)-Q(\phi(\bx, t))$]{\includegraphics[width=0.4\textwidth]{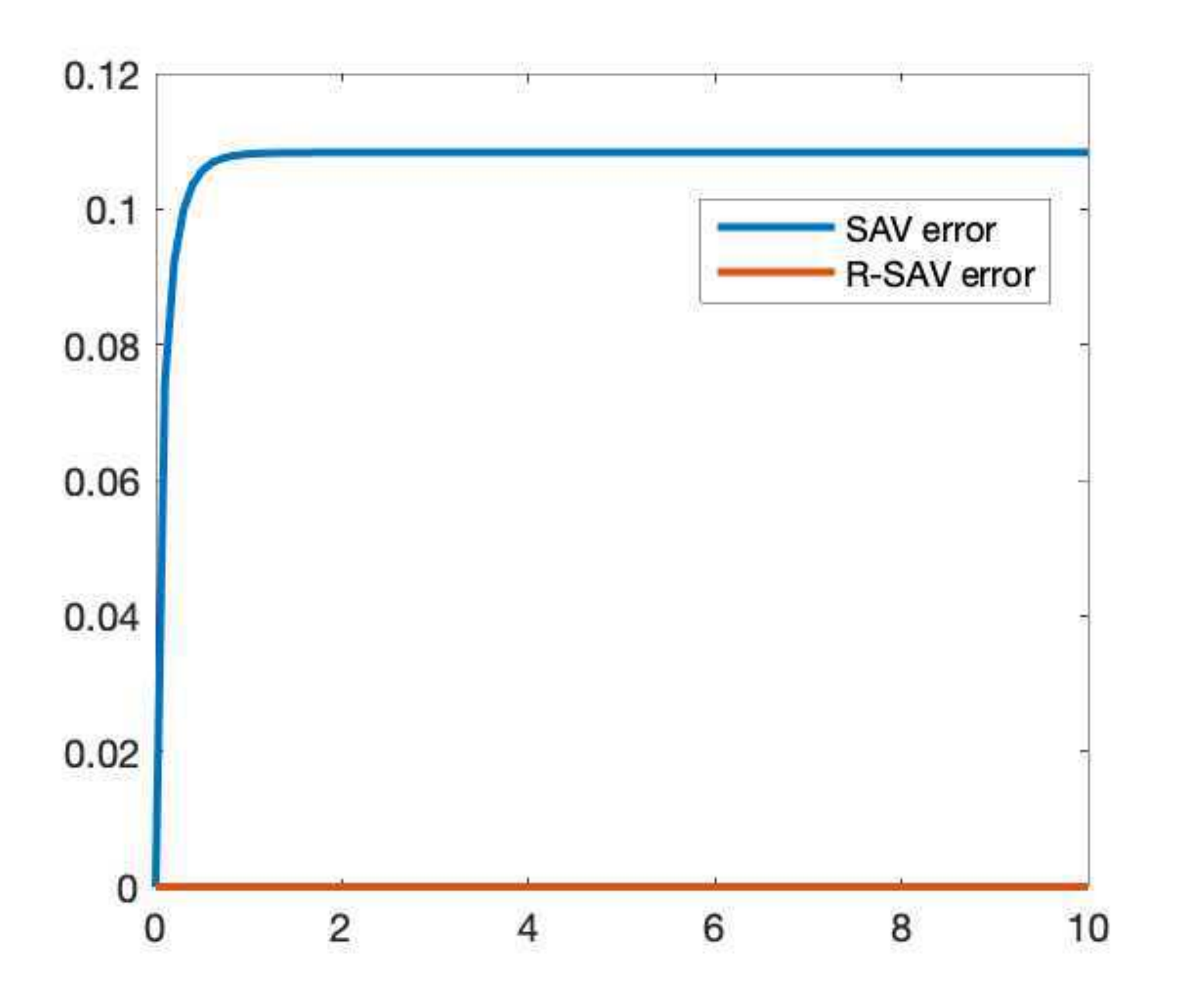}}
\caption{A comparison between the baseline SAV-CN method and the RSAV-CN method for solving the PFC model. In (a) the comparisons of numerical energies using different methods are shown. The RSAV-CN method provides more accurate results than the baseline SAV-CN method. In (b), the numerical errors for $q(t)-Q(\phi(\bx,t))$ are shown. We observe that the baseline SAV-CN method introduces numerical errors for $q(t)-Q(\phi(\bx,t))$, but the RSAV-CN method guarantees the consistency of $q(t)$ and $Q(\phi(\bx, t)$ numerically.}
\label{fig:PFC-compare}
\end{figure}

Also, the profiles of $\phi$ at various times using the relaxed SAV-CN scheme with a time step $\delta t =0.01$ are summarized in Figure \ref{fig:PFC-example}. It indicates the relaxed SAV-CN can be utilized to investigate long-time dynamics and provides accurate numerical results.

\begin{figure}
\center
\subfigure[$\phi$ at $t=0,20,50,80$]{
\includegraphics[width=0.242\textwidth]{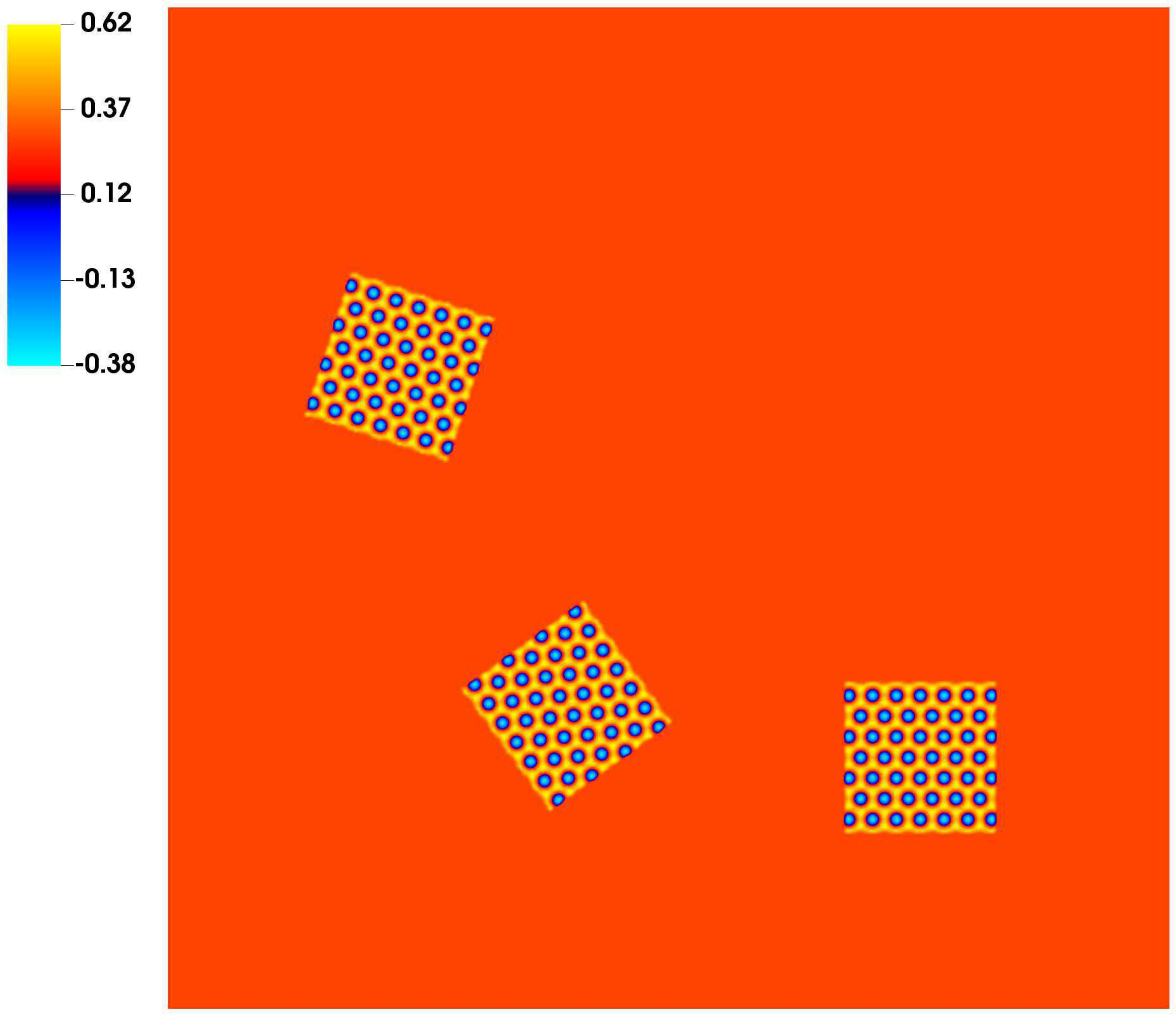}
\includegraphics[width=0.242\textwidth]{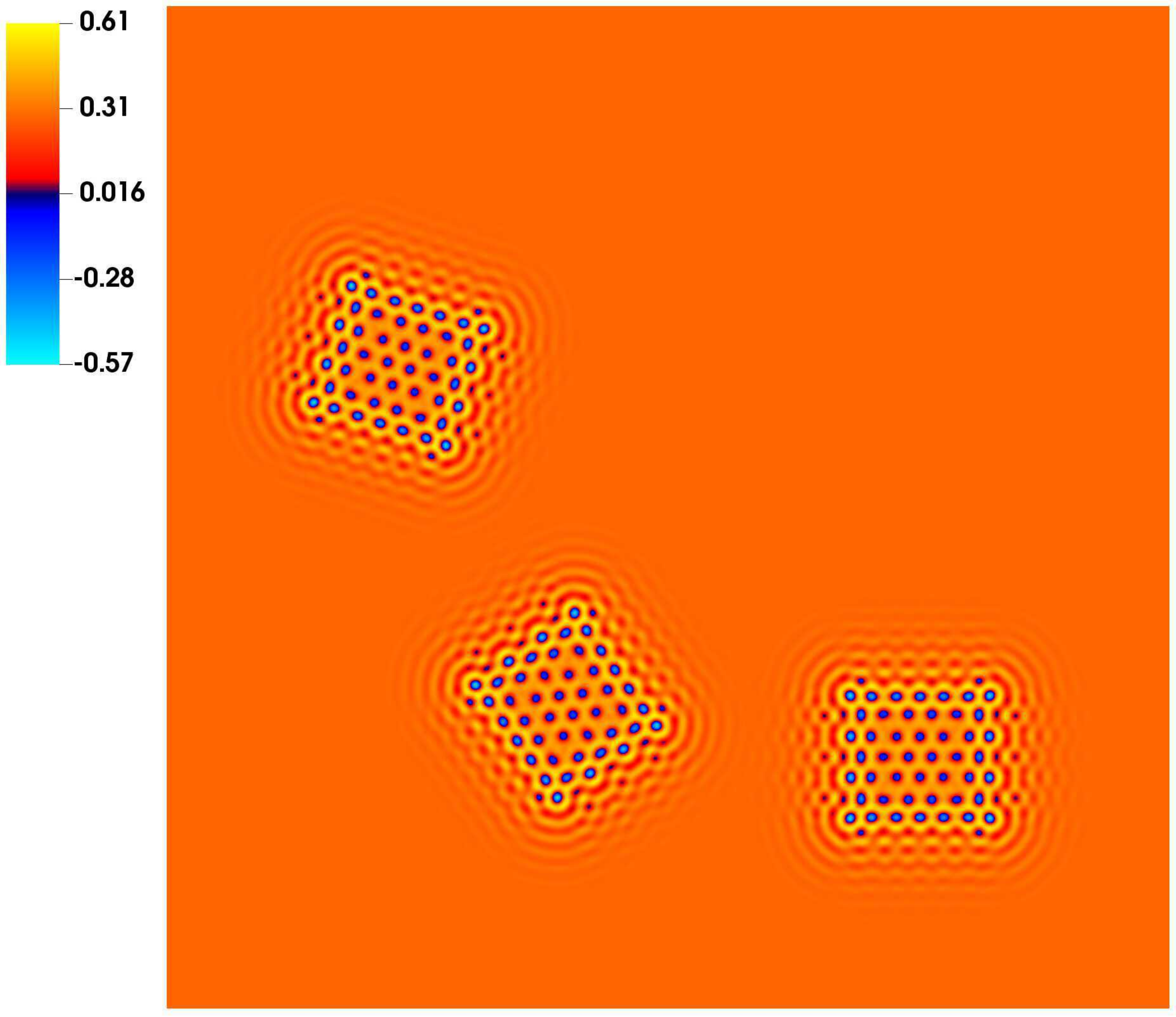}
\includegraphics[width=0.242\textwidth]{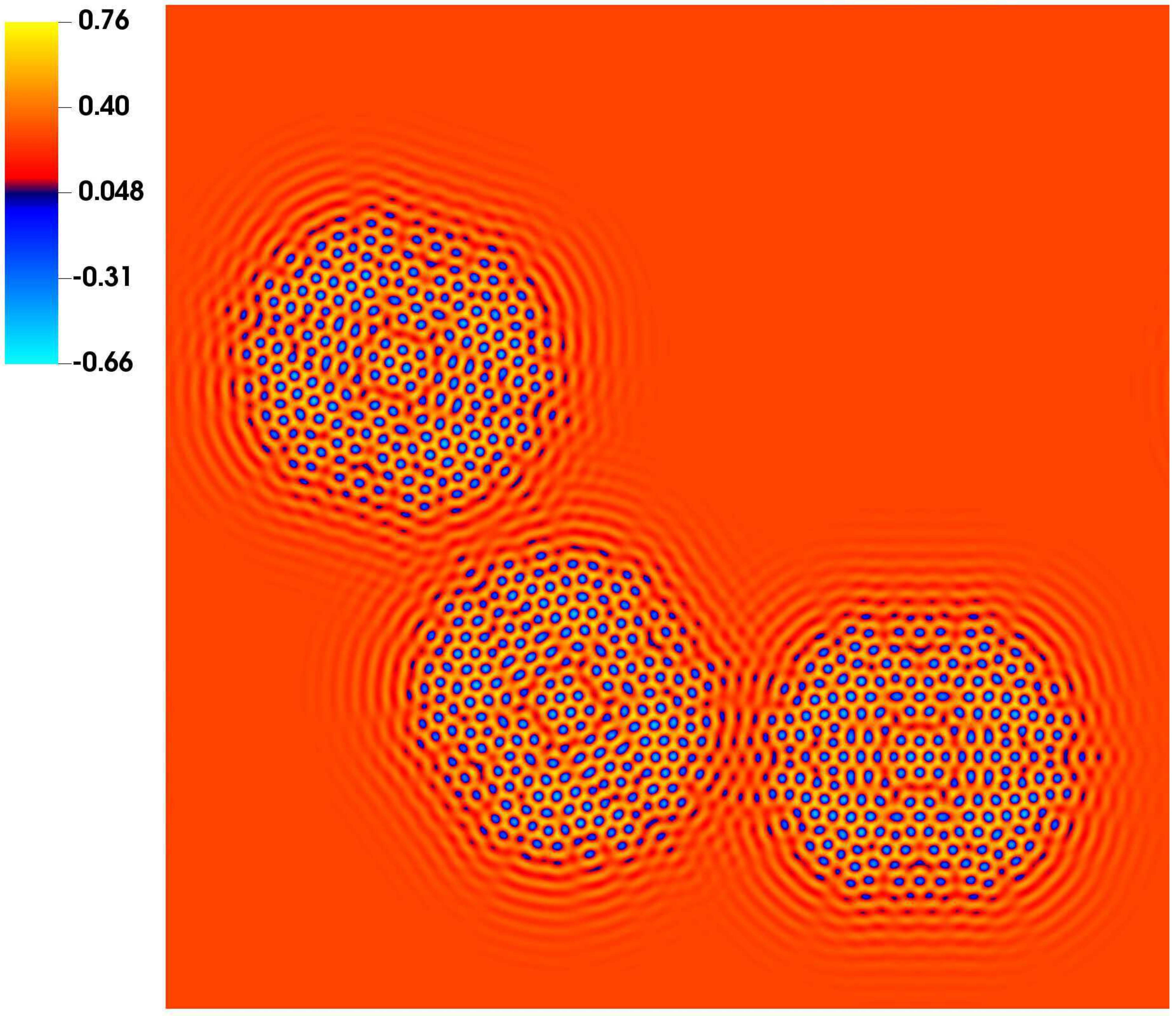}
\includegraphics[width=0.242\textwidth]{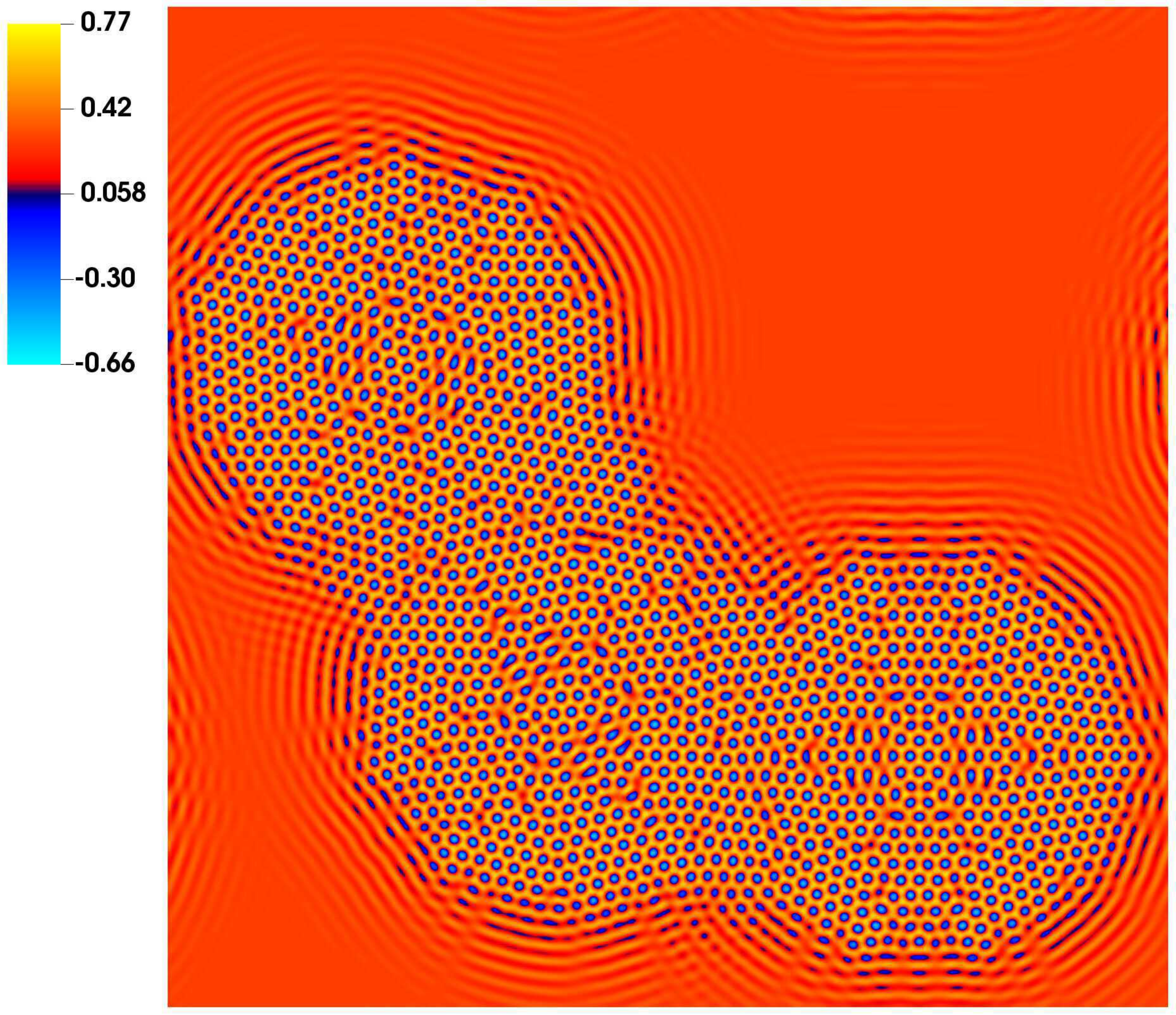}
}

\subfigure[$\phi$ at $t=90,100,150, 250$]{
\includegraphics[width=0.242\textwidth]{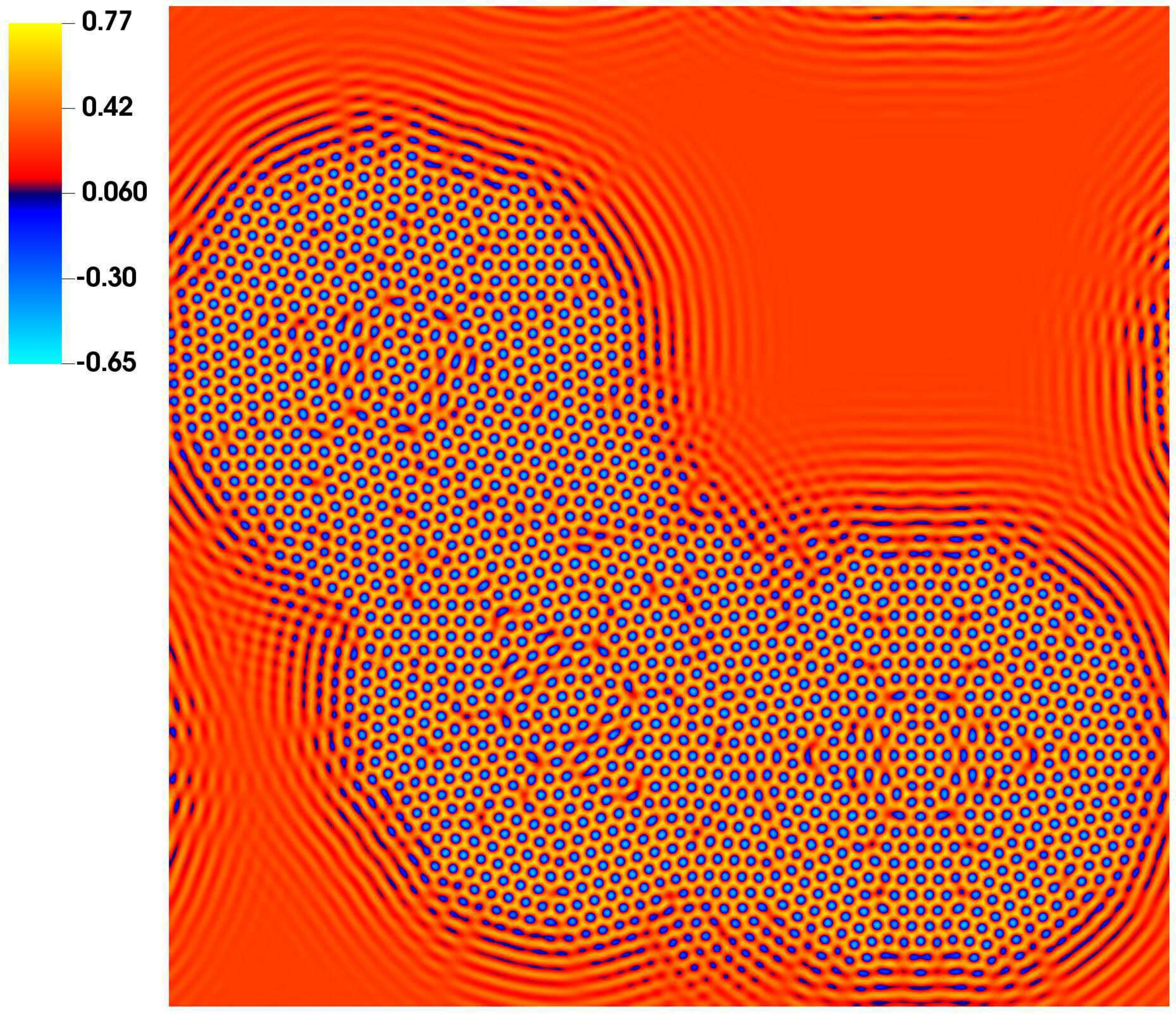}
\includegraphics[width=0.242\textwidth]{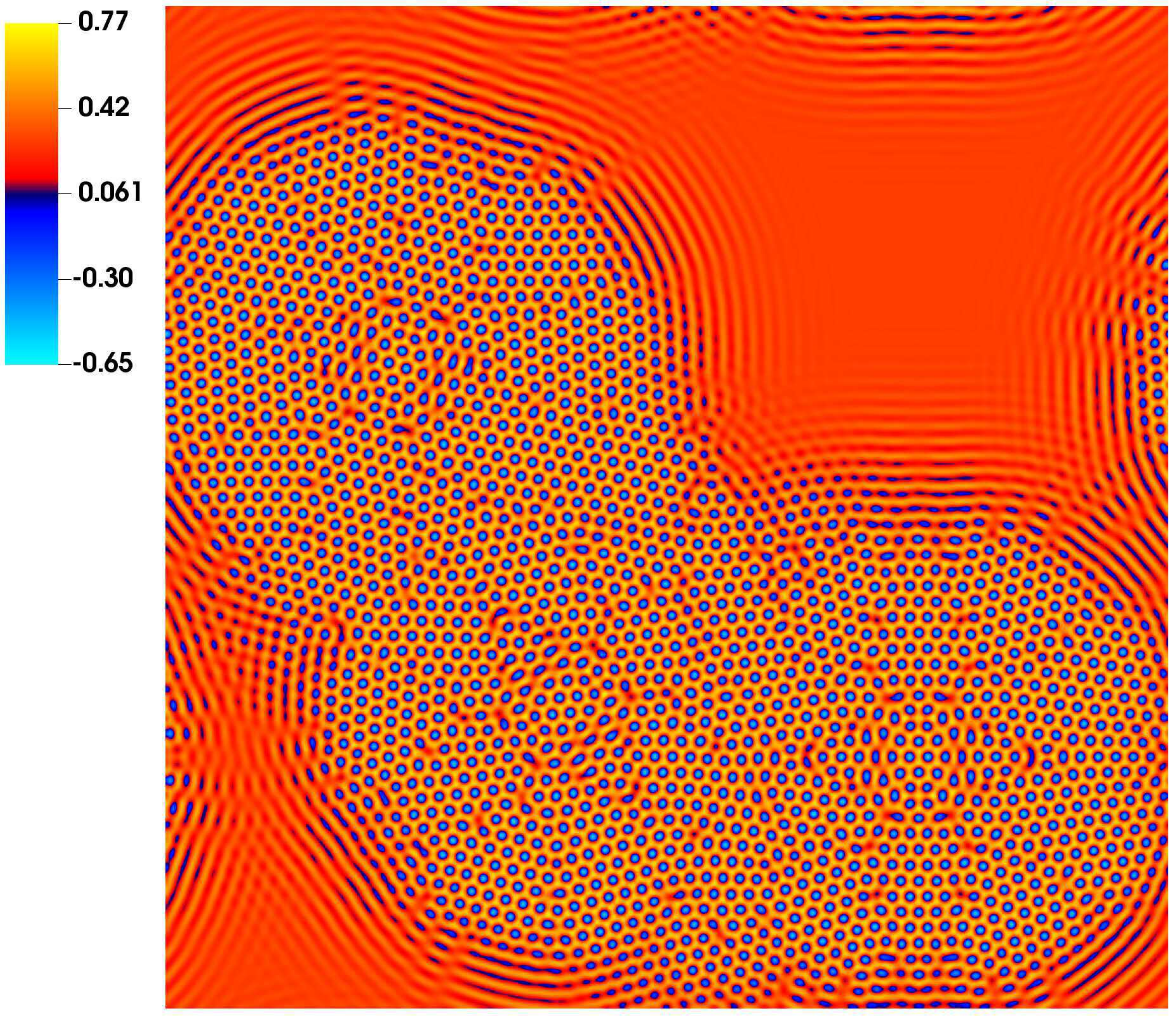}
\includegraphics[width=0.242\textwidth]{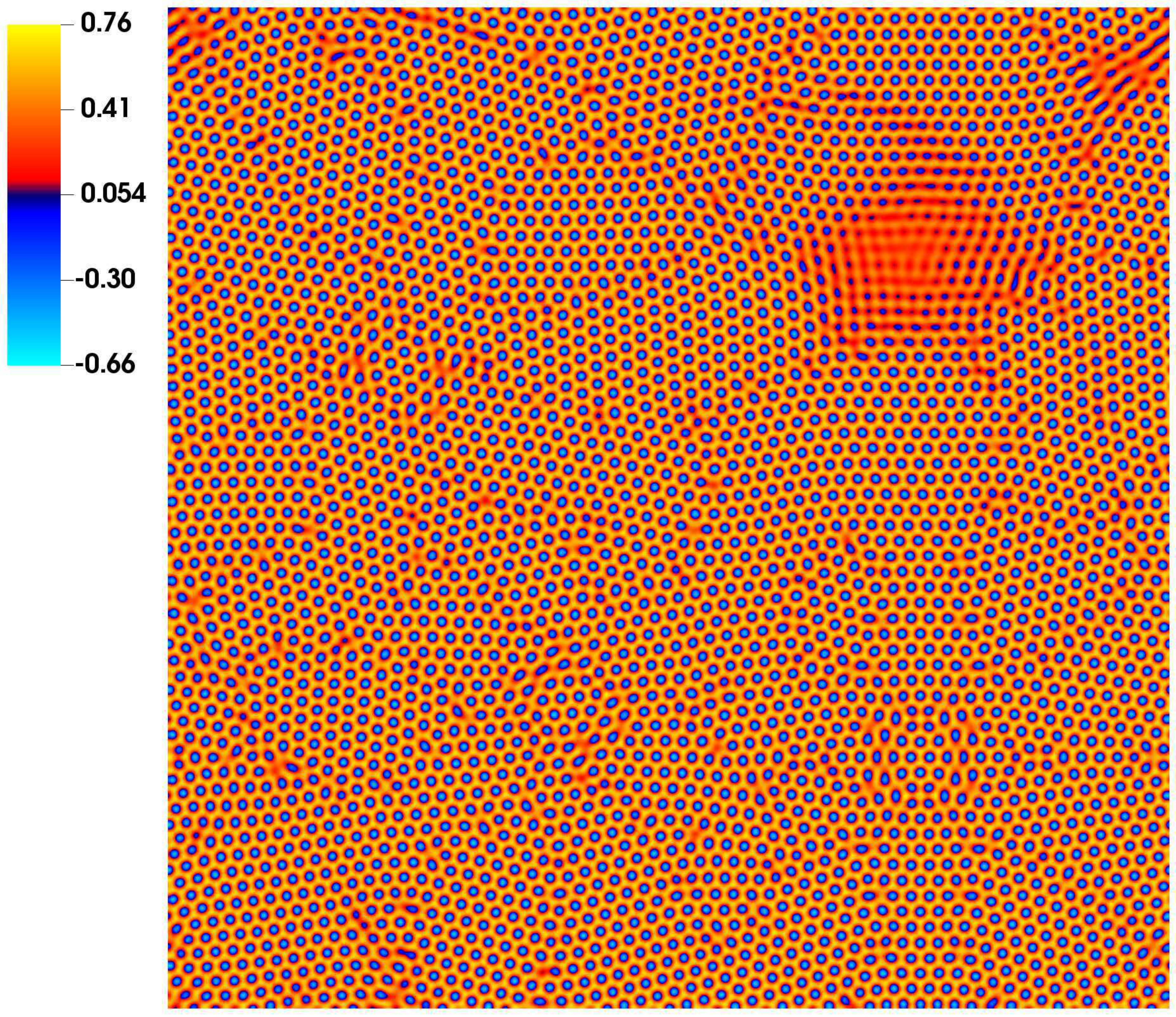}
\includegraphics[width=0.242\textwidth]{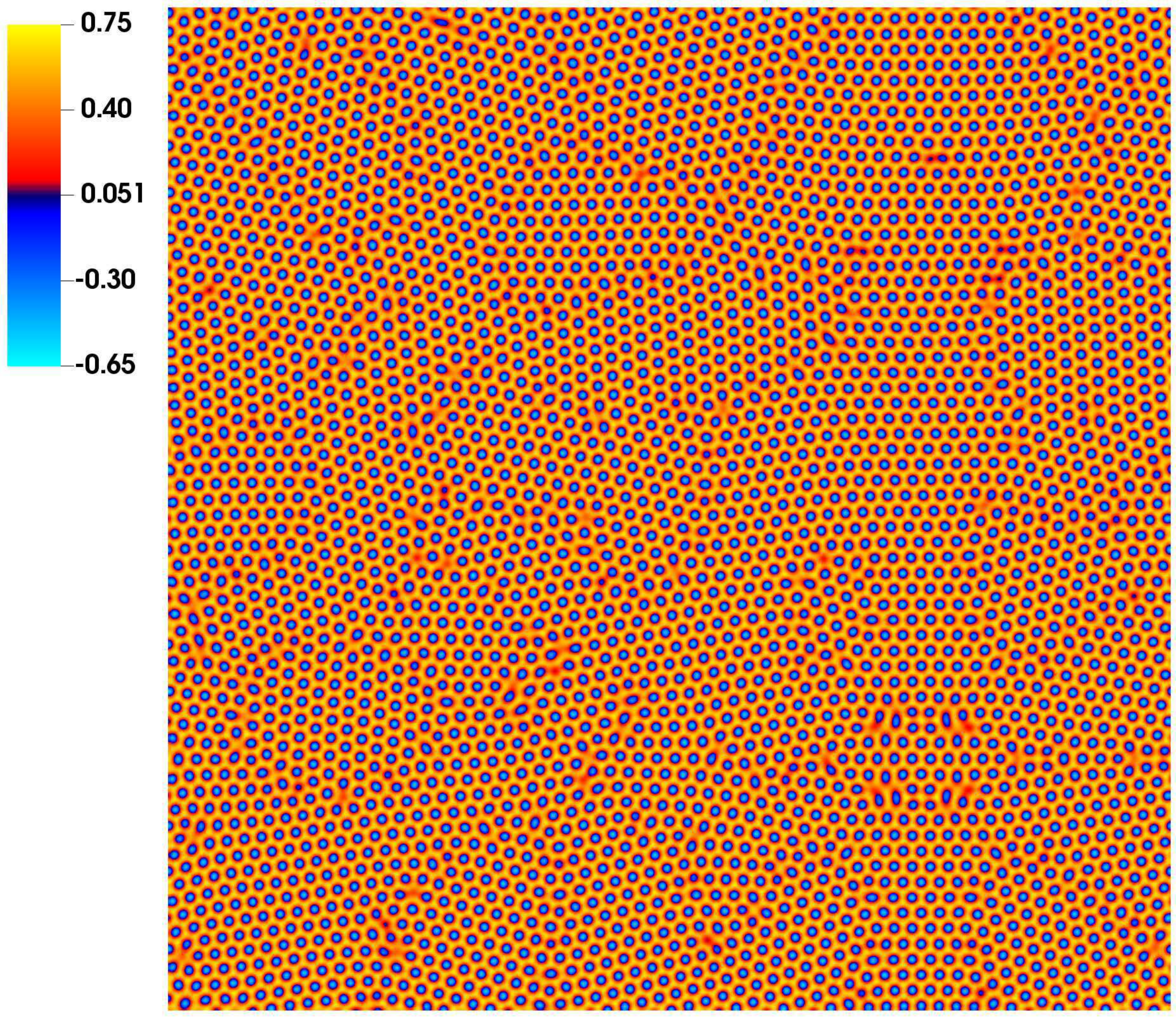}
}

\caption{Crystal growth dynamics driven by the phase-field crystal model. The profiles of $\phi$ at various time slots are shown.}
\label{fig:PFC-example}
\end{figure}

Furthermore, we use the relaxed SAV-CN scheme to investigate the dynamics driven by the PFC model. In this case, we consider the domain $\Omega=[0, 100]^2$, and choose the initial condition $\phi(x, y,t=0)= \hat{\phi}_0 + 0.01 rand(x, y)$, where $rand(x,y)$ generates random numbers between $-1$ and $1$ and $\hat{\phi}_0$ is a constant. To solve the PFC model, we use the numerical settings $N_x=N_y=256$, $\gamma_0 = 1$, $C_0=1$. The numerical results are summarized in Figure \ref{fig:PFC-pattern}. The stripe pattern is observed with $\hat{\phi}_0=0$, and the triangle pattern is observed with $\hat{\phi}_0=0.2$. These observations are consistent with phase diagram of the PFC model as reported in the literature.
\begin{figure}
\center
\subfigure[$\phi$ at $t=1,5,50, 200$ with $\hat{\phi}_0=0$]{
\includegraphics[width=0.24\textwidth]{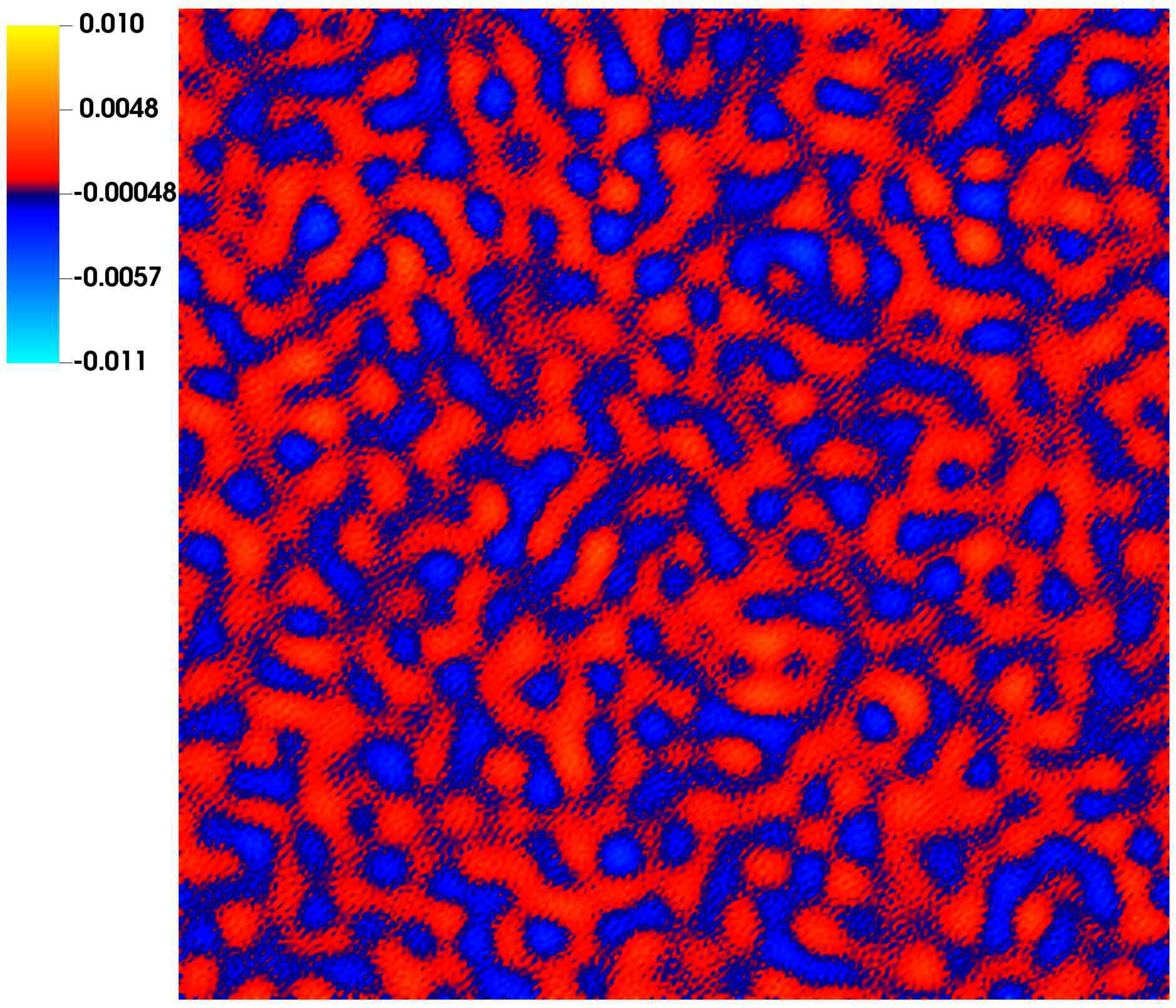}
\includegraphics[width=0.24\textwidth]{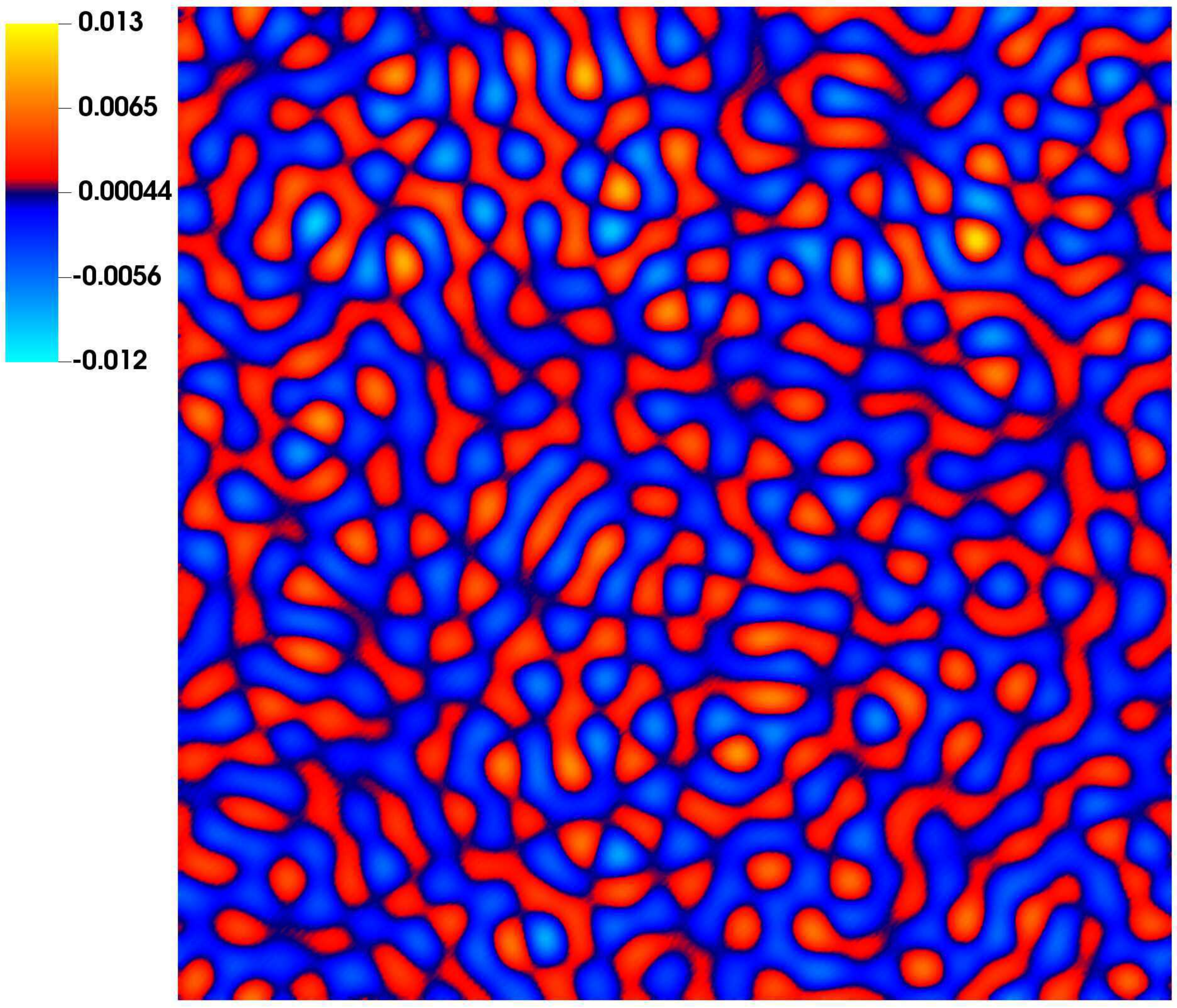}
\includegraphics[width=0.24\textwidth]{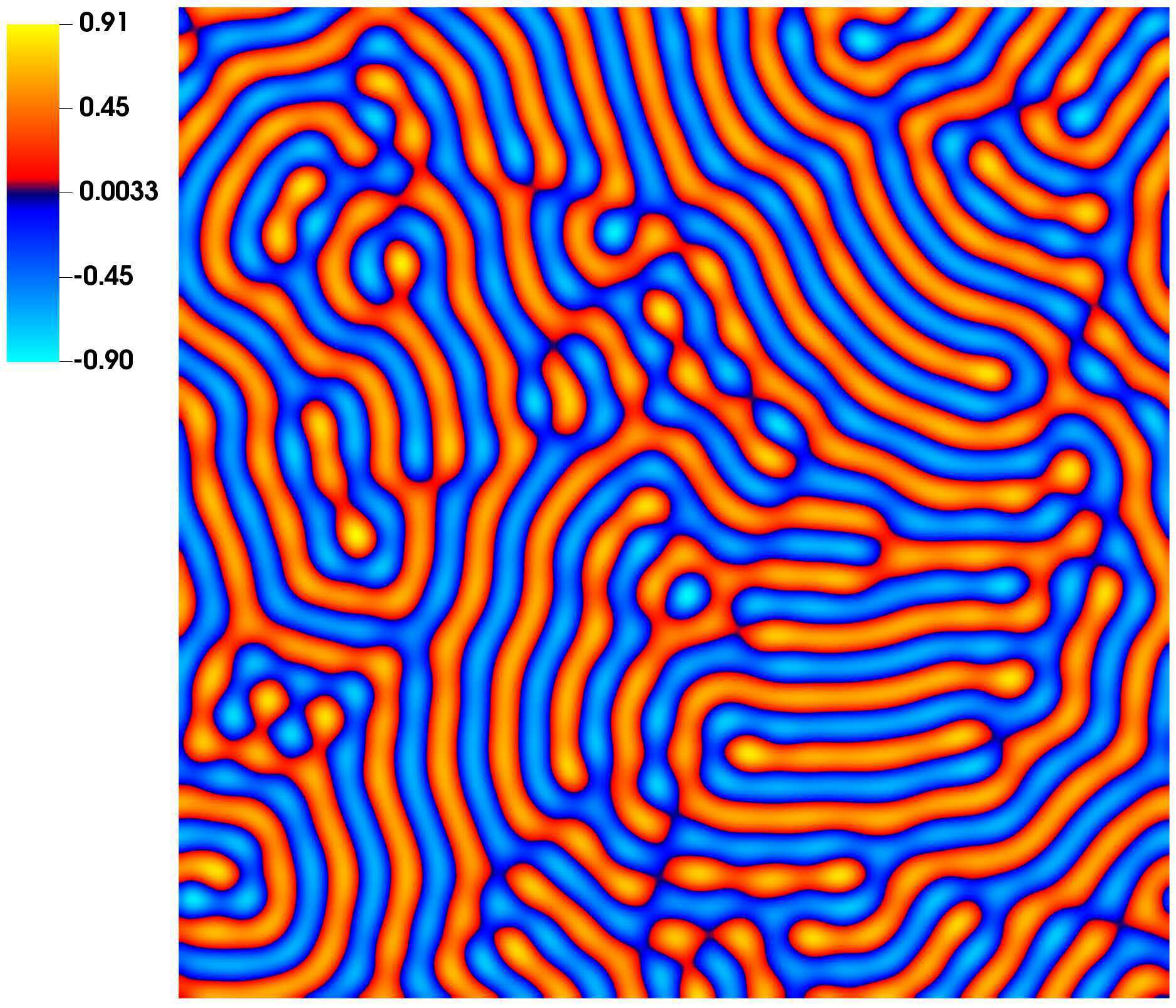}
\includegraphics[width=0.24\textwidth]{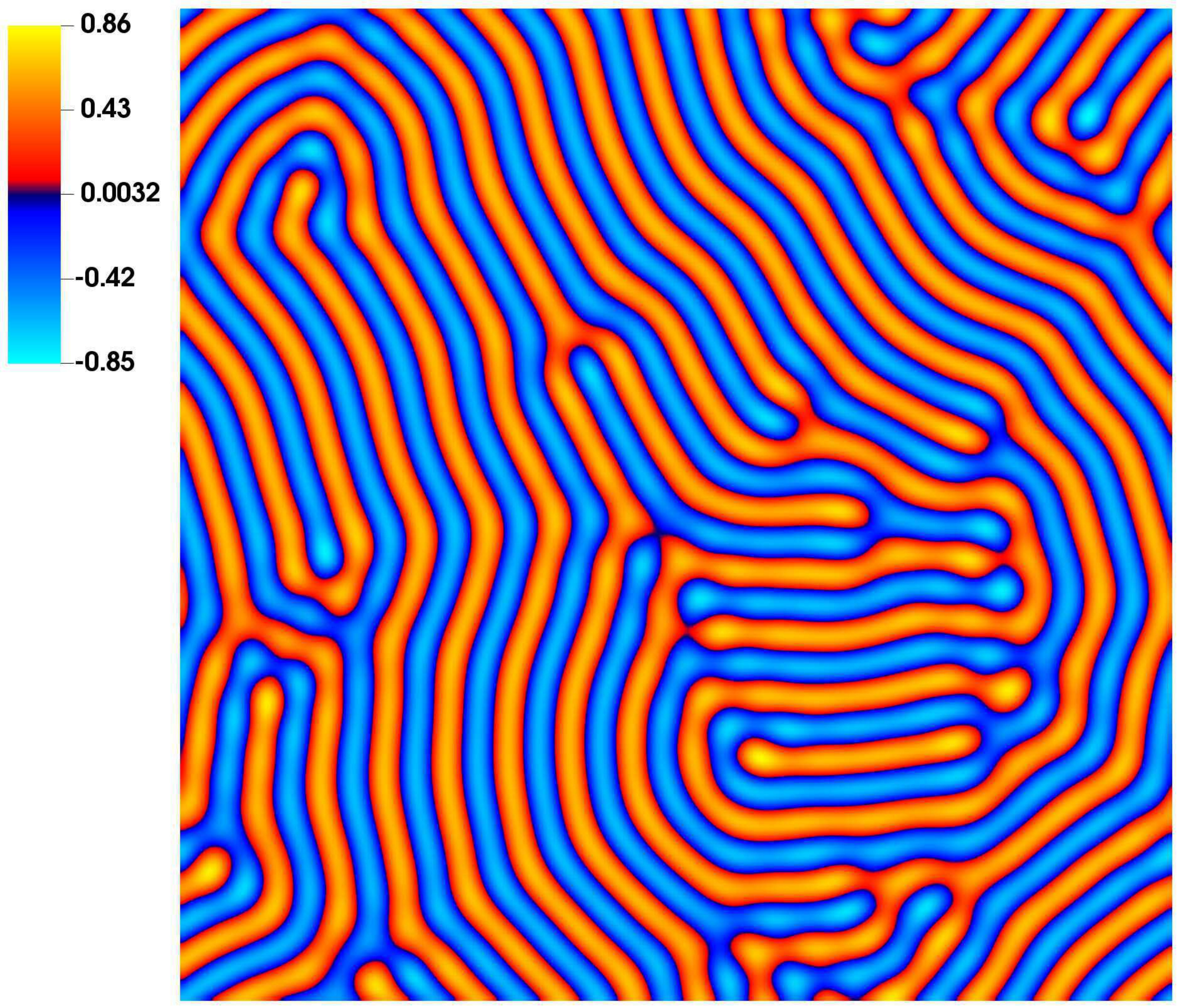}
}

\subfigure[$\phi$ at $t=1,5,50, 200$ with $\hat{\phi}_0=0.2$]{
\includegraphics[width=0.24\textwidth]{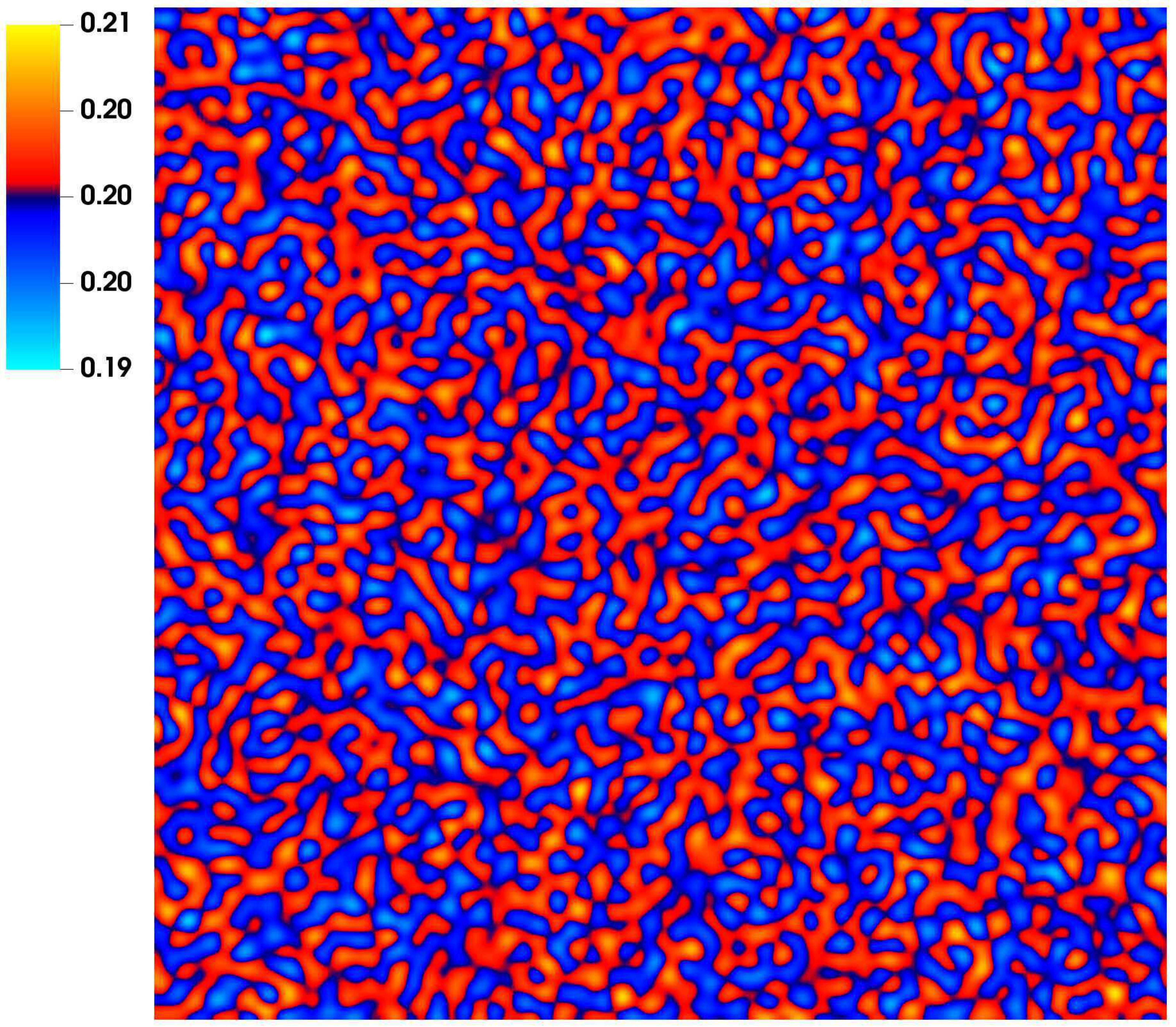}
\includegraphics[width=0.24\textwidth]{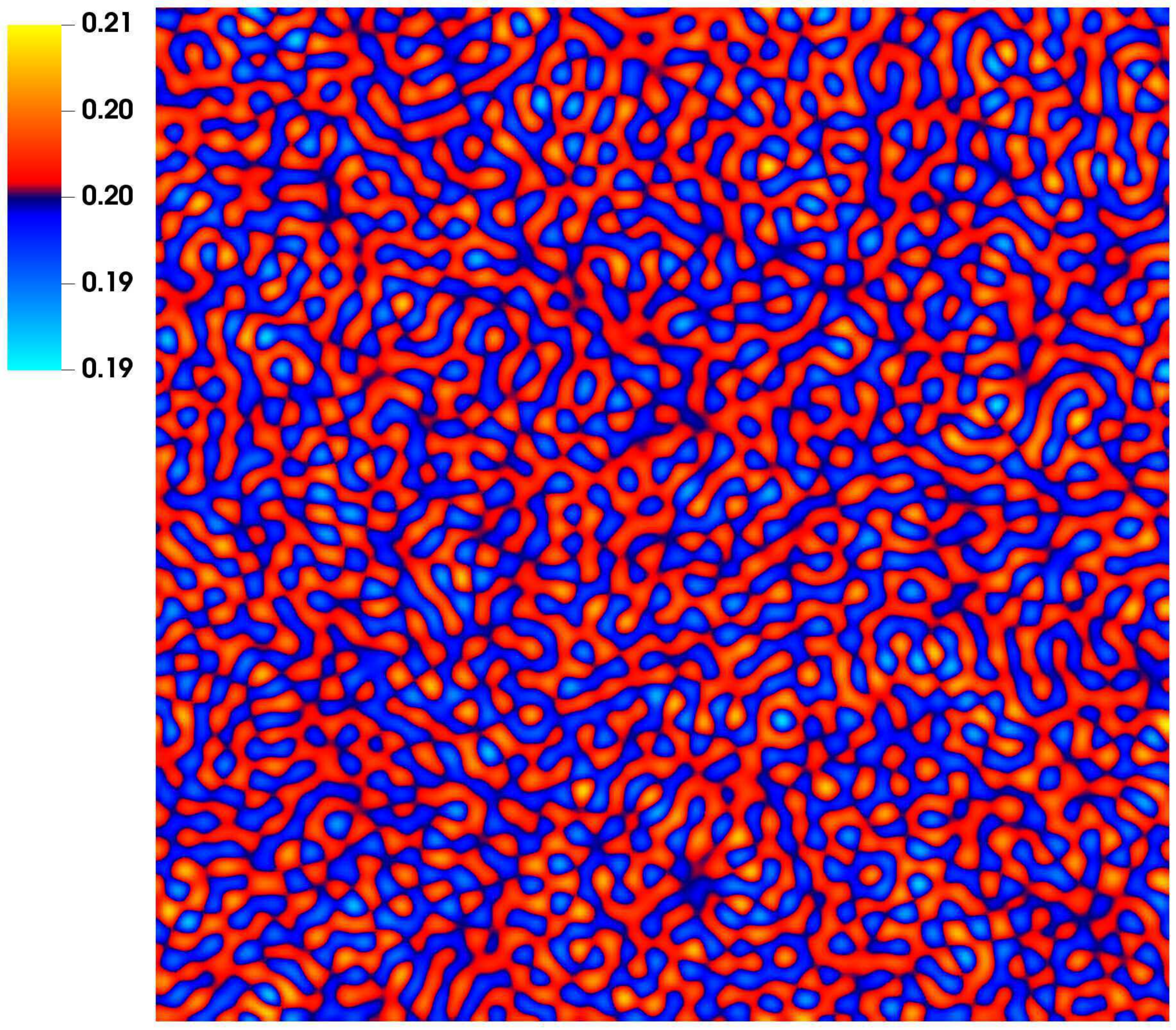}
\includegraphics[width=0.24\textwidth]{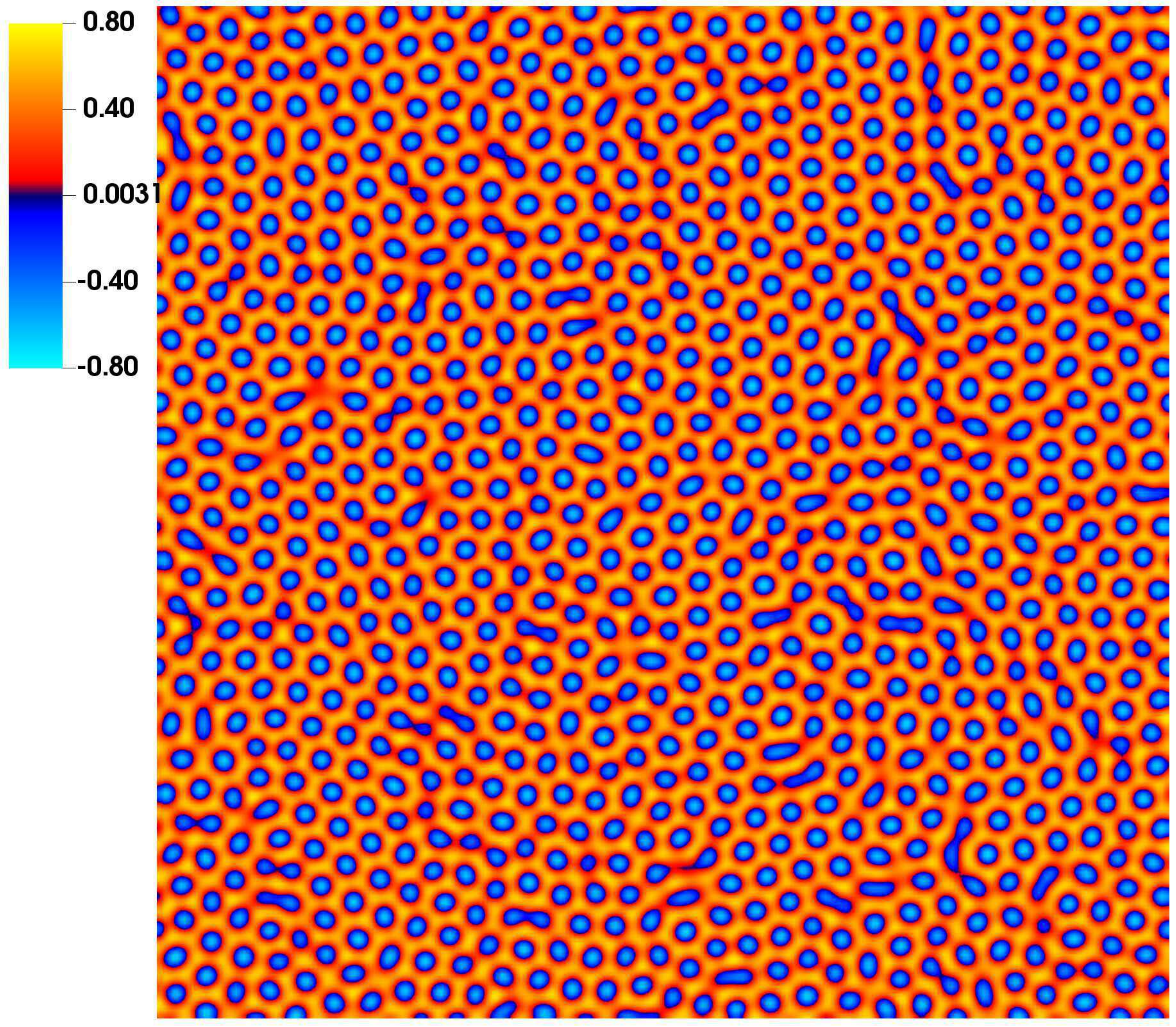}
\includegraphics[width=0.24\textwidth]{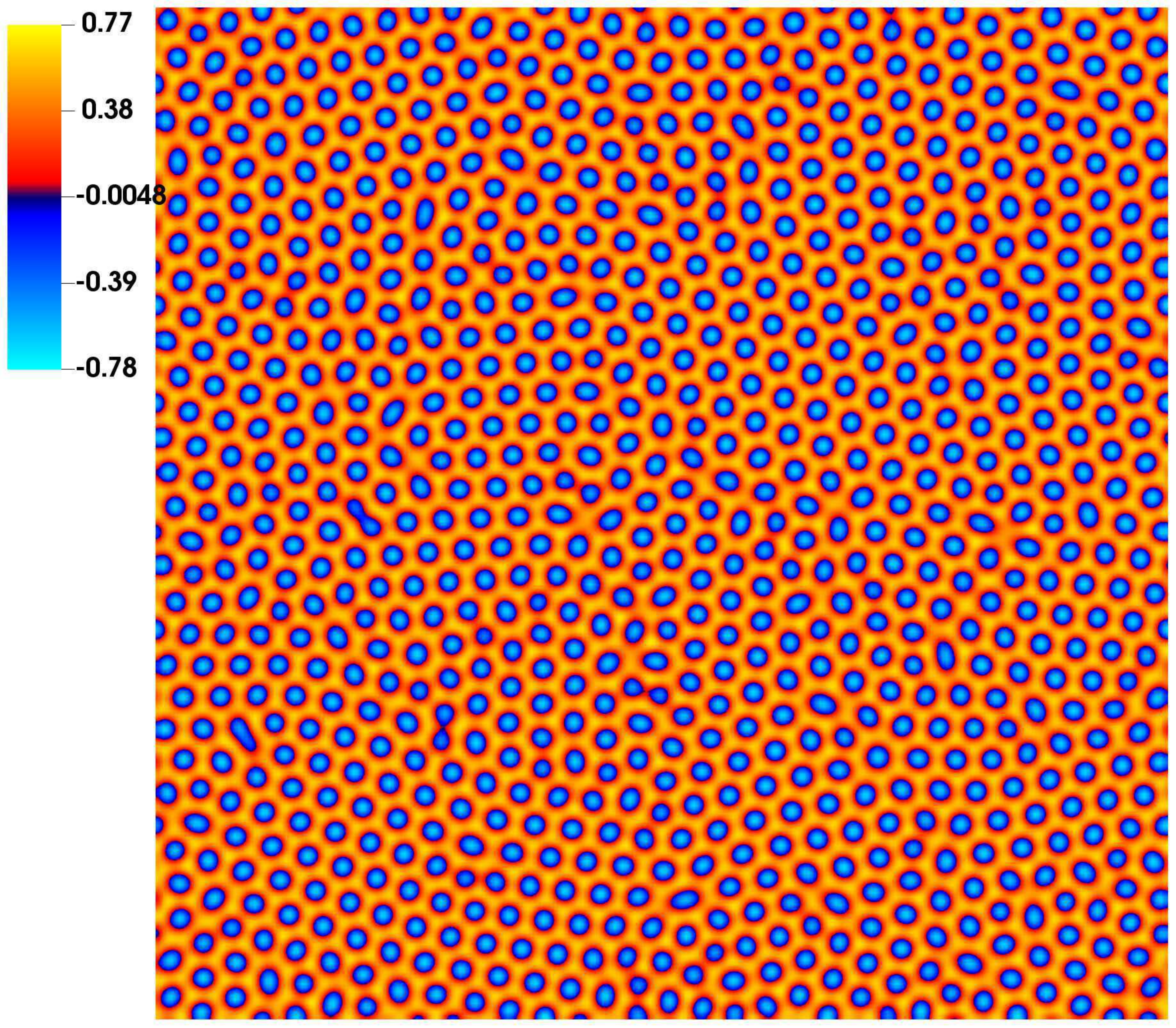}
}
\caption{Crystal growth pattern formation with different initial conditions governed by the PFC model. (a) $\hat{\phi}_0=0$; (b) $\hat{\phi}_0=0.2$.}
\label{fig:PFC-pattern}
\end{figure}

In the last example, we examine the diblock copolymer model with the proposed RSAV method. Consider the free energy 
$$\cE = \int_\Omega \frac{\varepsilon^2}{2}|\nabla \phi|^2 + \frac{1}{4}(\phi^2-1)^2d\bx + \frac{\sigma}{2} \int_\Omega \int_\Omega G(\bx-\mathbf{y}) (\phi(\bx)-\hat{\phi}_0)(\phi(\mathbf{y})-\hat{\phi}_0)d\bx d\mathbf{y}$$
with $\sigma$ a parameter for the nonlocal interaction strength. Here $G$ is the Green's function such that $\Delta G(\bx-\mathbf{y})=-\delta (\bx-\mathbf{y})$ with periodic boundary condition and $\delta$ is a Dirac delta function. And consider the mobility operator $\cG=-\lambda\Delta$. The general gradient flow model in \eqref{eq:generic-model} is specified into the phase-field diblock-copolymer model, which reads as
\begin{subequations}
\begin{align}
& \partial_t \phi = \lambda \Big[ \Delta \mu - \sigma (\phi    - \hat{\phi}_0) \Big], \\
& \mu = -\varepsilon^2 \Delta \phi + \phi^3 - \phi.
\end{align}
\end{subequations}
We consider a domain $\Omega=[0, 1]^2$, and parameters $\lambda=0.1$, $\varepsilon=0.01$, $\hat{\phi}_0=0.4$, and initial condition $\phi(x, y, t=0)=\hat{\phi}_0+0.05 rand(x,y)$ where $rand(x,y)$ generates random numbers in the range $[-1, 1]$. To solve the model, we use the numerical parameters $\gamma_0=4$, $C_0=1$, $N_x=N_y=128$. Here we test various nonlocal interaction strength $\sigma$. The numerical results at $t=500$ are summarized in Figure \ref{fig:OK-model}. We observe that the number of droplets scales with the nonlocal interaction strength $\sigma$.
\begin{figure}
\center
\subfigure[$\sigma=1$]{\includegraphics[width=0.24\textwidth]{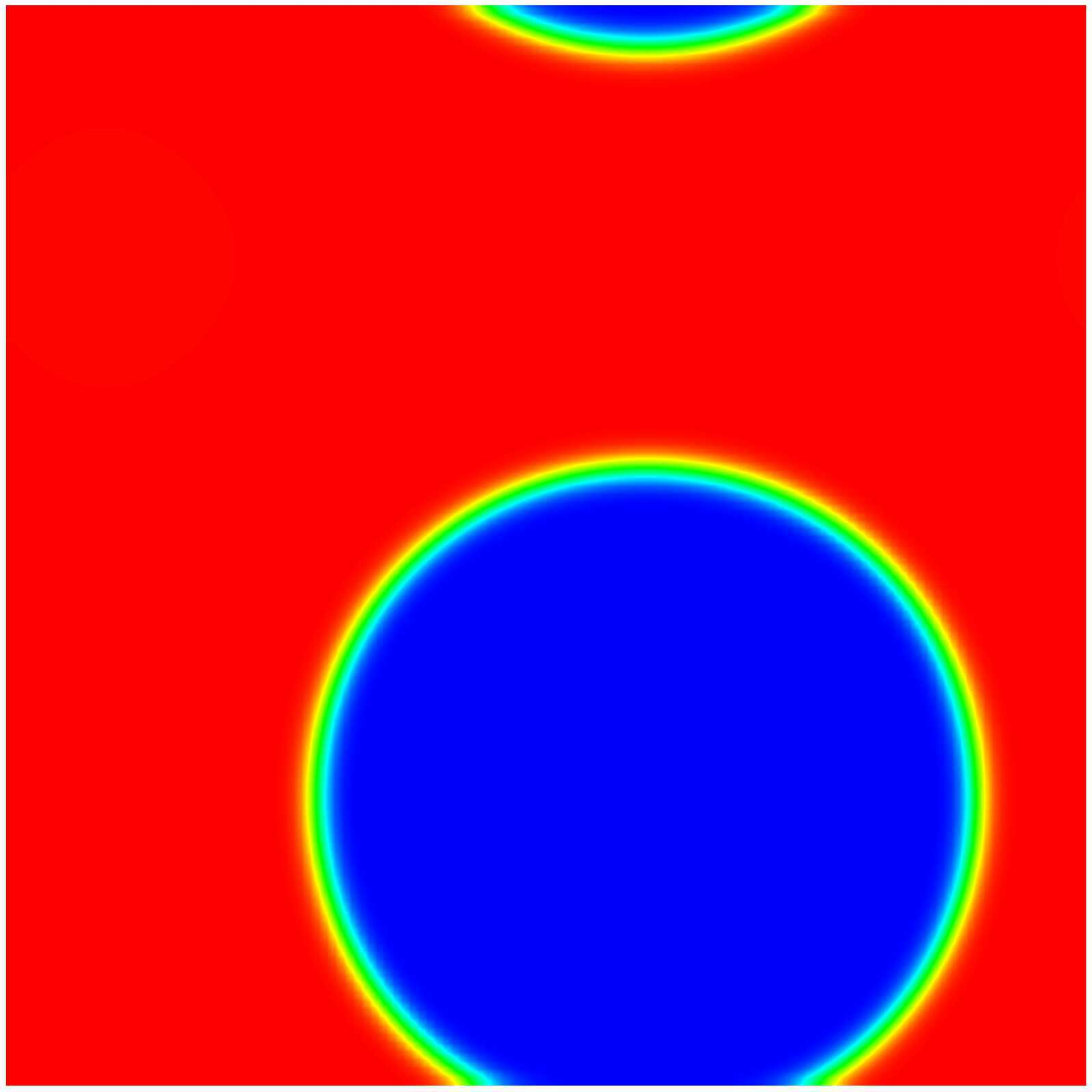}}
\subfigure[$\sigma=2$]{\includegraphics[width=0.24\textwidth]{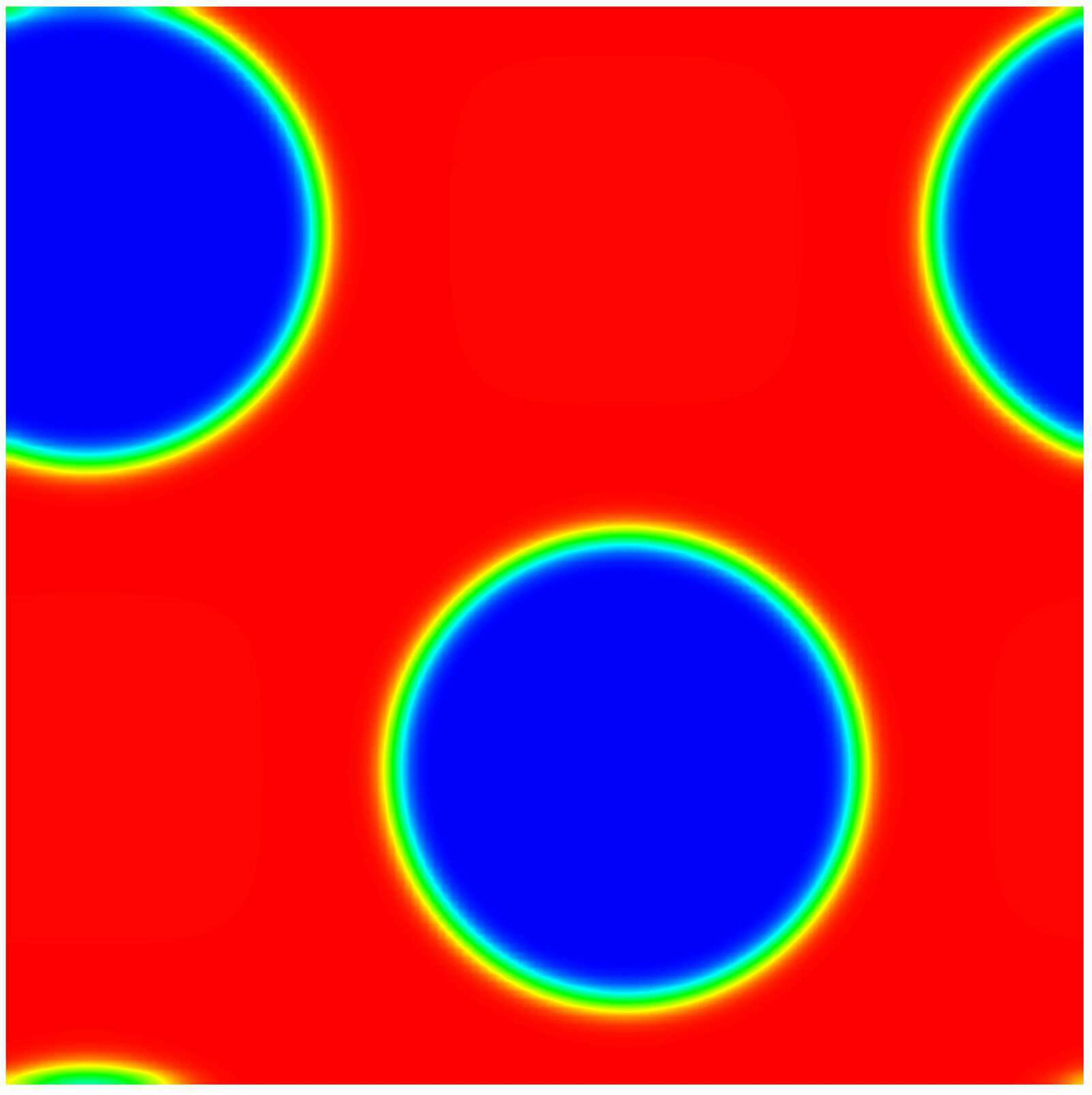}}
\subfigure[$\sigma=5$]{\includegraphics[width=0.24\textwidth]{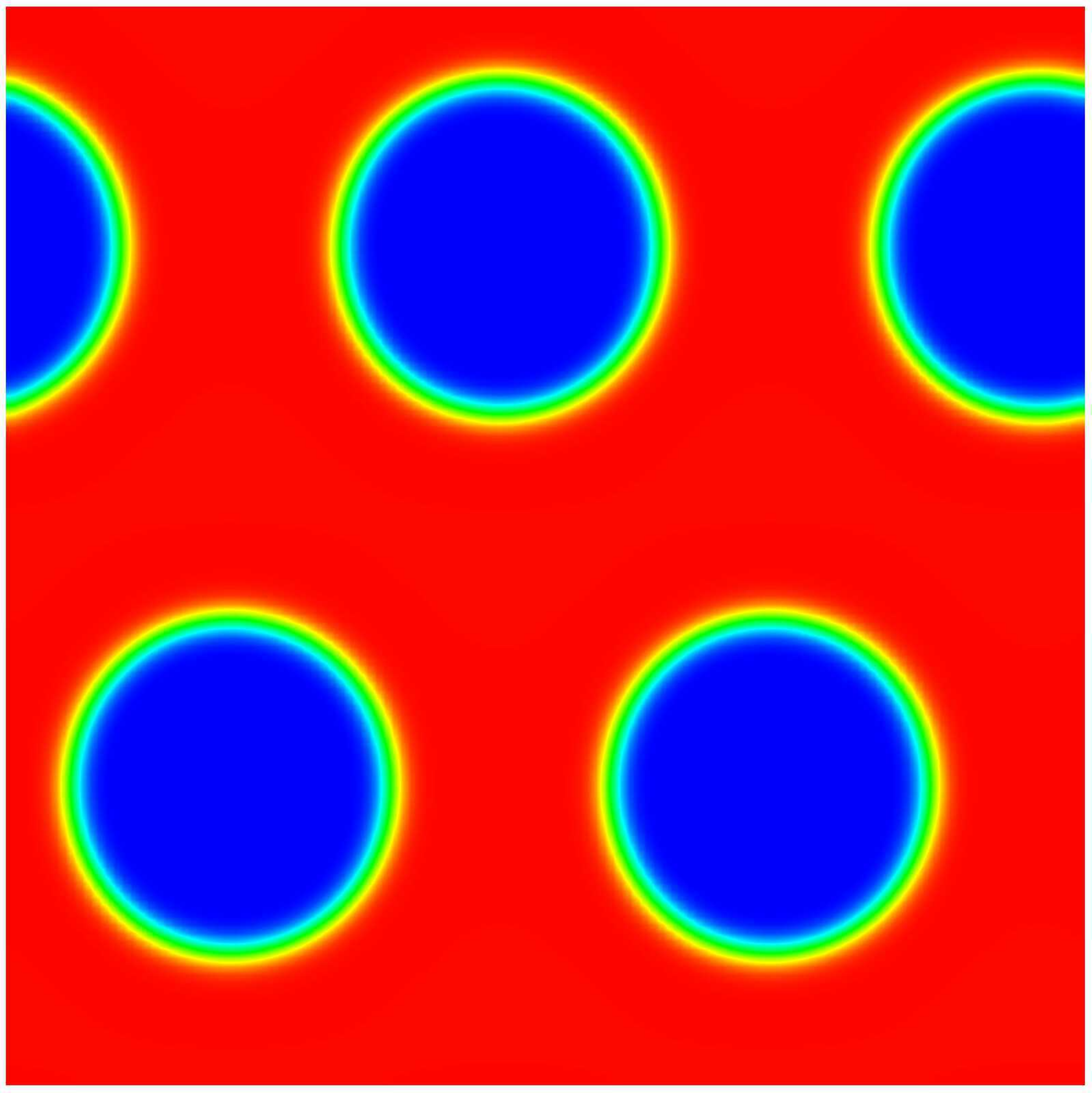}}

\subfigure[$\sigma=10$]{\includegraphics[width=0.24\textwidth]{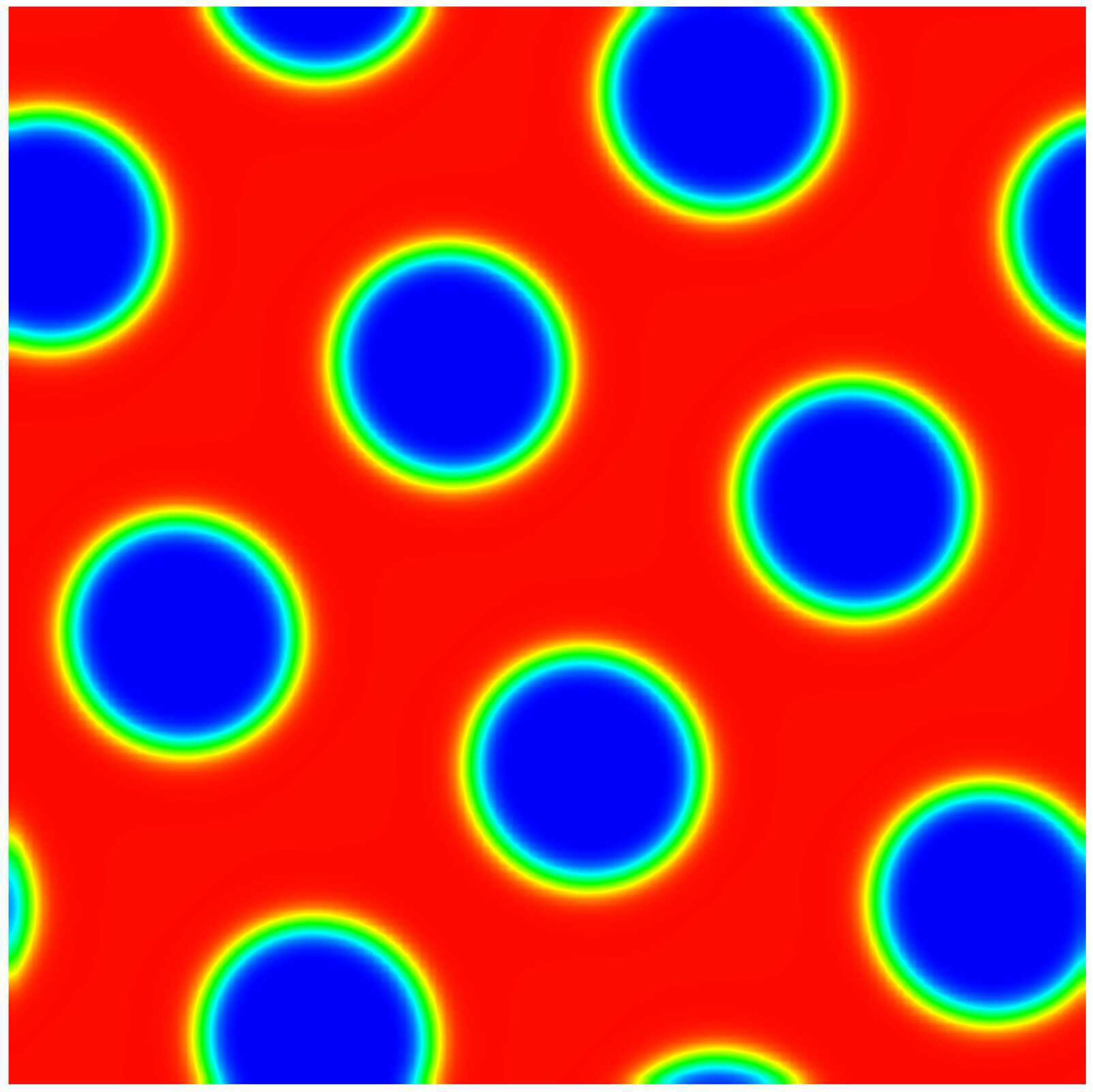}}
\subfigure[$\sigma=50$]{\includegraphics[width=0.24\textwidth]{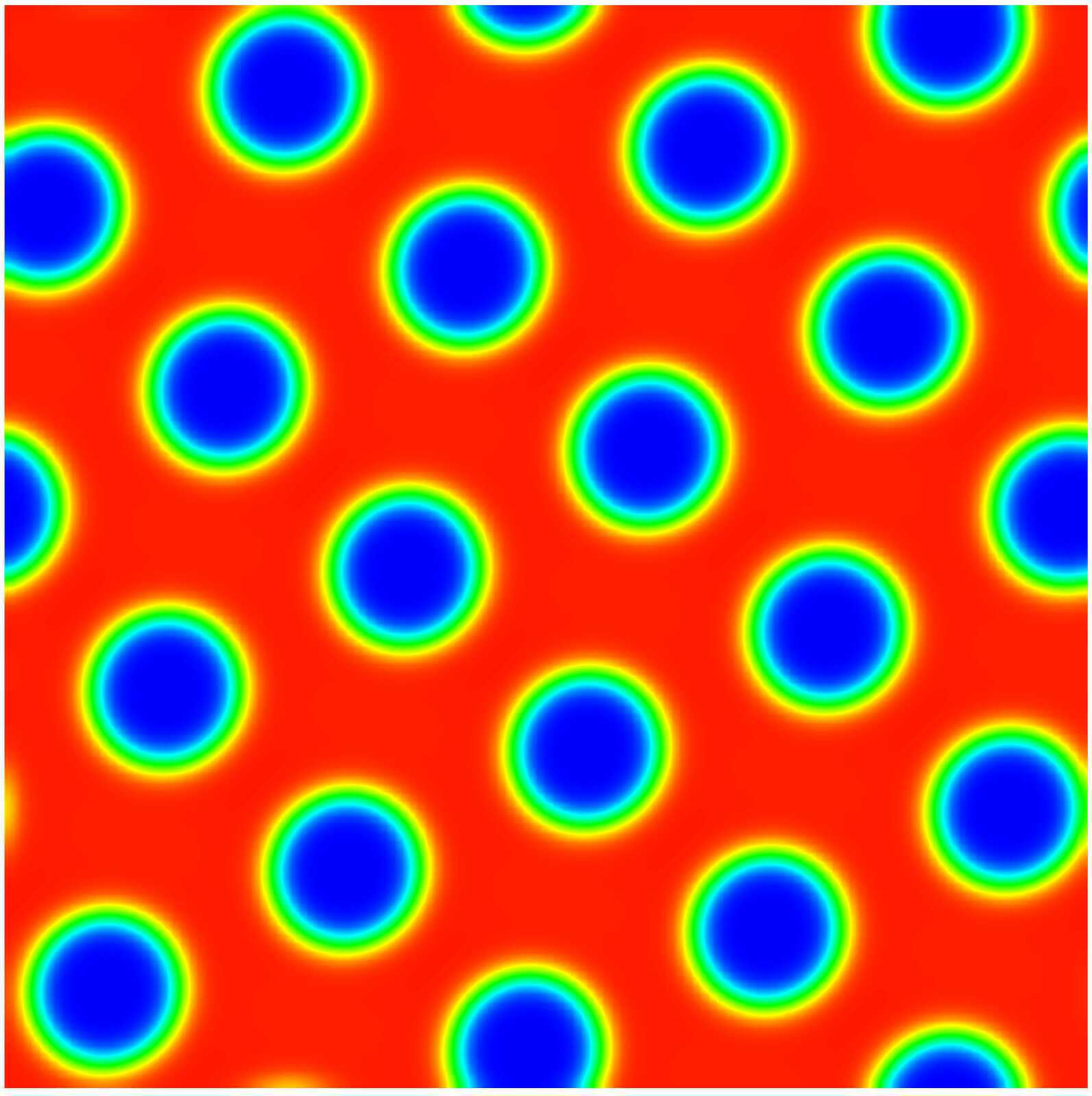}}
\subfigure[$\sigma=100$]{\includegraphics[width=0.24\textwidth]{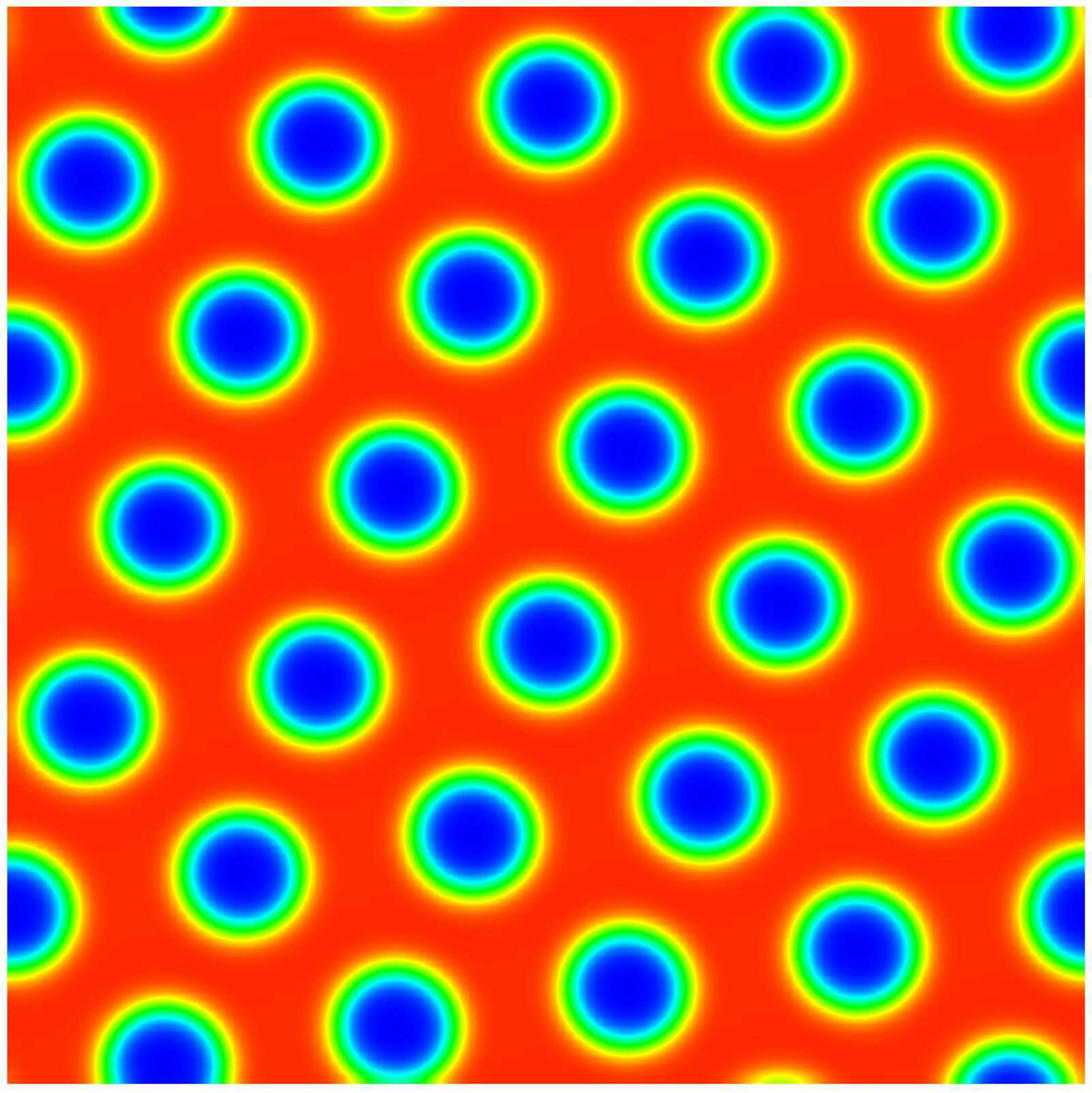}}
\caption{Coarsening dynamics driven by the dilbock copolymer model. The profiles of $\phi$ at $t=500$ are shown under various nonlocal interactions strength $\sigma$. It shows that more droplets are formed with stronger nonlocal interaction strength $\sigma$.}
\label{fig:OK-model}
\end{figure}

\section{Conclusion}
In this paper, we introduce a relaxation technique to improve the accuracy and consistency of the baseline SAV method for solving dissipative PDE models (phase-field models in particular). Our relaxed-SAV (RSAV) approach leads to linear, second-order, unconditional energy schemes. Most importantly, the RSAV schemes preserve the original energy given the relaxation parameter $\xi$ reaches 0.  Furthermore, we provide detailed proofs for the energy stability properties of the RSAV method. Then, we apply the RSAV method to solve the Allen-Cahn (AC)  equation, the Cahn-Hilliard (CH) equation, the Molecular Beam Epitaxy (MBE) model, the phase-field crystal (PFC) model, and the diblock copolymer model. Numerical experiments highlight the accuracy and efficiency of the proposed RSAV method. The numerical comparisons between the baseline SAV schemes and the RSAV schemes indicate that the RSAV method is unconditional energy stable according to the original energy laws and has better accuracy and consistency over the baseline SAV method.

\section*{Acknowledgments}
Jia Zhao would like to acknowledge the support from National Science Foundation with grant NSF-DMS-1816783.

\end{document}